\documentclass{article}

\usepackage[final,nonatbib]{neurips_2021}

\usepackage{graphicx} 
\usepackage{float}    
\usepackage{verbatim} 
\usepackage{amsthm}   
\usepackage{amsmath}  
\usepackage{amssymb}  
\usepackage[sort&compress,numbers]{natbib}
\usepackage{subfig}   
\usepackage[colorlinks=true,citecolor=blue,linkcolor=blue,urlcolor=blue]{hyperref}
\usepackage{enumerate}
\usepackage{paralist}
\usepackage{xspace}
\usepackage{caption}
\usepackage{bbm}
\usepackage{url}
\usepackage{mathrsfs}
\usepackage[linesnumbered,ruled,vlined]{algorithm2e}
\usepackage{tikz}

\usepackage{xcolor}
\usepackage{microtype}
\usepackage{graphicx,wrapfig,lipsum, booktabs}

\makeatother

\newcommand{\R}{\mathbb{R}}

\newcommand{\E}{\mathbb{E}}

\newcommand{\vect}{\mathrm{vec}}
\newcommand{\rank}{\mathrm{rank}}

\newcommand{\norm}[1]{\left\lVert#1\right\rVert}
\newcommand{\tr}[1]{\mathrm{Tr}\left(#1\right)}

\newcommand{\inner}[2]{\left\langle #1, #2 \right\rangle}

\global\long\def\P{\mathbf{P}}%

\DeclareMathOperator*{\argmax}{arg\,max}

\global\long\def\A{\mathbf{A}}%
\global\long\def\J{\mathbf{J}}%
\global\long\def\e{\mathbf{e}}%
\global\long\def\E{\mathbb{E}}%
\global\long\def\vect{\mathrm{vec}}%
\global\long\def\norm#1{\left\lVert #1\right\rVert }%
\global\long\def\tr#1{\mathrm{Tr}\left(#1\right)}%
\global\long\def\inner#1#2{\left\langle #1, #2 \right\rangle }%
\global\long\def\P{\mathbf{P}}%
\global\long\def\argmax{\arg\,\max}%
\global\long\def\rank{\mathrm{rank}}%
\global\long\def\AA{\mathcal{A}}%
\global\long\def\R{\mathbb{R}}%
\global\long\def\ub{\mathrm{ub}}%
\global\long\def\lb{\mathrm{lb}}%
\global\long\def\eqdef{\overset{\text{def}}{=}}%
\global\long\def\PP{\mathcal{P}}%

\theoremstyle{plain}
\newtheorem{thm}{\protect\theoremname}
\theoremstyle{definition}
\newtheorem{defn}[thm]{\protect\definitionname}
\theoremstyle{plain}
\newtheorem{prop}[thm]{\protect\propositionname}
\theoremstyle{remark}

\theoremstyle{plain}
\newtheorem{cor}[thm]{\protect\corollaryname}
\theoremstyle{plain}
\newtheorem{lem}[thm]{\protect\lemmaname}
\providecommand{\definitionname}{Definition}
\providecommand{\lemmaname}{Lemma}
\providecommand{\propositionname}{Proposition}
\providecommand{\remarkname}{Remark}
\providecommand{\theoremname}{Theorem}
\providecommand{\corollaryname}{Corollary}

\title{Preconditioned Gradient Descent for Over-Parameterized Nonconvex Matrix Factorization}

\author{Gavin Zhang\\
University of Illinois at Urbana\textendash Champaign\\
\texttt{jialun2@illinois.edu}~\\
\And Salar Fattahi\\
University of Michigan\\
\texttt{fattahi@umich.edu}\\
\And Richard Y. Zhang\\
University of Illinois at Urbana\textendash Champaign\\
\texttt{ryz@illinois.edu}}

\begin{document}

\maketitle

\begin{abstract}
In practical instances of nonconvex matrix factorization, the rank
of the true solution $r^{\star}$ is often unknown, so the rank $r$
of the model can be overspecified as $r>r^{\star}$. This over-parameterized
regime of matrix factorization significantly slows down the convergence
of local search algorithms, from a linear rate with $r=r^{\star}$
to a sublinear rate when $r>r^{\star}$. We propose an inexpensive
preconditioner for the matrix sensing variant of nonconvex matrix
factorization that restores the convergence rate of gradient descent
back to linear, even in the over-parameterized case, while also making
it agnostic to possible ill-conditioning in the ground truth. Classical
gradient descent in a neighborhood of the solution slows down due
to the need for the model matrix factor to become singular. Our key
result is that this singularity can be corrected by $\ell_{2}$ regularization
with a specific range of values for the damping parameter. In fact,
a good damping parameter can be inexpensively estimated from the current
iterate. The resulting algorithm, which we call preconditioned gradient
descent or PrecGD, is stable under noise, and converges linearly to
an information theoretically optimal error bound. Our numerical experiments
find that PrecGD works equally well in restoring the linear convergence
of other variants of nonconvex matrix factorization in the over-parameterized
regime.
\end{abstract}

\section{Introduction}
\global\long\def\rank{\mathrm{rank}}%
\global\long\def\AA{\mathcal{A}}%
\global\long\def\R{\mathbb{R}}%
\global\long\def\ub{\mathrm{ub}}%
\global\long\def\lb{\mathrm{lb}}%
\global\long\def\eqdef{\overset{\text{def}}{=}}%
\global\long\def\PP{\mathcal{P}}%
Numerous problems in machine learning can be reduced to the \emph{matrix
factorization} problem of recovering a low-rank positive semidefinite
matrix $M^{\star}\succeq0$, given a small number of potentially noisy
observations~\citep{yu2009fast,luo2014efficient,candes2011robust,chandrasekaran2011rank,ahmed2013blind,ling2015self,singer2011angular}.
In every case, the most common approach is to formulate an
$n\times n$ candidate matrix $M=XX^{T}$ in factored form, and to
minimize a \emph{nonconvex} empirical loss $f(X)$ over its $n\times r$
low-rank factor $X$. But in most real applications of nonconvex
matrix factorization, the rank of the ground truth $r^{\star}=\rank(M^{\star})$
is unknown. It is reasonable to choose the rank $r$ of the model
$XX^{T}$ conservatively, setting it to be potentially larger than
$r^{\star}$, given that the ground truth can be exactly recovered
so long as $r\ge r^{\star}$. In practice, this will often lead to
an \emph{over-parameterized} regime, in which $r>r^{\star}$, and
we have specified more degrees of freedom in our model $XX^{T}$ than
exists in the underlying ground truth $M^{\star}$.

\citet{zhuo2021computational} recently pointed out that nonconvex
matrix factorization becomes substantially less efficient in the over-parameterized
regime. For the prototypical instance of matrix factorization known
as \emph{matrix sensing} (see Section~\ref{sec:background} below
for details) it is well-known that, if $r=r^{\star}$, then (classic)
gradient descent or GD
\[
X_{k+1}=X_{k}-\alpha\nabla f(X_{k})\tag{GD}
\]
converges at a linear rate, to an $\epsilon$-accurate iterate in
$O(\kappa\log(1/\epsilon))$ iterations, where $\kappa = \lambda_1(M^\star)/\lambda_{r^*}(M^\star)$ is the condition number of the ground truth ~\citep{NIPS2015_32bb90e8,tu2016low}.
But in the case that $r>r^{\star}$, \citet{zhuo2021computational}
proved that gradient descent slows down to a \emph{sublinear} convergence
rate, now requiring $\mathrm{poly}(1/\epsilon)$ iterations to yield
a comparable $\epsilon$-accurate solution. This is a dramatic, exponential
slow-down: whereas 10 digits of accuracy can be expected in a just
few hundred iterations when $r=r^{\star}$, tens of thousands of iterations
might produce just 1-2 accurate digits once $r>r^{\star}$. The slow-down
occurs even if $r$ is just off by one, as in $r=r^{\star}+1$.

It is helpful to understand this pheonomenon by viewing over-parameterization as a special, extreme case of ill-conditioning, where the condition number of the ground truth, $\kappa$, is taken to infinity. In this limit, the classic linear rate $O(\kappa\log(1/\epsilon))$ breaks down, and in reality, the convergence rate deterioriates to sublinear.

In this paper, we present an inexpensive
\emph{preconditioner} for gradient descent. The resulting algorithm, which we call PrecGD, corrects for both ill-conditioning and over-parameterization at the same time, without viewing them as distinct concepts. We prove, for the matrix
sensing variant of nonconvex matrix factorization, that the preconditioner
restores the convergence rate of gradient descent back to linear,
even in the over-parameterized case, while also making it agnostic
to possible ill-conditioning in the ground truth. Moreover, PrecGD maintains
a similar per-iteration cost to regular gradient descent, is stable
under noise, and converges linearly to an information theoretically
optimal error bound.

We also perform numerical experiments on other variants of nonconvex
matrix factorization, with different choices of the empirical loss
function $f$. In particular, we consider different $\ell_{p}$ norms
with $1\le p<2$, in order to gauge the effectiveness of PrecGD for
increasingly nonsmooth loss functions. Our numerical experiments find
that, if regular gradient descent is capable of converging quickly
when the rank is known $r=r^{\star}$, then PrecGD restores this rapid
converging behavior when $r>r^{\star}$. PrecGD is able to overcome
ill-conditioning in the ground truth, and converge reliably without
exhibiting sporadic behavior. 

\section{Proposed Algorithm: Preconditioned Gradient Descent}

Our preconditioner is inspired by a recent work of \citet{tong2020accelerating}
on matrix sensing with an ill-conditioned ground truth $M^{\star}$.
Over-parameterization can be viewed as the limit of this regime, in
which $\lambda_{r}(M^{\star})$, the $r$-th largest eigenvalue of
$M^{\star}$, is allowed to approach all the way to zero. For finite
but potentially very small values of $\lambda_{r}(M^{\star})>0$,
 \citet{tong2020accelerating} suggests the following iterations,
which they named \emph{scaled} gradient descent or ScaledGD:
\[
X_{k+1}=X_{k}-\alpha\nabla f(X_{k})(X_{k}^{T}X_{k})^{-1}.\tag{ScaledGD}
\]
They prove that the scaling allows the iteration to make a large, constant
amount of progress at every iteration, independent of the value of
$\lambda_{r}(M^{\star})>0$. However, applying this same scheme to
the over-parameterized case with $\lambda_{r}(M^{\star})=0$ results
in an inconsistent, sporadic behavior. 

The issues encountered by both regular GD and ScaledGD with over-parameterization
$r>r^{\star}$ can be explained by the fact that our iterate $X_{k}$
must necessarily become \emph{singular} as our rank-$r$ model $X_{k}X_{k}^{T}$
converges towards the rank-$r^{\star}$ ground truth $M^{\star}$.
For GD, this singularity causes the per-iteration progress itself
to decay, so that more and more iterations are required for each fixed
amount of progress. ScaledGD corrects for this decay in per-iteration
progress by suitably rescaling the search direction. However, the
rescaling itself requires inverting a near-singular matrix, which
causes algorithm to take on sporadic values.

\begin{figure}[t]
\includegraphics[width=0.48\textwidth]{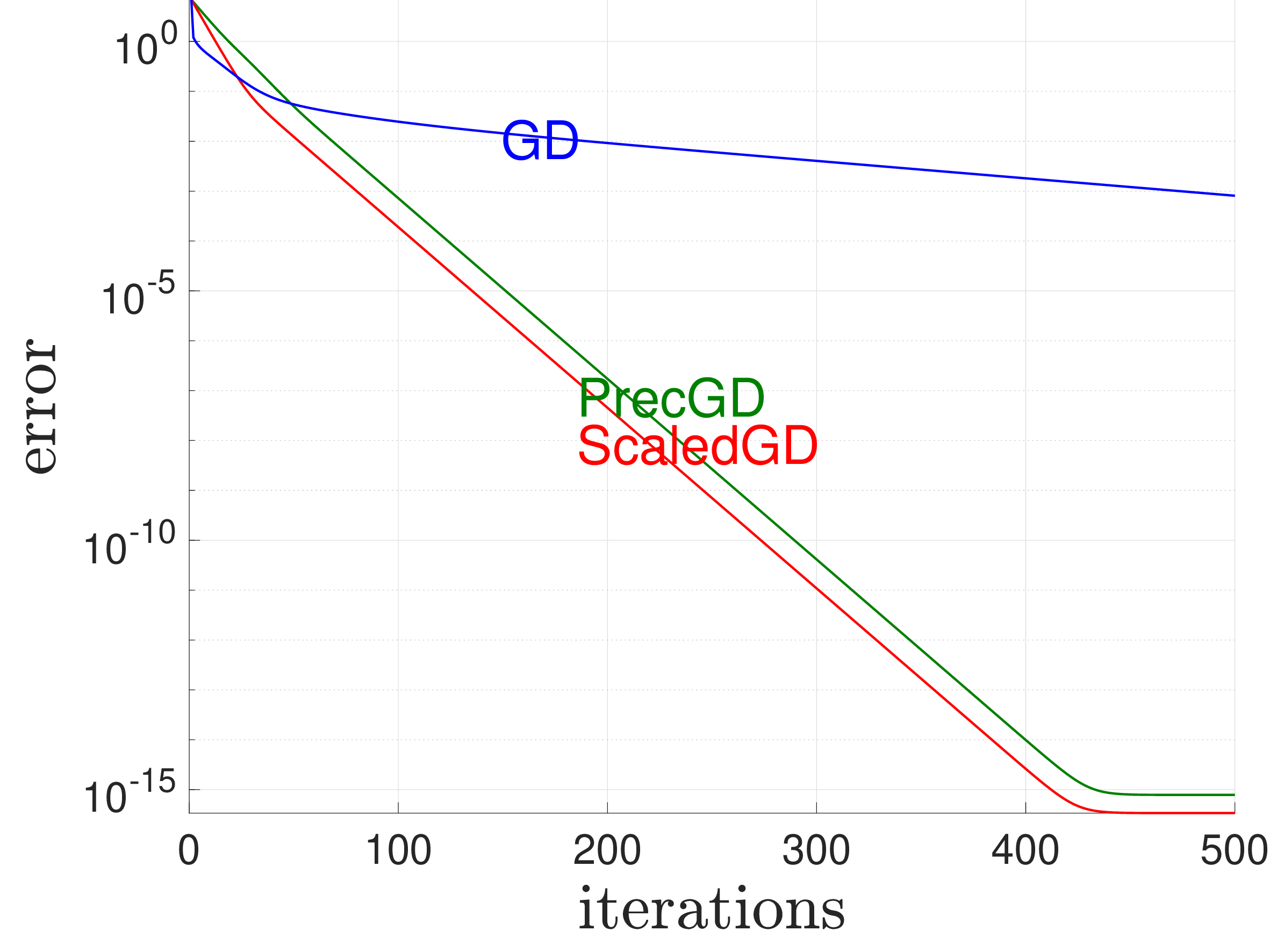} 
\hfill{}\includegraphics[width=0.48\textwidth]{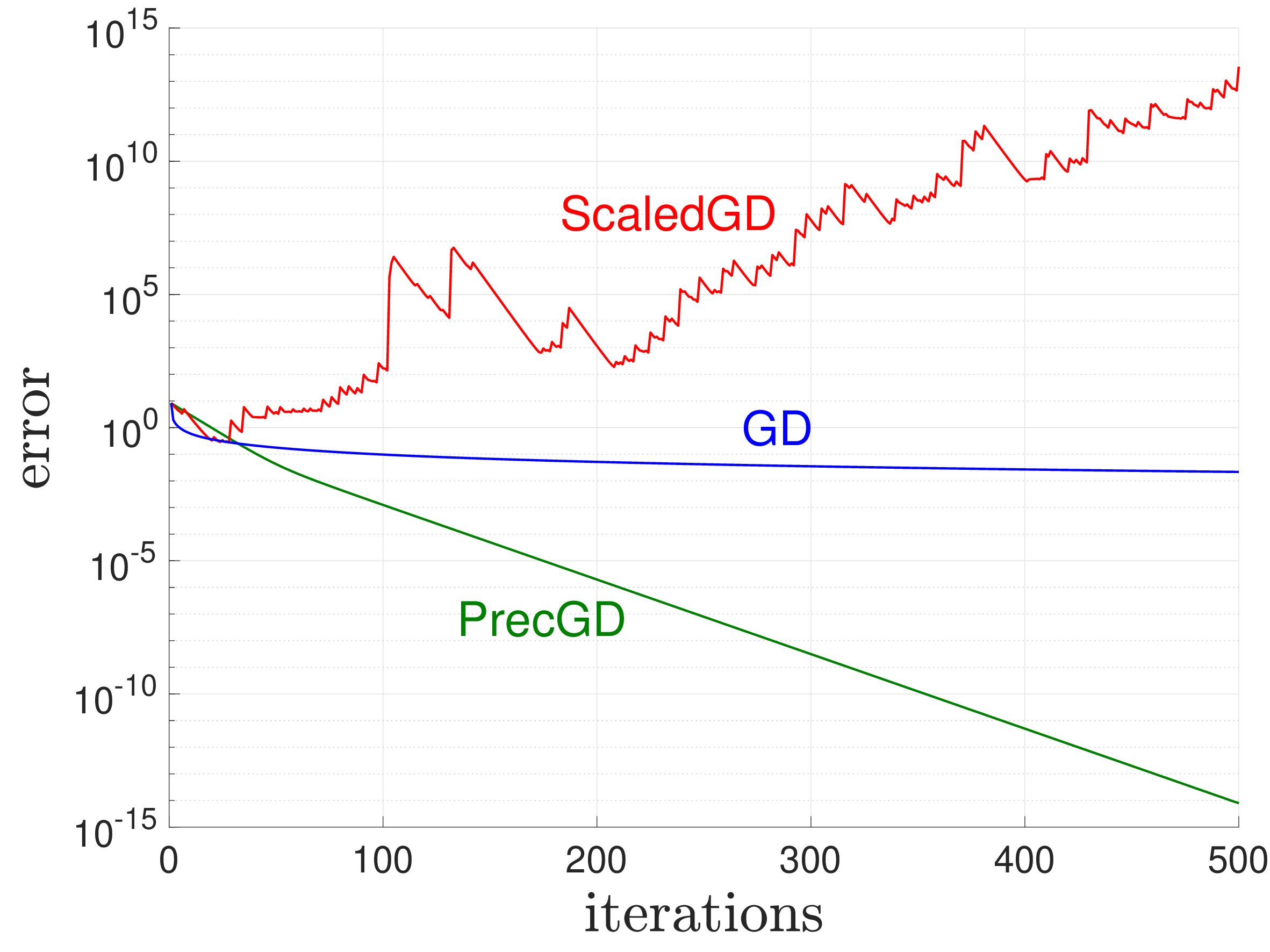} 
 \caption{\textbf{PrecGD converges linearly in the overparameterized regime.}
Convergence of regular gradient descent (GD), ScaledGD and PrecGD
for noiseless matrix sensing (with data taken from~\citep{NEURIPS2018_f8da71e5,zhang2019sharp})
from the same initial points and using the same learning rate $\alpha=2\times10^{-2}$.
\textbf{(Left $r=r^{*}$)} Set $n=4$ and $r^{*}=r=2$. All three
methods convergence at a linear rate, though GD converges at a slower
rate due to ill-conditioning in the ground truth. \textbf{(Right $r>r^{*}$)}
With $n=4$, $r=4$ and $r^{*}=2$, over-parameterization causes gradient
descent to slow down to a sublinear rate. ScaledGD also behaves sporadically.
Only PrecGD converges linearly to the ground truth.}
\end{figure}

A classical remedy to issues posed by singular matrices is $\ell_{2}$
regularization, in which the singular matrix is made ``less singular''
by adding a small identity perturbation. Applying this idea to ScaledGD
yields the following iterations
\[
X_{k+1}=X_{k}-\alpha\nabla f(X_{k})(X_{k}^{T}X_{k}+\eta_{k}I_{r})^{-1},\tag{PrecGD}
\]
where $\eta_{k}\ge0$ is the \emph{damping} parameter specific to
the $k$-th iteration. There are several interpretations to this scheme,
but the most helpful is to view $\eta$ as a parameter that allows
us to interpolate between ScaledGD (with $\eta=0$) and regular GD
(in the limit $\eta\to\infty)$. In this paper, we prove for matrix
sensing that, if the $k$-th damping parameter $\eta_{k}$ is chosen
within a constant factor of the error
\begin{equation}
C_{\lb}\|X_{k}X_{k}^{T}-M^{\star}\|_{F}\le\eta_{k}\le C_{\ub}\|X_{k}X_{k}^{T}-M^{\star}\|_{F},\quad\text{ where }C_{\lb},C_{\ub}>0\text{ are abs. const.}\label{eq:condition}
\end{equation}
then the resulting iterations are guaranteed to converge linearly,
at a rate that is independent of both over-parameterization and ill-conditioning
in the ground truth $M^{\star}$. With noisy measurements, setting
$\eta_{k}$ to satisfy (\ref{eq:condition}) will allow the iterations
to converge to an error bound that is well-known to be minimax optimal
up to logarithmic factors~\citep{candes2011tight}.

We refer to the resulting iterations (with a properly chosen $\eta_{k}$)
as \emph{preconditioned} gradient descent, or PrecGD for short. For
matrix sensing with noiseless measurements, an optimal $\eta_{k}$
that satisfies the condition (\ref{eq:condition}) is obtained for
free by setting $\eta_{k}=\sqrt{f(X_{k})}$. In the case of noisy
measurements, we show that a good choice of $\eta_{k}$ is available
based on an approximation of the noise variance.

\section{\label{sec:background}Background and Related Work}
\textbf{Notations.} We use $\|\cdot\|_F$ to denote the Frobenius norm of a matrix and $\inner{\cdot}{\cdot}$ is the corresponding inner product. We use $\gtrsim$ to denote an inequality that hides a constant factor. The big-O notation $\tilde{O}$ hides logarithimic factors. The gradient of the objective is denoted by $\nabla f(X) \in \R^{n\times r}.$ The eigenvalues are assumed to be in decreasing order: $\lambda_1\geq \lambda_2 \geq \dots \geq \lambda_r.$ 

The symmetric, linear variant of matrix factorization known as matrix
sensing aims to recover a positive semidefinite, rank-$r^{\star}$
ground truth matrix $M^{\star}$, from a small number $m$ of possibly
noisy measurements 
\[
y=\mathcal{A}(M^{\star})+\epsilon,\qquad\quad\text{where }\mathcal{A}(M^{\star})=[\langle A_{1},M^{\star}\rangle,\langle A_{2},M^{\star}\rangle,\dots,\langle A_{m},M^{\star}\rangle]^{T},
\]
in which $\mathcal{A}$ is a linear measurement operator, and the
length-$m$ vector $\epsilon$ models the unknown measurement noise.
A distinguishing feature of matrix sensing is that $\mathcal{A}$
is assumed to satisfy the \textit{restricted isometry property }\citep{recht2010guaranteed,candes2011tight}.
Throughout this paper, we will always assume that $\AA$ satisfies
RIP with parameters $(2r,\delta)$.
\begin{defn}[RIP]
\label{def:rip}
The linear operator $\mathcal{A}$ satisfies RIP with parameters
$(2r,\delta)$ if there exists constants $0\le\delta<1$ and $m>0$
such that, for every rank-$2r$ matrix $M$, we have 
\[
(1-\delta)\|M\|_{F}^{2}\leq\frac{1}{m}\|\mathcal{A}(M)\|^{2}\leq(1+\delta)\|M\|_{F}^{2}.
\]
\end{defn}

A common approach for matrix sensing is to use a simple algorithm
like gradient descent to minimize the \textit{nonconvex} loss function:
\begin{equation}
f(X)=\frac{1}{m}\norm{y-\mathcal{A}(XX^{T})}=\frac{1}{m}\norm{\mathcal{A}(M^{\star}-XX^{T})+\epsilon}^{2}.\label{eq_l2}
\end{equation}
Recent work has provided a theoretical explanation for the empirical
success of this nonconvex approach. Two lines of work have emerged.

\textbf{Local Guarantees.} One line of work studies gradient descent
initialized inside a neighborhood of the ground truth where $X_{0}X_{0}^{T}\approx M^{\star}$
already holds~\citep{tu2016low,zheng2015convergent,bhojanapalli2016dropping,candes2015phase,ma2021implicit}.
Such an initial point can be found using spectral initialization,
see also~\citep{keshavan2010matrix,candes2015phase,chen2015fast,sun2016guaranteed,netrapalli2014non}.
With exact rank $r=r^{\star}$, previous authors showed that gradient
descent converges at a linear rate~\citep{NIPS2015_32bb90e8,tu2016low}.
In the over-parameterized regime, however, local restricted convexity
no longer holds, so the linear convergence rate is lost. \citet{zhuo2021computational}
showed that while spectral initialization continues to work under
over-parameterization, gradient descent now slows down to a sublinear
rate, but it still converges to a statistical error bound of $\tilde{\mathcal{O}}(\sigma^{2}nr^{\star}/m)$, where $\sigma$ denotes the noise variance.
This is known to be minimax optimal up to logarithmic factors~\citep{candes2011tight}.
In this paper, we prove that PrecGD with a damping parameter $\eta_{k}$
satisfying (\ref{eq:condition}) also converges to an $\tilde{\mathcal{O}}(\sigma^{2}nr^{\star}/m)$
statistical error bound. 

\textbf{Global Guarantees.} A separate line of work \citep{bhojanapalli2016global,li2019non,sun2018geometric,ge2016matrix,ge2017no,chen2017memory,sun2016complete,zhang2019sharp,zhang2021sharp}
established global properties of the landscapes of the nonconvex objective
$f$ in (\ref{eq_l2}) and its variants and showed that local search
methods can converge globally. With exact rank $r=r^{\star}$, \citet{bhojanapalli2016global}
proved that $f$ has no spurious local minima, and that all saddles
points have a strictly negative descent direction (strict saddle property~\citep{ge2015escaping},
see also~\citep{jin2017escape,ge2017no}). In the over-parameterized
regime, however,
we are no longer guaranteed to recover the ground truth in polynomial
time. 

\textbf{Other related work.} Here we mention some other techniques
can be use to solve matrix sensing in the over-parameterized regime.
Classically, matrix factorization was solved via its convex SDP relaxation~\citep{recht2010guaranteed,candes2011tight,meka2009guaranteed,candes2009exact,candes2010power}.
The resulting $\mathcal{O}(n^{3})$ to $\mathcal{O}(n^{6})$ time
complexity~\citep{alizadeh1995interior} limits this technique to
smaller problems, but these guarantees hold without prior knowledge
on the true rank $r^{\star}$. First-order methods, such as ADMM~\citep{wen2010alternating,o2016conic,zheng2020chordal}
and soft-thresholding~\citep{cai2010singular}, can be used to solve
these convex problems with a per-iteration complexity comparable to
nonconvex gradient descent, but they likewise suffer from a sublinear
convergence rate. Local recovery via spectral initialization was originally
proposed for alternating minimization and other projection techniques~\citep{jain2013low,hardt2014fast,meka2009guaranteed,netrapalli2014non,chen2015fast,yi2016fast,soltanolkotabi2019structured}.
These also continue to work, though a drawback here is a higher per-iteration
cost when compared to simple gradient methods. Finally, we mention
a recent result of \citet{li2018algorithmic}, which showed in the
over-parameterized regime that gradient descent with early termination
enjoys an algorithmic regularization effect. 

\section{\label{sec:Fail}Sublinear Convergence of Gradient Descent}

In order to understand how to improve gradient descent in the over-parameterized
regime, we must first understand why existing methods fail. For an
algorithm that moves in a search direction $D$ with step-size $\alpha$,
it is a standard technique to measure the corresponding decrement
in $f$ with a Taylor-like expansion
\begin{equation}
f(X-\alpha D)\leq f(X)-\alpha\underbrace{\langle\nabla f(X),D\rangle}_{\text{linear progress}}+\alpha^{2}\underbrace{(L/2)\|D\|_{F}^{2}}_{\text{inverse step-size}}\label{eq_improve}
\end{equation}
in which $L$ is the usual gradient Lipschitz constant (see e.g.~\citet[Chapter~3]{nocedal2006numerical}).
A good search direction $D$ is one that maximizes the linear progress
$\inner{\nabla f(X)}D$ while also keeping the inverse step-size $(L/2)\|D\|_{F}^{2}$
sufficiently small in order to allow a reasonably large step to be
taken. As we will show in this section, the main issue with gradient
descent in the over-parameterized regime is the first term, namely,
that the linear progress goes down to zero as the algorithm makes
progress towards the solution. 

Classical gradient descent uses the search direction $D=\nabla f(X)$.
Here, a common technique is to bound the linear progress at each iteration
by a condition known as \emph{gradient dominance} (or the Polyak-\L{}ojasiewicz
or PL inequality), which is written as
\begin{align}
\inner{\nabla f(X)}D=\|\nabla f(X)\|_{F}^{2} & \geq\mu(f(X)-f^{\star})\quad\text{ where }\mu>0\text{ and }f^{\star}=\min_X f(X).\label{eq:pl1}
\end{align}
Substituting the inequality (\ref{eq:pl1}) into the Taylor-like expansion
(\ref{eq_improve}) leads to 
\begin{align}
f(X-\alpha D) & \leq f(X)-\alpha\|\nabla f(X)\|_{F}^{2}+\alpha^{2}(L/2)\|\nabla f(X)\|_{F}^{2}\nonumber \\
f(X-\alpha D)-f^{\star} & \le[1-\mu\alpha(1-\alpha L/2)]\cdot(f(X)-f^{\star}).\label{eq:gdconv}
\end{align}
Here, we can always pick a small enough step-size $\alpha$ to guarantee
linear convergence: 
\begin{equation}
Q=1-\mu\alpha+\mu\alpha^{2}L/2<1\implies f(X_{k})-f^{\star}\le Q^{k}[f(X_{0})-f^{\star}].\label{eq:linconv}
\end{equation}
In particular, picking the optimal step-size $\alpha=1/L$ minimizes
the convergence quotient $Q=1-1/(2\kappa)$, where $\kappa=L/\mu$
is the usual \emph{condition number}. This shows that, with an optimal
step-size, gradient descent needs at most $O(\kappa\log(1/\epsilon))$
iterations to find an $\epsilon$-suboptimal $X$. 

Matrix sensing with exact rank $r=r^{\star}$ is easily shown to satisfy
gradient dominance (\ref{eq:pl1}) by manipulating existing results
on (restricted) local strong convexity. In the over-parameterized
case $r>r^{\star}$, however, local strong convexity is lost, and
gradient dominance can fail to hold. Indeed, consider the following
instance of matrix sensing, with true rank $r^{\star}=1$, search
rank $r=2$, and $\AA$ set to the identity
\begin{equation}
f(X)=\|XX^{T}-zz^{T}\|_{F}^{2}\text{ where }X=\begin{bmatrix}1 & 0\\
0 & \xi
\end{bmatrix}\text{ and }z=\begin{bmatrix}1\\
0
\end{bmatrix}.\label{eq:counter1}
\end{equation}
We can verify that $\|\nabla f(X)\|^{2}=4\xi^{2}[f(X)-f^{\star}]$,
and this suggests that $f$ satisfies gradient dominance (\ref{eq:pl1})
with a constant of $\mu\le2\xi^{2}$. But $\xi$ is itself a variable
that goes to zero as the candidate $XX^{T}$ approaches to ground
truth $zz^{T}$. For every fixed $\mu>0$ in the gradient dominance
condition (\ref{eq:pl1}), we can find a counterexample $X$ in (\ref{eq:counter1})
with $\xi<\sqrt{\mu}/2$. Therefore, we must conclude that gradient
dominance fails to hold, because the inequality in (\ref{eq:pl1})
can only hold for $\mu=0$. 

In fact, this same example also shows why classical gradient descent
slows down to a sublinear rate. Applying gradient descent $X_{k+1}=X_{k}-\alpha\nabla f(X_{k})$
with fixed step-size $\alpha$ to (\ref{eq:counter1}) yields a sequence
of iterates of the same form
\begin{align*}
X_{0} & =\begin{bmatrix}1 & 0\\
0 & \xi_{0}
\end{bmatrix}, & X_{k+1} & =\begin{bmatrix}1 & 0\\
0 & \xi_{k+1}
\end{bmatrix}=\begin{bmatrix}1 & 0\\
0 & \xi_{k}-\alpha\xi_{k}^{3}
\end{bmatrix},
\end{align*}
from which we can verify that $f(X_{k+1})=(1-\alpha\xi_{k}^2)^{4}\cdot f(X_{k})$.
As each $k$-th $X_{k}X_{k}^{T}$ approaches $zz^{T}$, the element
$\xi_{k}$ converges towards zero, and the convergence quotient $Q=(1-\alpha\xi_{k}^2)^{4}$
approaches 1. We see a process of diminishing returns: every improvement
to $f$ worsens the quotient $Q$, thereby reducing the progress achievable
in the subsequent step. This is precisely the notion that characterizes
sublinear convergence.

\section{\label{sec:Proposed-Algorithm}Linear Convergence for the Noiseless
Case}

To understand how it is possible make gradient descent converge linearly
in the over-parameterized regime, we begin by considering gradient
method under a \emph{change of metric}. Let $\P$ be a real symmetric,
positive definite $nr\times nr$ matrix. We define a corresponding
$P$-inner product, $P$-norm, and dual $P$-norm on $\R^{n\times r}$
as follows 
\begin{align*}
\inner XY_{P} & \eqdef\vect(X)^{T}\P\vect(Y), & \|X\|_{P} & \eqdef\sqrt{\inner XX_{P}}, & \|X\|_{P*} & \eqdef\sqrt{\vect(X)^{T}\P^{-1}\vect(X)},
\end{align*}
where $\vect:\R^{n\times r}\to\R^{nr}$ is the usual column-stacking
operation. Consider descending in the direction $D$ satisfying $\vect(D)=\P^{-1}\vect(\nabla f(X))$;
the resulting decrement in $f$ can be quantified by a $P$-norm analog
of the Taylor-like expansion (\ref{eq_improve}) 
\begin{align}
f(X-\alpha D) & \leq f(X)-\alpha\langle\nabla f(X),D\rangle+\alpha^{2}(L_{P}/2)\|D\|_{P}^{2}\label{eq_improve-1}\\
 & =f(X)-\alpha(1-\alpha(L_{P}/2))\|\nabla f(X)\|_{P*}^{2}
\end{align}
where $L_{P}$ is a $P$-norm gradient Lipschitz constant. If we can
demonstrate gradient dominance under the dual $P$-norm,
\begin{equation}
\|\nabla f(X)\|_{P*}^{2}\geq\mu_{P}(f(X)-f^{\star})\quad\text{ where }\mu_{P}>0\text{ and }f^{\star}=\min f(X),\label{eq:gd-P}
\end{equation}
then we have the desired linear convergence
\begin{align}
f(X-\alpha D)-f^{\star} & \le[1-\mu_{P}\alpha(1-\alpha L_{P}/2)]\cdot(f(X)-f^{\star})\label{eq:gdconv-1}\\
 & =[1-1/(2\kappa_{P})]\cdot(f(X)-f^{\star})\text{ with }\alpha=1/L_{P},\label{eq:cond}
\end{align}
in which the condition number $\kappa_{P}=L_{P}/\mu_{P}$ should be
upper-bounded. To make the most progress per iteration, we want to
pick a metric $\P$ to make the condition number $\kappa_{P}$ as
small as possible.

The best choice of $\P$ for the fastest convergence is simply the
Hessian $\nabla^{2}f(X)$ itself, but this simply recovers Newton's
method, which would force us to invert a large $nr\times nr$ matrix
to compute the search direction $D$ at every iteration. Instead,
we look for a \emph{preconditioner} $\P$ that is cheap to apply while
still assuring a relatively small condition number $\kappa_{P}$.
The following choice is particularly interesting (the Kronecker product $\otimes$ is defined to satisfy $\vect(AXB^T)=(B\otimes A) \vect(X)$)
\begin{align*}
\P & =(X^{T}X+\eta I_{r})\otimes I_{n}=X^{T}X\otimes I_{n}+\eta I_{nr},
\end{align*}
because the resulting $D=\nabla f(X)(X^{T}X+\eta I)^{-1}$ allow us to
\emph{interpolate} between regular GD and the ScaledGD of \citet{tong2020accelerating}.
Indeed, we recover regular GD in the limit $\eta\to\infty$, but as
we saw in Section~\ref{sec:Fail}, gradient dominance (\ref{eq:gd-P})
fails to hold, so the condition number $\kappa_{P}=L_{P}/\mu_{P}$
grows unbounded as $\mu_{P}\to0$. Instead, setting $\eta=0$ recovers
ScaledGD. The key insight of \citet{tong2020accelerating} is that
under this choice of $\P$, gradient dominance (\ref{eq:gd-P}) is
guaranteed to hold, with a large value of $\mu_{P}$ that is independent
of the current iterate and the ground truth. But as we will now show,
this change of metric can magnify the Lipschitz constant $L_{P}$
by a factor of $\lambda_{\min}^{-1}(X^{T}X)$, so the condition
number $\kappa_{P}=L_{P}/\mu_{P}$ becomes unbounded in the over-parameterized
regime.
\begin{lem}[Lipschitz-like inequality]
\label{lem:lipp}Let $\|D\|_{P}=\|D(X^{T}X+\eta I_r)^{1/2}\|_{F}$. Then
we have
\[
f(X+D)\le f(X)+\inner{\nabla f(X)}D+\frac{1}{2}L_{P}(X,D)\|D\|_{P}^{2}
\]
where
\[
L_{P}(X,D)=2(1+\delta)\left[4+\frac{2\|XX^{T}-M^{\star}\|_{F}+4\|D\|_{P}}{\lambda_{\min}(X^{T}X)+\eta}+\left(\frac{\|D\|_{P}}{\lambda_{\min}(X^{T}X)+\eta}\right)^{2}\right]
\]
\end{lem}

\begin{lem}[Bounded gradient]
\label{lem:bndgrad}For the search direction $D=\nabla f(X)(X^{T}X+\eta I)^{-1}$,
we have $\|D\|_{P}^{2}=\|\nabla f(X)\|_{P*}^{2}\le16(1+\delta)f(X)$.
\end{lem}

The proofs of Lemma~\ref{lem:lipp} and Lemma~\ref{lem:bndgrad}
follows from straightforward linear algebra, and can be found in the
Appendix. Substituting Lemma~\ref{lem:bndgrad} into Lemma~\ref{lem:lipp},
we see for ScaledGD (with $\eta=0$) that the Lipschitz-like constant
is bounded as follows
\begin{equation}
L_{P}(X,D)\lesssim\left(\|XX^{T}-M^{\star}\|_{F}/\lambda_{\min}(X^{T}X)\right)^{2}.\label{eq:Lipbnd}
\end{equation}
In the exact rank case $r=r^{\star}$, the distance of $X$ from singularity
can be lower-bounded, within a ``good'' neighborhood of the ground
truth, since $\lambda_{\min}(X^{T}X)=\lambda_{r}(X^{T}X)$ and
\begin{equation}
\|XX^{T}-M^{\star}\|_{F}\le\rho\lambda_{r}(M^{\star}),\quad\rho<1\implies\lambda_{r}(X^{T}X)\ge(1-\rho)\lambda_{r}(M^{\star})>0.\label{eq:nei1}
\end{equation}
Within this ``good'' neighborhood, substituting (\ref{eq:nei1})
into (\ref{eq:Lipbnd}) yields a Lipschitz constant $L_{P}$ that
depends only on the radius $\rho$. The resulting iterations converge
rapidly, independent of any ill-conditioning in the model $XX^{T}$
nor in the ground-truth $M^{\star}$. In turn, ScaledGD can be initialized
within the good neighborhood using spectral initialization (see
Proposition~\ref{prop:spectinit} below). 

In the over-parameterized case $r>r^{\star}$, however, the iterate
$X$ must become singular in order for $XX^{T}$ to converge to $M^{\star}$,
and the radius of the ``good'' neighorhood reduces to zero. The
ScaledGD direction guarantees a large linear progress no matter how
singular $X$ may be, but the method may not be able to take a substantial
step in this direction if $X$ becomes singular too quickly. To illustrate:
the algorithm would fail entirely if it lands at on a point where
$\lambda_{\min}(X^{T}X)=0$ but $XX^{T}\not=M^{\star}$.

While regular GD struggles to make the smallest eigenvalues of $XX^{T}$
converge to zero, ScaledGD gets in trouble by making these eigenvalues
converge quickly. In finding a good mix between these two methods,
an intuitive idea is to use the damping parameter $\eta$ to control
the rate at which $X$ becomes singular. More rigorously, we can pick
an $\eta\approx\|XX^{T}-ZZ^{T}\|_{F}$ and use Lemma~\ref{lem:lipp}
to keep the Lipschitz constant $L_{P}$ bounded. Substituting Lemma~\ref{lem:bndgrad}
into Lemma~\ref{lem:lipp} and using RIP to upper-bound $f(X)\le(1+\delta)\|XX^{T}-M^{\star}\|_{F}^{2}$
and $\delta\le1$ yields
\begin{equation}
\eta\ge C_{\lb}\|XX^{T}-ZZ^{T}\|_{F}\implies L_{P}(X,D)\le16+136/C_{\lb}+256/C_{\lb}^{2}.\label{eq:lipbnd}
\end{equation}
However, the gradient dominance condition (\ref{eq:gd-P}) will necessarily
fail if $\eta$ is set too large. Our main result in this paper is
that keeping $\eta$ within the same order of magnitude as the error
norm $\|XX^{T}-ZZ^{T}\|_{F}$ is enough to maintain gradient dominance.
The following is the noiseless version of this result.
\begin{thm}[Noiseless gradient dominance]
\label{thm:pl}Let $\min_{X}f(X)=0$ for $M^{\star}\ne0$. Suppose
that $X$ satisfies $f(X)\le\rho^{2}\cdot(1-\delta)\lambda_{r^{\star}}^{2}(M^{\star})$
with radius $\rho>0$ that satisfies $\rho^{2}/(1-\rho^{2})\le(1-\delta^{2})/2$.
Then, we have
\[
\eta\le C_{\ub}\|XX^{T}-ZZ^{T}\|_{F}\quad\implies\quad\|\nabla f(X)\|_{P*}^{2}\ge2\mu_{P}f(X)
\]
where
\begin{equation}
\mu_{P}=\left(\sqrt{\frac{1+\delta^{2}}{2}}-\delta\right)^{2}\cdot\min\left\{ \left(\frac{C_{\ub}}{\sqrt{2}-1}\right)^{-1},\left(1+3C_{\ub}\sqrt{\frac{(r-r^{\star})}{1-\delta^{2}}}\right)^{-1}\right\}. \label{eq:mueq}
\end{equation}
\end{thm}

The proof of Theorem~\ref{thm:pl} is involved and 
we defer the details to the Appendix. In the noiseless case, we get
a good estimate of $\eta$ for free as a consequence of RIP:
\[
\eta=\sqrt{f(X)}\implies\sqrt{1-\delta}\|XX^{T}-M^{\star}\|_{F}\le\eta\le\sqrt{1+\delta}\|XX^{T}-M^{\star}\|_{F}.
\]
Repeating (\ref{eq_improve-1})-(\ref{eq:cond}) with Lemma~\ref{lem:lipp},
(\ref{eq:lipbnd}) and (\ref{eq:mueq}) yields our main result below. 

\begin{cor}[Linear convergence]
\label{cor:main}Let $X$ satisfy the same initial conditions as
in Theorem \ref{thm:pl}. The search direction $D=\nabla f(X)(X^{T}X+\eta I)^{-1}$
with damping parameter $\eta=\sqrt{f(X)}$ and step-size $\alpha\le1/L_{P}$
yields 
\[
f(X-\alpha D)\le\left(1-\alpha\mu_{P}/2\right)f(X)
\]
where $L_{P}$ is as in (\ref{eq:lipbnd}) with $C_{\lb}=\sqrt{1-\delta}$
and $\mu_{P}$ is as in (\ref{eq:mueq}) with $C_{\ub}=\sqrt{1+\delta}$.
\end{cor}

For a fixed RIP constant $\delta$, Corollary~\ref{cor:main} says
that PrecGD converges at a linear rate that is independent of the
current iterate $X$, and also independent of possible ill-conditioning
in the ground truth. However, it does require an initial point $X_{0}$
that satisfies
\begin{align}
\|\mathcal{A}(X_{0}X_{0}^{T}-M^{*})\|^{2} & <\rho^{2}(1-\delta)\lambda_{r^{*}}(M^{\star})^{2}\label{eq:radius}
\end{align}
with a radius $\rho>0$ satisfying $\rho^{2}/(1-\rho^{2})\le(1-\delta^{2})/2.$
Such an initial point can be found using spectral initialization,
even if the measurements are tainted with noise. Concretely, we choose
the initial point $X_{0}$ as 
\begin{equation}
X_{0}=\mathcal{P}_{r}\left(\frac{1}{m}\sum_{i=1}^{m}y_{i}A_{i}\right)\text{ where }\mathcal{P}_{r}(M)=\arg\min_{X\in\R^{n\times r}}\|XX^{T}-M\|_{F},\label{eq:spec}
\end{equation}
where we recall that $y=\mathcal{A}(M^{\star})+\epsilon$ are the
$m$ possibly noisy measurements collected of the ground truth, and
that the rank-$r$ projection operator can be efficiently implemented
with a singular value decomposition. The proof of the following proposition can be found in the appendix. 
\begin{prop}[Spectral Initialization]
\label{prop:spectinit}Suppose that $\delta\leq(8\kappa\sqrt{r^{*}})^{-1}$
and $m\gtrsim\frac{1+\delta}{1-\delta}\frac{\sigma^{2}rn\log n}{\rho^{2}\lambda_{r^{\star}}^{2}(M^{\star})}$
where $\kappa=\lambda_{1}(M^{\star})/\lambda_{r^{\star}}(M^{\star})$.
Then, with high probability, the initial point $X_{0}$ produced by~\eqref{eq:spec}
satisfies the radius condition (\ref{eq:radius}).
\end{prop}

However, if the measurements $y$ are noisy, then $\sqrt{f(X)}=\|\AA(XX^{T}-M^{\star})+\varepsilon\|$
now gives a biased estimate of our desired damping parameter $\eta$.
In the next section, we show that a good choice of $\eta_{k}$ is
available based on an approximation of the noise variance.

\section{Extension to Noisy Setting}

In this section, we extend our analysis to the matrix sensing with
noisy measurements. Our main goal is to show that, with a proper choice
of the damping coefficient $\eta$, the proposed algorithm converges
linearly to an ``optimal'' estimation error.
\begin{thm}[Noisy measurements with optimal $\eta$]
\label{thm_oracle} Suppose that the noise vector $\epsilon\in\mathbb{R}^{m}$
has sub-Gaussian entries with zero mean and variance $\sigma^{2}=\frac{1}{m}\sum_{i=1}^{m}\mathbb{E}[\epsilon_{i}^{2}]$.
Moreover, suppose that $\eta_{k}=\frac{1}{\sqrt{m}}\|\mathcal{A}(X_{k}X_{k}^{T}-M^{*})\|$,
for $k=0,1,\dots,K$, and that the initial point $X_{0}$ satisfies
$\|\mathcal{A}(X_{0}X_{0}^{T}-M^{*})\|^{2}<\rho^{2}(1-\delta)\lambda_{r^{*}}(M^{\star})^{2}$.
Consider $k^{*}=\arg\min_{k}\eta_{k}$, 
and suppose that the step-size
$\alpha \leq 1/L,$ where $L>0$ is a constant that only depends on $\delta$.
Then, with high probability, we have 
\begin{align}
\|X_{k^{*}}X_{k^{*}}^{T}-M^{\star}\|_{F}^{2}\lesssim\max\left\{ \frac{1+\delta}{1-\delta}\left(1-\alpha\frac{\mu_{P}}{2}\right)^{K}\|X_{0}X_{0}^{T}-M^{*}\|_{F}^{2},\mathcal{E}_{stat}\right\} ,
\end{align}
where $\mathcal{E}_{stat}:=\frac{\sigma^{2}nr\log n}{\mu_{P}(1-\delta)m}$. 
\end{thm}

Assuming fixed parameters for the problem, the above theorem shows
that PrecGD outputs a solution with an estimation error of $\mathcal{O}\left(\mathcal{E}_{stat}\right)$
in $\mathcal{O}\left(\log\left({1}/{\mathcal{E}_{stat}}\right)\right)$
iterations. Moreover, the error $\mathcal{O}\left(\mathcal{E}_{stat}\right)$
is minimax optimal (modulo logarithmic factors), and cannot be improved
significantly. In particular,~\citet{candes2011tight} showed that
\textit{any} estimator $\widehat{X}$ must satisfy $\|\widehat{X}\widehat{X}^{T}-M^{*}\|_{F}^{2}\gtrsim\sigma^{2}nr/m$
with non-negligible probability. The classical methods for achieving
this minimax rate suffer from computationally-prohibitive per
iteration costs~\citep{chen2015fast,recht2010guaranteed,negahban2011estimation}.
Regular gradient descent alleviates this issue at the expense of a
slower convergence rate of $\mathcal{O}(\sqrt{1/\mathcal{E}_{stat}})$~\citep{zhuo2021computational}.
Our proposed PrecGD achieves the best of both worlds: it converges
to the minimax optimal error with cheap per-iteration complexity of
$\mathcal{O}(nr^{2}+r^{3})$, while benefiting from an exponentially
faster convergence rate than regular gradient descent in the over-parameterized
regime.

Theorem~\ref{thm_oracle} highlights the critical role of the damping
coefficient $\eta$ in the guaranteed linear convergence of the algorithm.
In the noiseless regime, we showed in the previous section that an
``optimal'' choice $\eta=\sqrt{f(X)}$ is available for free. In
the noisy setting, however, the same choice of $\eta$ becomes biased
by the noise variance, and is therefore no longer optimal. As is typically
the case for regularized estimation methods~\citep{de2005model,cawley2006leave,guo2011joint},
selecting the ideal parameter would amount to some kind of \textit{resampling},
such as via cross-validation or bootstrapping~\citep{good2006resampling,efron1994introduction,cox1979theoretical},
which is generally expensive to implement and use in practice. As
an alternative approach, we show in our next theorem that a good choice
of $\eta$ is available based on an approximation of the noise variance
$\sigma^{2}$.
\begin{thm}[Noisy measurements with variance proxy]
\label{thm_noisy_var} Suppose that the noise vector $\epsilon\in\mathbb{R}^{m}$
has sub-Gaussian entries with zero mean and variance $\sigma^{2}=\frac{1}{m}\sum_{i=1}^{m}\mathbb{E}[\epsilon_{i}^{2}]$.
Moreover, suppose that $\eta_{k}=\sqrt{|f(X_{k})-\hat{\sigma}^{2}|}$
for $k=0,1,\dots,K$, where $\hat{\sigma}^{2}$ is an approximation
of $\sigma^{2}$, and that the initial point $X_{0}$ satisfies $\|\mathcal{A}(X_{0}X_{0}^{T}-M^{*})\|_{F}^{2}<\rho^{2}(1-\delta)\lambda_{r^{*}}(M^{\star})^{2}$.
Consider $k^{*}=\arg\min_{k}\eta_{k}$, 
and suppose that the step-size
$\alpha \leq 1/L,$ where $L>0$ is a constant that only depends on $\delta$.
Then, with high probability, we have 
\begin{align}
\|X_{k^{*}}X_{k^{*}}^{T}-M^{*}\|_{F}^{2}\lesssim\max\Bigg\{ & \frac{1+\delta}{1-\delta}\left(1-\alpha\frac{\mu_{P}}{2}\right)^{K}\|X_{0}X_{0}^{T}-M^{*}\|_{F}^{2},\mathcal{E}_{stat},\mathcal{E}_{dev},\mathcal{E}_{var}\Bigg\},
\end{align}
where 
\begin{align}
\mathcal{E}_{stat}:=\frac{\sigma^{2}nr\log n}{\mu_{P}(1-\delta)m},\quad\mathcal{E}_{dev}:=\frac{\sigma^{2}}{1-\delta}\sqrt{\frac{\log n}{m}},\quad\mathcal{E}_{var}:=|\sigma^{2}-\hat{\sigma}^{2}|.
\end{align}
\end{thm}

In the above theorem, $\mathcal{E}_{dev}$ captures the deviation
of the empirical variance $\frac{1}{m}\sum_{i=1}^{m}\epsilon_{i}^{2}$
from its expectation $\sigma^{2}$. On the other hand, $\mathcal{E}_{var}$
captures the approximation error of the true variance. According to
Theorem~\ref{thm_noisy_var}, it is possible to chose the damping
factor $\eta_{k}$ merely based on $f(X_k)$ and an approximation of
$\sigma^{2}$, at the expense of a suboptimal estimation error rate.
In particular, suppose that the noise variance is known precisely,
i.e., $\hat{\sigma}^{2}=\sigma^{2}$. Then, the above theorem implies
that the estimation error is reduced to 
\[
\|X_{k^{*}}X_{k^{*}}^{T}-M^{*}\|_{F}^{2}\lesssim\max\left\{ \mathcal{E}_{stat},\mathcal{E}_{dev}\right\} \quad\text{after}\quad\mathcal{O}\left(\log\left(\frac{1}{\max\left\{ \mathcal{E}_{stat},\mathcal{E}_{dev}\right\} }\right)\right)\ \text{iterations.}
\]
If $m$ is not too large, i.e., $m\lesssim\sigma^{2}n^{2}r^{2}\log n$,
the estimation error can be improved to $\|X_{k^{*}}X_{k^{*}}^{T}-M^{*}\|_{F}^{2}\lesssim\mathcal{E}_{stat}$,
which is again optimal (modulo logarithmic factors). As $m$ increases, the estimation error will become smaller, but the convergence rate will decrease.
This suboptimal
rate is due to the heavy tail phenomenon arising from the concentration
of the noise variance. In particular, one can write 
\begin{align}
f(X)-\sigma^{2}=\frac{1}{m}\|\mathcal{A}(XX^{T}-M^{\star})\|^{2}+\underbrace{\frac{1}{m}\|\epsilon\|^{2}-\sigma^{2}}_{\text{variance deviation}}+\underbrace{\frac{2}{m}\langle{\AA(ZZ^{T}-XX^{T})},{\epsilon}\rangle}_{\text{cross-term}}
\end{align}
Evidently, $f(X)-\sigma^{2}$ is in the order of $\frac{1}{m}\|\mathcal{A}(XX^{T}-M^{\star})\|^{2}$
if both variance deviation and cross-term are dominated by $\frac{1}{m}\|\mathcal{A}(XX^{T}-M^{\star})\|^{2}$.
In the proof of Theorem~\ref{thm_noisy_var}, we show that, with
high probability, the variance deviation is upper bounded by $(1-\delta)\mathcal{E}_{dev}$
and it dominates the cross-term. This implies that the choice of $\eta=\sqrt{|f(X)-\sigma^{2}|}$
behaves similar to $\frac{1}{\sqrt{m}}\|\mathcal{A}(XX^{T}-M^{\star})\|$,
and hence, the result of Theorem~\ref{thm_oracle} can be invoked,
so long as 
\[
\frac{1}{m}\|\mathcal{A}(XX^{T}-M^{\star})\|^{2}\geq(1-\delta)\|XX^{T}-M^{\star}\|_{F}^{2}\gtrsim(1-\delta)\mathcal{E}_{dev}.
\]

\section{Numerical Experiments}

Finally, we numerically compare PrecGD on other matrix factorization
problems that fall outside of the matrix sensing framework. We consider
the $\ell_{p}$ empirical loss $f_{p}(X)=\sum_{i=1}^{m}|\langle A_{i},XX^{T}-M^{\star}\rangle|^{p}$
for $1\le p<2$, in order to gauge the effectiveness of PrecGD for
increasing nonsmooth loss functions. Here, we set the damping parameter
$\eta_{k}=[f_{p}(X_{k})]^{1/p}$ as a heuristic for the error $\|XX^{T}-M^{\star}\|_{F}$.
 The data matrices $A_{1},\dots,A_{m}$ were taken from~\citep[Example~12]{zhang2019sharp},
the ground truth $M^{\star}=ZZ^{T}$ was constructed by sampling each
column of $Z\in\R^{n\times r^{\star}}$ from the standard Gaussian,
and then rescaling the last column to achieve a desired condition
number. 

The recent work of Tong et al. \cite{tong2021low} showed that in the exactly-parameterized setting, ScaledGD works well for the $\ell_1$ loss function. In particular, if the initial point is close to the ground truth, then with a Polyak stepsize $\alpha_k = f(X_k)/\|\nabla f(X_k)\|_{P}^*$, ScaledGD converges linearly to the ground truth. However, these theoretical guarantees no longer hold in the over-parameterized regime. 

When $r>r^*$, our numerical experiments show that ScaledGD blows up due to singularity near the ground truth while PrecGD continues to converge linearly in this nonsmooth, over-parameterized setting. In Figure
\ref{fig:p_over} we compare GD, ScaledGD and PrecGD in the exact and
over-parameterized regimes for the $\ell_{p}$ norm, with $p=1.1,1.4$
and $1.7$. For ScaledGD and PrecGD, we used a modified version of the Polyak step-size where $\alpha_k = f(X_k)^p/\|\nabla f(X_k)\|_P^*$. For GD we use a decaying stepsize. 
When $r=r^{*},$ we see that both ScaledGD and PrecGD converge linearly, but
GD stagnates due to ill-conditioning of the ground truth. When $r>r^{*}$, GD still converges
slowly and ScaledGD blows up very quickly, while PrecGD continues to
converge reliably.


\begin{figure}[ht]
\begin{tikzpicture}
\newcommand{\coord}{4.1}
\newcommand{\wid}{0.3}
\newcommand{\rl}{0.25}
\node[] at (0.5,2) {\large $p=1.1$
};
\node[] at (4.6,2) {\large $p=1.4$
};
\node[] at (8.7,2) {\large $p=1.7$
};

\node[rotate=0] at (-2.5,-0) {\large $r=r^*$
};
\node[rotate=0] at (-2.5,-3.5) {\large $r>r^*$
};
\node (image) at (\rl,0) {
\includegraphics[width = \wid\textwidth]{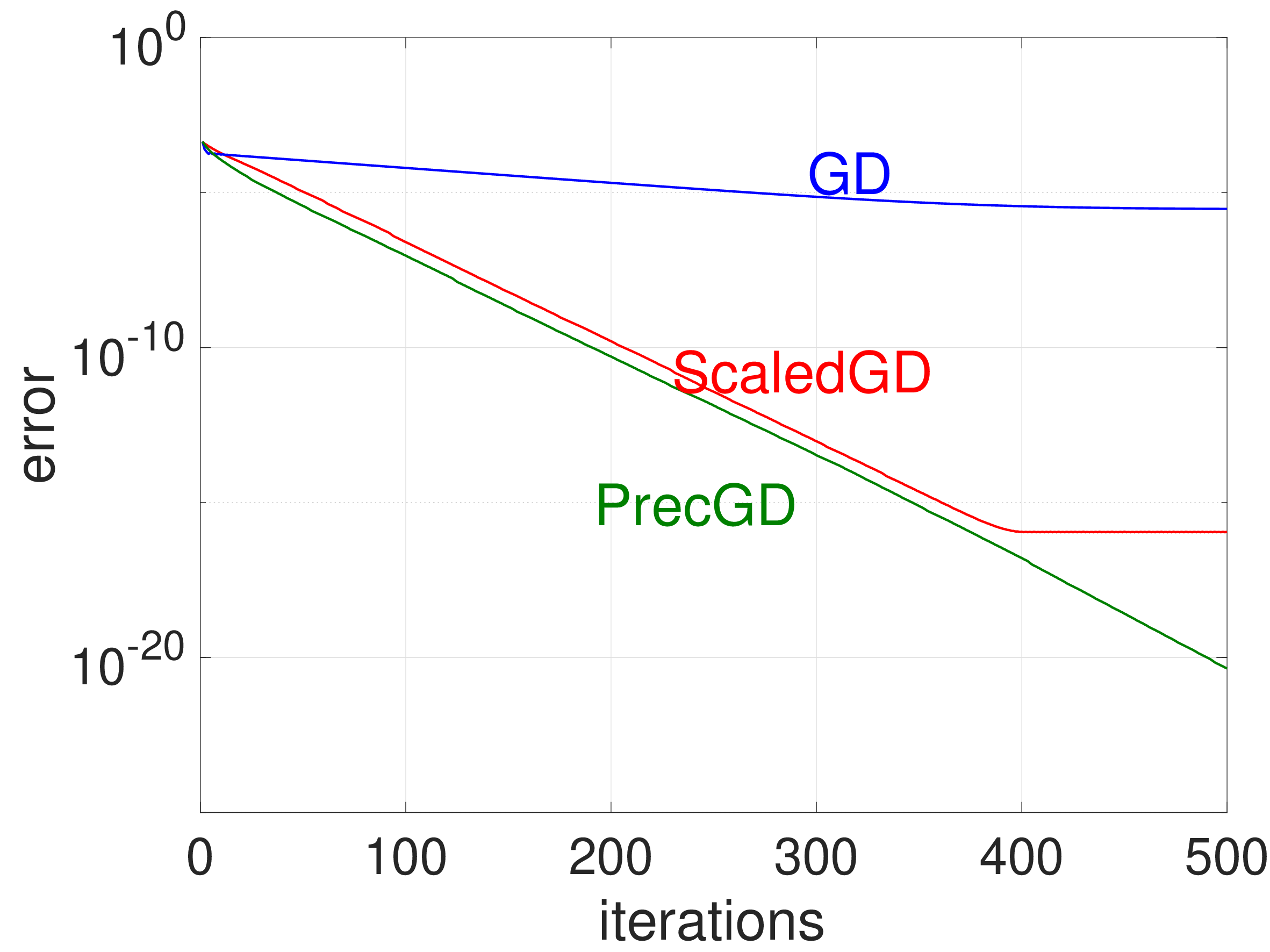}
};
\node (image) at (\rl+\coord,0) {
\includegraphics[width = \wid\textwidth]{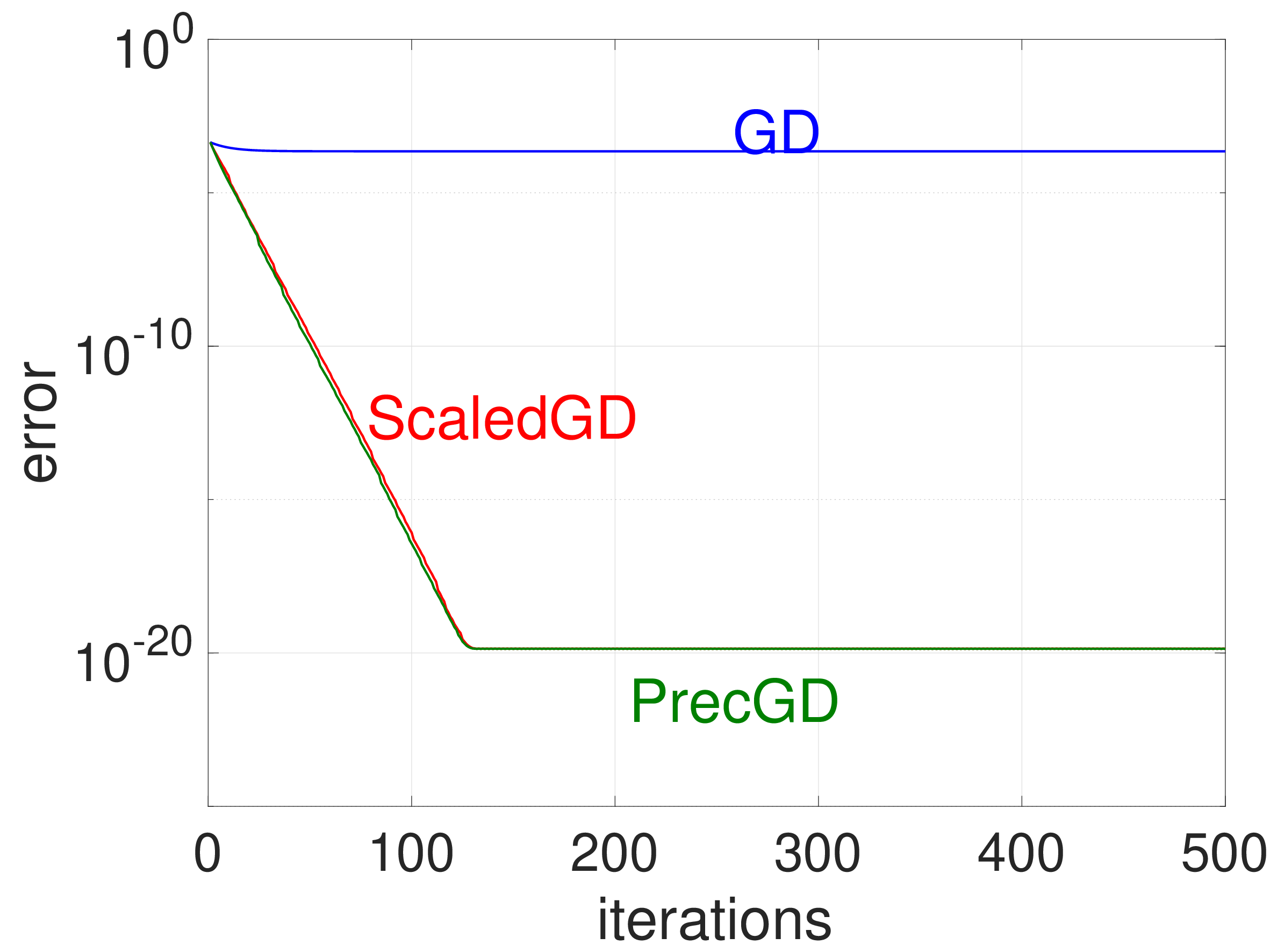}
};
\node (image) at (\rl+2*\coord,0) {
\includegraphics[width = \wid\textwidth]{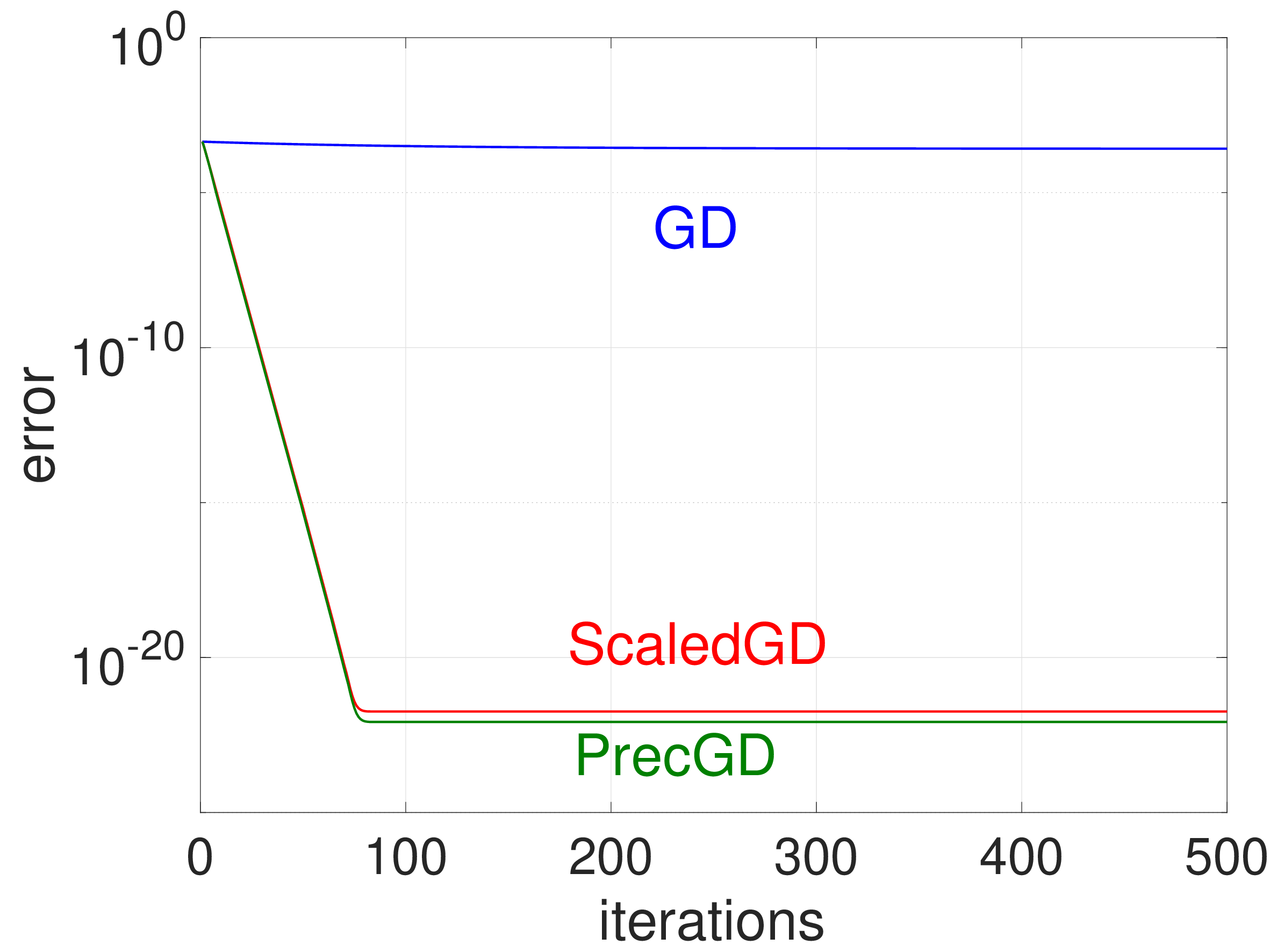}
};
\node (image) at (\rl,-3.5) {
\includegraphics[width = \wid\textwidth]{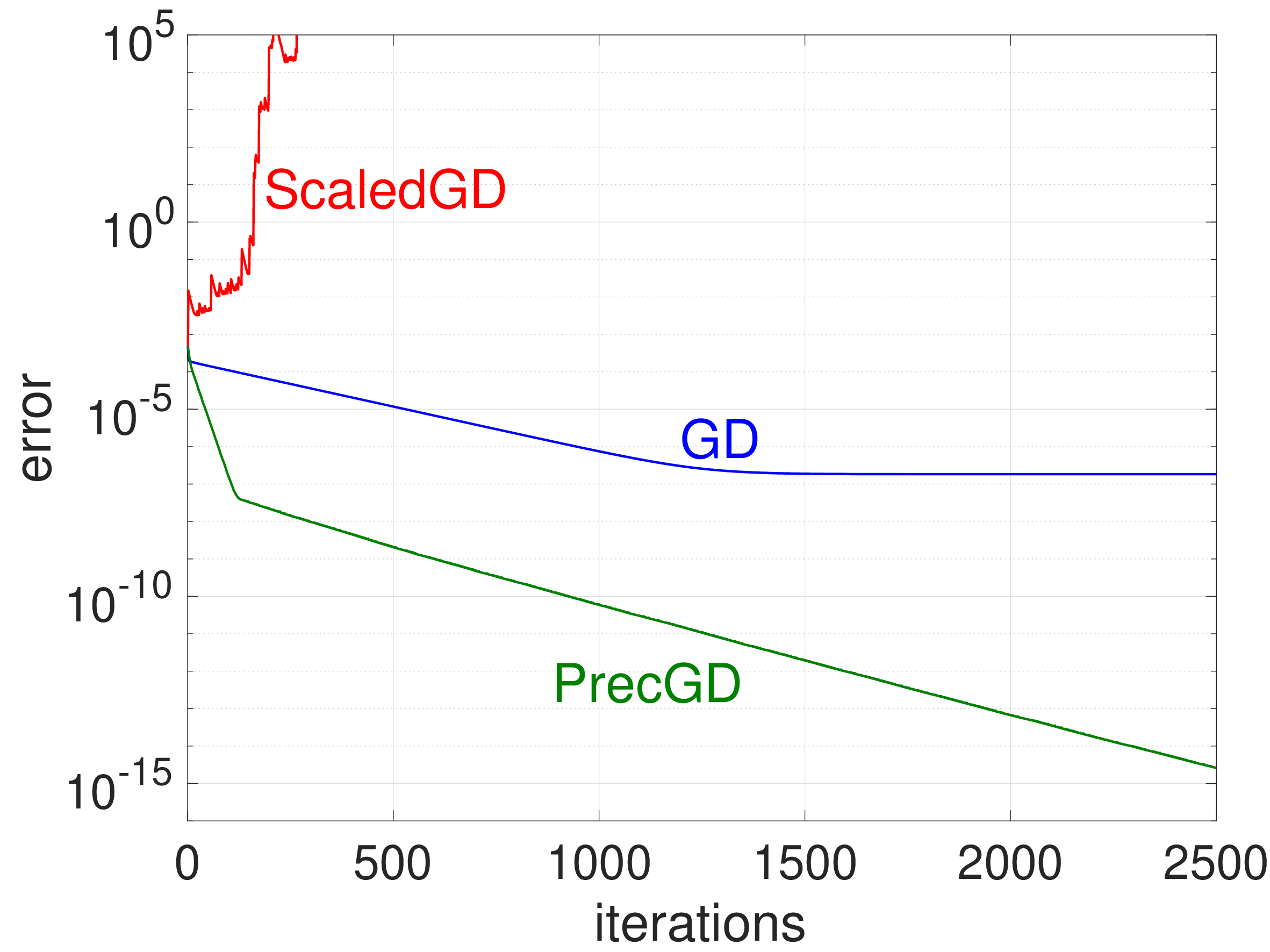}
};
\node (image) at (\rl+\coord,-3.5) {
\includegraphics[width = \wid\textwidth]{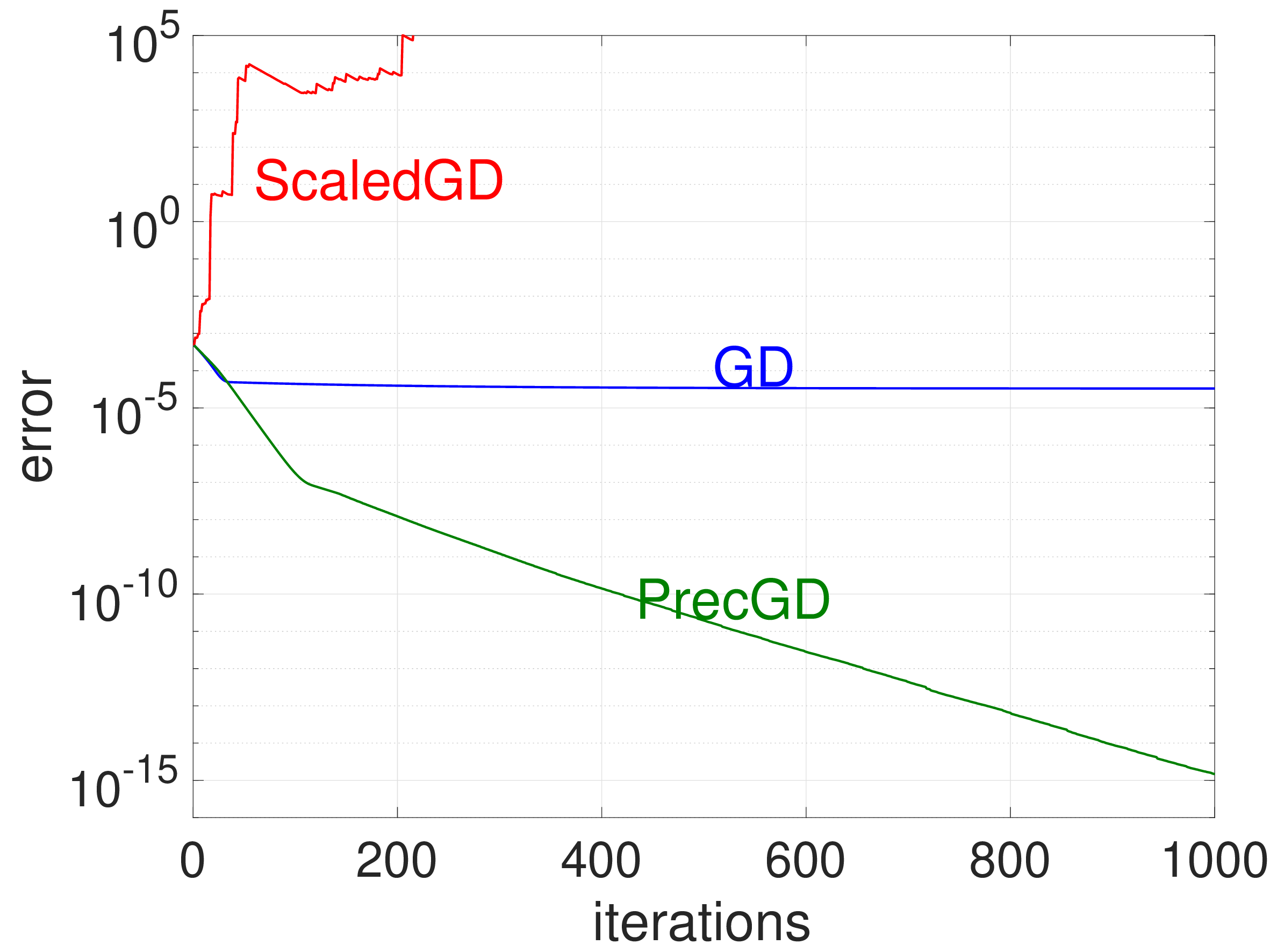}
};
\node (image) at (\rl+2*\coord,-3.5) {
\includegraphics[width = \wid\textwidth]{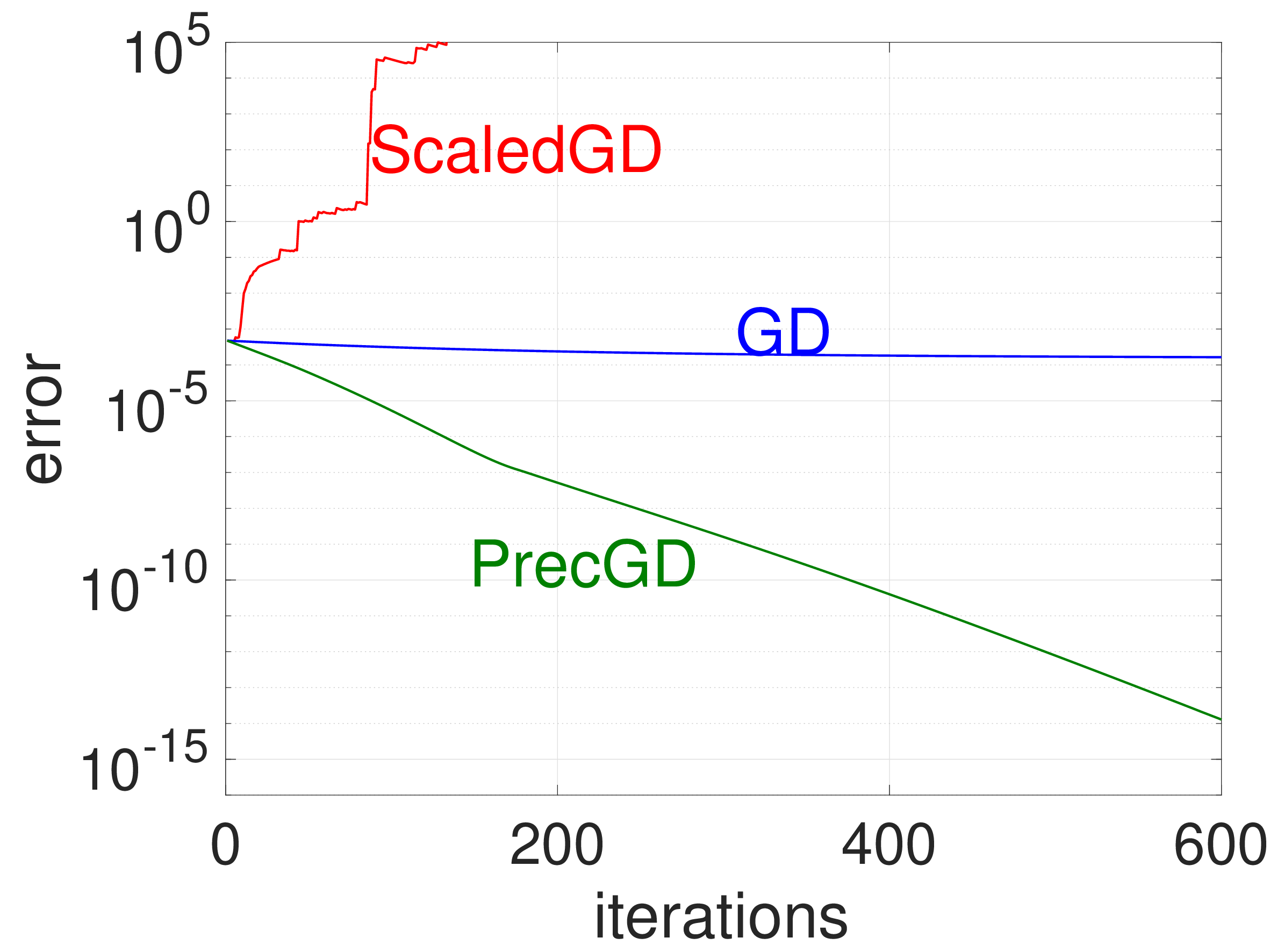}
};
\end{tikzpicture}

\caption{\textbf{Nonconvex matrix factorization with the $\ell_{p}$ empirical
loss}. We compare $\ell_{p}$ matrix sensing with $n=10$ and $r^{\star}=2$
and $\protect\AA$ taken from~\citep{zhang2019sharp}. The ground
truth is chosen to be ill-conditioned ($\kappa=10^{2})$. For ScaledGD and PrecGD, we use the Polyak step-size in \cite{tong2021low}. For GD we use a decaying step-size.
\textbf{(Top $r=r^{*}$)} For all three values of $p$,
GD stagnates due to the ill-conditioning of the ground truth, while ScaledGD and PrecGD converge linearly in all three cases. 
\textbf{(Bottom $r>r^{*}$)} With $r=4$, the problem is over-parameterized. GD again converges slowly
and ScaledGD is sporadic due to near-singularity caused by over-parameterization.
Once again we see PrecGD converge at a linear rate.}
\label{fig:p_over} 
\end{figure}
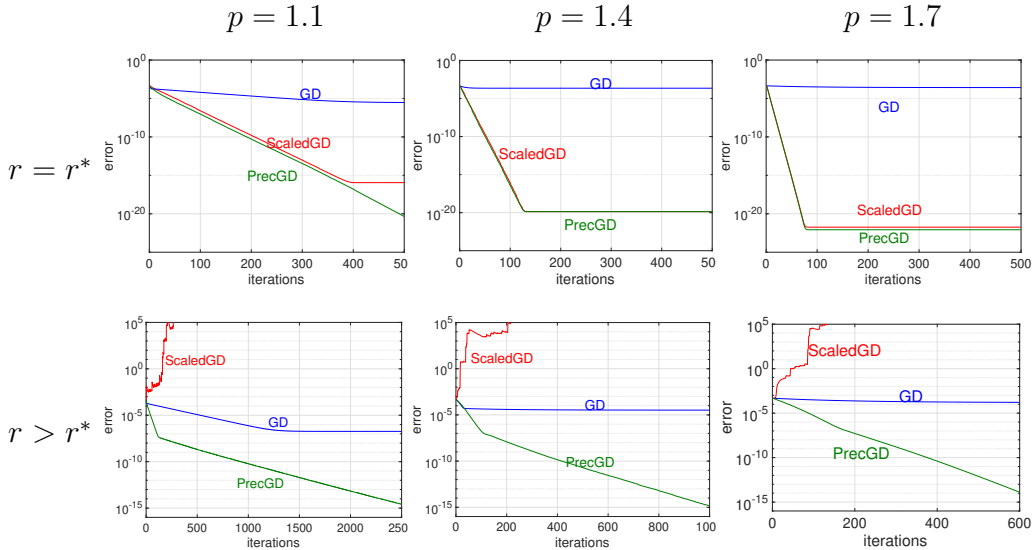

\section{Conclusions}

In this paper, we propose a \emph{preconditioned} gradient descent
or PrecGD for nonconvex matrix factorization with a comparable per-iteration
cost to classical gradient descent. For over-parameterized matrix
sensing, gradient descent slows down to a sublinear convergence rate,
but PrecGD restores the convergence rate back to linear, while also
making the iterations immune to ill-conditioning in the ground truth.
While the thoeretical analysis in our paper uses some properties specific to RIP matrix sensing, our numerical experiments find that PrecGD works well for even for nonsmooth loss functions. We believe that these current results can be extended to similar problems such as matrix completion and robust PCA, where properties like incoherence can be used to select the damping parameter $\eta_k$ with the desired properties, so that PrecGD converges linearly as well. It remains future work to provide rigorous justification for these observations. 

\section*{Acknowledgements}
G.Z. and R.Y.Z are supported by the NSF CAREER Award ECCS-2047462. S.F. is supported by MICDE Catalyst Grant and MIDAS PODS Grant. We also thank an anonymous reviewer who provided a simplified proof of Lemma \ref{lem:reform} and made various insightful comments to help us improve an earlier version of this work.

\medskip
\bibliographystyle{unsrtnat}
\bibliography{reference}

\begin{thebibliography}{57}
\providecommand{\natexlab}[1]{#1}
\providecommand{\url}[1]{\texttt{#1}}
\expandafter\ifx\csname urlstyle\endcsname\relax
  \providecommand{\doi}[1]{doi: #1}\else
  \providecommand{\doi}{doi: \begingroup \urlstyle{rm}\Url}\fi

\bibitem[Yu et~al.(2009)Yu, Zhu, Lafferty, and Gong]{yu2009fast}
Kai Yu, Shenghuo Zhu, John Lafferty, and Yihong Gong.
\newblock Fast nonparametric matrix factorization for large-scale collaborative
  filtering.
\newblock In \emph{Proceedings of the 32nd international ACM SIGIR conference
  on Research and development in information retrieval}, pages 211--218, 2009.

\bibitem[Luo et~al.(2014)Luo, Zhou, Xia, and Zhu]{luo2014efficient}
Xin Luo, Mengchu Zhou, Yunni Xia, and Qingsheng Zhu.
\newblock An efficient non-negative matrix-factorization-based approach to
  collaborative filtering for recommender systems.
\newblock \emph{IEEE Transactions on Industrial Informatics}, 10\penalty0
  (2):\penalty0 1273--1284, 2014.

\bibitem[Cand{\`e}s et~al.(2011)Cand{\`e}s, Li, Ma, and
  Wright]{candes2011robust}
Emmanuel~J Cand{\`e}s, Xiaodong Li, Yi~Ma, and John Wright.
\newblock Robust principal component analysis?
\newblock \emph{Journal of the ACM (JACM)}, 58\penalty0 (3):\penalty0 1--37,
  2011.

\bibitem[Chandrasekaran et~al.(2011)Chandrasekaran, Sanghavi, Parrilo, and
  Willsky]{chandrasekaran2011rank}
Venkat Chandrasekaran, Sujay Sanghavi, Pablo~A Parrilo, and Alan~S Willsky.
\newblock Rank-sparsity incoherence for matrix decomposition.
\newblock \emph{SIAM Journal on Optimization}, 21\penalty0 (2):\penalty0
  572--596, 2011.

\bibitem[Ahmed et~al.(2013)Ahmed, Recht, and Romberg]{ahmed2013blind}
Ali Ahmed, Benjamin Recht, and Justin Romberg.
\newblock Blind deconvolution using convex programming.
\newblock \emph{IEEE Transactions on Information Theory}, 60\penalty0
  (3):\penalty0 1711--1732, 2013.

\bibitem[Ling and Strohmer(2015)]{ling2015self}
Shuyang Ling and Thomas Strohmer.
\newblock Self-calibration and biconvex compressive sensing.
\newblock \emph{Inverse Problems}, 31\penalty0 (11):\penalty0 115002, 2015.

\bibitem[Singer(2011)]{singer2011angular}
Amit Singer.
\newblock Angular synchronization by eigenvectors and semidefinite programming.
\newblock \emph{Applied and computational harmonic analysis}, 30\penalty0
  (1):\penalty0 20--36, 2011.

\bibitem[Zhuo et~al.(2021)Zhuo, Kwon, Ho, and Caramanis]{zhuo2021computational}
Jiacheng Zhuo, Jeongyeol Kwon, Nhat Ho, and Constantine Caramanis.
\newblock On the computational and statistical complexity of over-parameterized
  matrix sensing.
\newblock \emph{arXiv preprint arXiv:2102.02756}, 2021.

\bibitem[Zheng and Lafferty(2015{\natexlab{a}})]{NIPS2015_32bb90e8}
Qinqing Zheng and John Lafferty.
\newblock A convergent gradient descent algorithm for rank minimization and
  semidefinite programming from random linear measurements.
\newblock In \emph{Advances in Neural Information Processing Systems},
  volume~28, 2015{\natexlab{a}}.

\bibitem[Tu et~al.(2016)Tu, Boczar, Simchowitz, Soltanolkotabi, and
  Recht]{tu2016low}
Stephen Tu, Ross Boczar, Max Simchowitz, Mahdi Soltanolkotabi, and Ben Recht.
\newblock Low-rank solutions of linear matrix equations via procrustes flow.
\newblock In \emph{International Conference on Machine Learning}, pages
  964--973. PMLR, 2016.

\bibitem[Tong et~al.(2020)Tong, Ma, and Chi]{tong2020accelerating}
Tian Tong, Cong Ma, and Yuejie Chi.
\newblock Accelerating ill-conditioned low-rank matrix estimation via scaled
  gradient descent.
\newblock \emph{arXiv preprint arXiv:2005.08898}, 2020.

\bibitem[Zhang et~al.(2018)Zhang, Josz, Sojoudi, and
  Lavaei]{NEURIPS2018_f8da71e5}
Richard Zhang, Cedric Josz, Somayeh Sojoudi, and Javad Lavaei.
\newblock How much restricted isometry is needed in nonconvex matrix recovery?
\newblock In \emph{Advances in Neural Information Processing Systems},
  volume~31, 2018.

\bibitem[Zhang et~al.(2019)Zhang, Sojoudi, and Lavaei]{zhang2019sharp}
Richard~Y Zhang, Somayeh Sojoudi, and Javad Lavaei.
\newblock Sharp restricted isometry bounds for the inexistence of spurious
  local minima in nonconvex matrix recovery.
\newblock \emph{Journal of Machine Learning Research}, 20\penalty0
  (114):\penalty0 1--34, 2019.

\bibitem[Candes and Plan(2011)]{candes2011tight}
Emmanuel~J Candes and Yaniv Plan.
\newblock Tight oracle inequalities for low-rank matrix recovery from a minimal
  number of noisy random measurements.
\newblock \emph{IEEE Transactions on Information Theory}, 57\penalty0
  (4):\penalty0 2342--2359, 2011.

\bibitem[Recht et~al.(2010)Recht, Fazel, and Parrilo]{recht2010guaranteed}
Benjamin Recht, Maryam Fazel, and Pablo~A Parrilo.
\newblock Guaranteed minimum-rank solutions of linear matrix equations via
  nuclear norm minimization.
\newblock \emph{SIAM review}, 52\penalty0 (3):\penalty0 471--501, 2010.

\bibitem[Zheng and Lafferty(2015{\natexlab{b}})]{zheng2015convergent}
Qinqing Zheng and John Lafferty.
\newblock A convergent gradient descent algorithm for rank minimization and
  semidefinite programming from random linear measurements.
\newblock \emph{arXiv preprint arXiv:1506.06081}, 2015{\natexlab{b}}.

\bibitem[Bhojanapalli et~al.(2016{\natexlab{a}})Bhojanapalli, Kyrillidis, and
  Sanghavi]{bhojanapalli2016dropping}
Srinadh Bhojanapalli, Anastasios Kyrillidis, and Sujay Sanghavi.
\newblock Dropping convexity for faster semi-definite optimization.
\newblock In \emph{Conference on Learning Theory}, pages 530--582. PMLR,
  2016{\natexlab{a}}.

\bibitem[Candes et~al.(2015)Candes, Li, and Soltanolkotabi]{candes2015phase}
Emmanuel~J Candes, Xiaodong Li, and Mahdi Soltanolkotabi.
\newblock Phase retrieval via wirtinger flow: Theory and algorithms.
\newblock \emph{IEEE Transactions on Information Theory}, 61\penalty0
  (4):\penalty0 1985--2007, 2015.

\bibitem[Ma and Fattahi(2021)]{ma2021implicit}
Jianhao Ma and Salar Fattahi.
\newblock Implicit regularization of sub-gradient method in robust matrix
  recovery: Don't be afraid of outliers.
\newblock \emph{arXiv preprint arXiv:2102.02969}, 2021.

\bibitem[Keshavan et~al.(2010)Keshavan, Montanari, and Oh]{keshavan2010matrix}
Raghunandan~H Keshavan, Andrea Montanari, and Sewoong Oh.
\newblock Matrix completion from a few entries.
\newblock \emph{IEEE transactions on information theory}, 56\penalty0
  (6):\penalty0 2980--2998, 2010.

\bibitem[Chen and Wainwright(2015)]{chen2015fast}
Yudong Chen and Martin~J Wainwright.
\newblock Fast low-rank estimation by projected gradient descent: General
  statistical and algorithmic guarantees.
\newblock \emph{arXiv preprint arXiv:1509.03025}, 2015.

\bibitem[Sun and Luo(2016)]{sun2016guaranteed}
Ruoyu Sun and Zhi-Quan Luo.
\newblock Guaranteed matrix completion via non-convex factorization.
\newblock \emph{IEEE Transactions on Information Theory}, 62\penalty0
  (11):\penalty0 6535--6579, 2016.

\bibitem[Netrapalli et~al.(2014)Netrapalli, Niranjan, Sanghavi, Anandkumar, and
  Jain]{netrapalli2014non}
Praneeth Netrapalli, UN~Niranjan, Sujay Sanghavi, Animashree Anandkumar, and
  Prateek Jain.
\newblock Non-convex robust pca.
\newblock \emph{arXiv preprint arXiv:1410.7660}, 2014.

\bibitem[Bhojanapalli et~al.(2016{\natexlab{b}})Bhojanapalli, Neyshabur, and
  Srebro]{bhojanapalli2016global}
Srinadh Bhojanapalli, Behnam Neyshabur, and Nathan Srebro.
\newblock Global optimality of local search for low rank matrix recovery.
\newblock \emph{arXiv preprint arXiv:1605.07221}, 2016{\natexlab{b}}.

\bibitem[Li et~al.(2019)Li, Zhu, and Tang]{li2019non}
Qiuwei Li, Zhihui Zhu, and Gongguo Tang.
\newblock The non-convex geometry of low-rank matrix optimization.
\newblock \emph{Information and Inference: A Journal of the IMA}, 8\penalty0
  (1):\penalty0 51--96, 2019.

\bibitem[Sun et~al.(2018)Sun, Qu, and Wright]{sun2018geometric}
Ju~Sun, Qing Qu, and John Wright.
\newblock A geometric analysis of phase retrieval.
\newblock \emph{Foundations of Computational Mathematics}, 18\penalty0
  (5):\penalty0 1131--1198, 2018.

\bibitem[Ge et~al.(2016)Ge, Lee, and Ma]{ge2016matrix}
Rong Ge, Jason~D Lee, and Tengyu Ma.
\newblock Matrix completion has no spurious local minimum.
\newblock \emph{arXiv preprint arXiv:1605.07272}, 2016.

\bibitem[Ge et~al.(2017)Ge, Jin, and Zheng]{ge2017no}
Rong Ge, Chi Jin, and Yi~Zheng.
\newblock No spurious local minima in nonconvex low rank problems: A unified
  geometric analysis.
\newblock In \emph{International Conference on Machine Learning}, pages
  1233--1242. PMLR, 2017.

\bibitem[Chen and Li(2017)]{chen2017memory}
Ji~Chen and Xiaodong Li.
\newblock Memory-efficient kernel pca via partial matrix sampling and nonconvex
  optimization: a model-free analysis of local minima.
\newblock \emph{arXiv preprint arXiv:1711.01742}, 2017.

\bibitem[Sun et~al.(2016)Sun, Qu, and Wright]{sun2016complete}
Ju~Sun, Qing Qu, and John Wright.
\newblock Complete dictionary recovery over the sphere i: Overview and the
  geometric picture.
\newblock \emph{IEEE Transactions on Information Theory}, 63\penalty0
  (2):\penalty0 853--884, 2016.

\bibitem[Zhang(2021)]{zhang2021sharp}
Richard~Y Zhang.
\newblock Sharp global guarantees for nonconvex low-rank matrix recovery in the
  overparameterized regime.
\newblock \emph{arXiv preprint arXiv:2104.10790}, 2021.

\bibitem[Ge et~al.(2015)Ge, Huang, Jin, and Yuan]{ge2015escaping}
Rong Ge, Furong Huang, Chi Jin, and Yang Yuan.
\newblock Escaping from saddle points—online stochastic gradient for tensor
  decomposition.
\newblock In \emph{Conference on learning theory}, pages 797--842. PMLR, 2015.

\bibitem[Jin et~al.(2017)Jin, Ge, Netrapalli, Kakade, and
  Jordan]{jin2017escape}
Chi Jin, Rong Ge, Praneeth Netrapalli, Sham~M Kakade, and Michael~I Jordan.
\newblock How to escape saddle points efficiently.
\newblock In \emph{International Conference on Machine Learning}, pages
  1724--1732. PMLR, 2017.

\bibitem[Meka et~al.(2009)Meka, Jain, and Dhillon]{meka2009guaranteed}
Raghu Meka, Prateek Jain, and Inderjit~S Dhillon.
\newblock Guaranteed rank minimization via singular value projection.
\newblock \emph{arXiv preprint arXiv:0909.5457}, 2009.

\bibitem[Cand{\`e}s and Recht(2009)]{candes2009exact}
Emmanuel~J Cand{\`e}s and Benjamin Recht.
\newblock Exact matrix completion via convex optimization.
\newblock \emph{Foundations of Computational mathematics}, 9\penalty0
  (6):\penalty0 717--772, 2009.

\bibitem[Cand{\`e}s and Tao(2010)]{candes2010power}
Emmanuel~J Cand{\`e}s and Terence Tao.
\newblock The power of convex relaxation: Near-optimal matrix completion.
\newblock \emph{IEEE Transactions on Information Theory}, 56\penalty0
  (5):\penalty0 2053--2080, 2010.

\bibitem[Alizadeh(1995)]{alizadeh1995interior}
Farid Alizadeh.
\newblock Interior point methods in semidefinite programming with applications
  to combinatorial optimization.
\newblock \emph{SIAM journal on Optimization}, 5\penalty0 (1):\penalty0 13--51,
  1995.

\bibitem[Wen et~al.(2010)Wen, Goldfarb, and Yin]{wen2010alternating}
Zaiwen Wen, Donald Goldfarb, and Wotao Yin.
\newblock Alternating direction augmented lagrangian methods for semidefinite
  programming.
\newblock \emph{Mathematical Programming Computation}, 2\penalty0
  (3-4):\penalty0 203--230, 2010.

\bibitem[O’donoghue et~al.(2016)O’donoghue, Chu, Parikh, and
  Boyd]{o2016conic}
Brendan O’donoghue, Eric Chu, Neal Parikh, and Stephen Boyd.
\newblock Conic optimization via operator splitting and homogeneous self-dual
  embedding.
\newblock \emph{Journal of Optimization Theory and Applications}, 169\penalty0
  (3):\penalty0 1042--1068, 2016.

\bibitem[Zheng et~al.(2020)Zheng, Fantuzzi, Papachristodoulou, Goulart, and
  Wynn]{zheng2020chordal}
Yang Zheng, Giovanni Fantuzzi, Antonis Papachristodoulou, Paul Goulart, and
  Andrew Wynn.
\newblock Chordal decomposition in operator-splitting methods for sparse
  semidefinite programs.
\newblock \emph{Mathematical Programming}, 180\penalty0 (1):\penalty0 489--532,
  2020.

\bibitem[Cai et~al.(2010)Cai, Cand{\`e}s, and Shen]{cai2010singular}
Jian-Feng Cai, Emmanuel~J Cand{\`e}s, and Zuowei Shen.
\newblock A singular value thresholding algorithm for matrix completion.
\newblock \emph{SIAM Journal on optimization}, 20\penalty0 (4):\penalty0
  1956--1982, 2010.

\bibitem[Jain et~al.(2013)Jain, Netrapalli, and Sanghavi]{jain2013low}
Prateek Jain, Praneeth Netrapalli, and Sujay Sanghavi.
\newblock Low-rank matrix completion using alternating minimization.
\newblock In \emph{Proceedings of the forty-fifth annual ACM symposium on
  Theory of computing}, pages 665--674, 2013.

\bibitem[Hardt and Wootters(2014)]{hardt2014fast}
Moritz Hardt and Mary Wootters.
\newblock Fast matrix completion without the condition number.
\newblock In \emph{Conference on learning theory}, pages 638--678. PMLR, 2014.

\bibitem[Yi et~al.(2016)Yi, Park, Chen, and Caramanis]{yi2016fast}
Xinyang Yi, Dohyung Park, Yudong Chen, and Constantine Caramanis.
\newblock Fast algorithms for robust pca via gradient descent.
\newblock \emph{arXiv preprint arXiv:1605.07784}, 2016.

\bibitem[Soltanolkotabi(2019)]{soltanolkotabi2019structured}
Mahdi Soltanolkotabi.
\newblock Structured signal recovery from quadratic measurements: Breaking
  sample complexity barriers via nonconvex optimization.
\newblock \emph{IEEE Transactions on Information Theory}, 65\penalty0
  (4):\penalty0 2374--2400, 2019.

\bibitem[Li et~al.(2018)Li, Ma, and Zhang]{li2018algorithmic}
Yuanzhi Li, Tengyu Ma, and Hongyang Zhang.
\newblock Algorithmic regularization in over-parameterized matrix sensing and
  neural networks with quadratic activations.
\newblock In \emph{Conference On Learning Theory}, pages 2--47. PMLR, 2018.

\bibitem[Nocedal and Wright(2006)]{nocedal2006numerical}
Jorge Nocedal and Stephen Wright.
\newblock \emph{Numerical optimization}.
\newblock Springer Science \& Business Media, 2006.

\bibitem[Negahban and Wainwright(2011)]{negahban2011estimation}
Sahand Negahban and Martin~J Wainwright.
\newblock Estimation of (near) low-rank matrices with noise and
  high-dimensional scaling.
\newblock \emph{The Annals of Statistics}, pages 1069--1097, 2011.

\bibitem[De~Vito et~al.(2005)De~Vito, Caponnetto, and Rosasco]{de2005model}
Ernesto De~Vito, Andrea Caponnetto, and Lorenzo Rosasco.
\newblock Model selection for regularized least-squares algorithm in learning
  theory.
\newblock \emph{Foundations of Computational Mathematics}, 5\penalty0
  (1):\penalty0 59--85, 2005.

\bibitem[Cawley(2006)]{cawley2006leave}
Gavin~C Cawley.
\newblock Leave-one-out cross-validation based model selection criteria for
  weighted ls-svms.
\newblock In \emph{The 2006 IEEE international joint conference on neural
  network proceedings}, pages 1661--1668. IEEE, 2006.

\bibitem[Guo et~al.(2011)Guo, Levina, Michailidis, and Zhu]{guo2011joint}
Jian Guo, Elizaveta Levina, George Michailidis, and Ji~Zhu.
\newblock Joint estimation of multiple graphical models.
\newblock \emph{Biometrika}, 98\penalty0 (1):\penalty0 1--15, 2011.

\bibitem[Good(2006)]{good2006resampling}
Phillip~I Good.
\newblock \emph{Resampling methods}.
\newblock Springer, 2006.

\bibitem[Efron and Tibshirani(1994)]{efron1994introduction}
Bradley Efron and Robert~J Tibshirani.
\newblock \emph{An introduction to the bootstrap}.
\newblock CRC press, 1994.

\bibitem[Cox and Hinkley(1979)]{cox1979theoretical}
David~Roxbee Cox and David~Victor Hinkley.
\newblock \emph{Theoretical statistics}.
\newblock CRC Press, 1979.

\bibitem[Tong et~al.(2021)Tong, Ma, and Chi]{tong2021low}
Tian Tong, Cong Ma, and Yuejie Chi.
\newblock Low-rank matrix recovery with scaled subgradient methods: Fast and
  robust convergence without the condition number.
\newblock \emph{IEEE Transactions on Signal Processing}, 69:\penalty0
  2396--2409, 2021.

\bibitem[Tropp(2015)]{tropp2015introduction}
Joel~A Tropp.
\newblock An introduction to matrix concentration inequalities.
\newblock \emph{arXiv preprint arXiv:1501.01571}, 2015.

\bibitem[Wainwright(2019)]{wainwright2019high}
Martin~J Wainwright.
\newblock \emph{High-dimensional statistics: A non-asymptotic viewpoint},
  volume~48.
\newblock Cambridge University Press, 2019.

\end{thebibliography}

\newpage 
\appendix

\section{Preliminaries for the Noiseless Case}

Recall that the matrix inner product is defined $\inner XY\eqdef\tr{X^{T}Y}$,
and that it induces the Frobenius norm as $\|X\|_{F}=\sqrt{\inner XX}$.
The vectorization $\vect(X)$ is the usual column-stacking operation
that turns an $m\times n$ matrix into a length-$mn$ vector; it preserves
the matrix inner product $\inner XY=\vect(X)^{T}\vect(Y)$ and the
Frobenius norm $\|\vect(X)\|=\|X\|_{F}$. The Kronecker product $\otimes$
is implicitly defined to satisfy $\vect(AXB^{T})=(B\otimes A)\vect X$.

We denote $\lambda_{i}(M)$ and $\sigma_{i}(M)$ as the $i$-th eigenvalue
and singular value of a symmetric matrix $M=M^{T}$, ordered from
the most positive to the most negative. We will often write $\lambda_{\max}(M)$
and $\lambda_{\min}(M)$ to index the most positive and most negative
eigenvalues, and $\sigma_{\max}(M)$ and $\sigma_{\min}(M)$ for the
largest and smallest singular values. 

We denote $\A=[\vect(A_{1}),\ldots,\vect(A_{m})]^{T}$ as the matrix
representation of $\AA$, and note that $\AA(X)=\A\,\vect(X)$. For
fixed $X$ and $M^{\star}$, we can rewrite $f$ in terms of the error
matrix $E$ or its vectorization $\e$ as follows
\begin{equation}
f(X)=\|\AA(E)\|^{2}=\|\A\e\|^{2}\text{ where }E=XX^{T}-M^{\star},\quad\e=\vect(E).\label{eq:edef}
\end{equation}
The gradient satisfies for any matrix $D\in\R^{n\times r}$ 
\begin{equation}
\inner{\nabla f(X)}D=2\left\langle \mathcal{A}\left(XD^{T}+DX^{T}\right),\mathcal{A}\left(E\right)\right\rangle .\label{eq:grad_v}
\end{equation}
Letting $\J$ denote the Jacobian of the vectorized error $\e$ with
respect to $X$ implicitly as the matrix that satisfies
\begin{align}
 & \J\,\vect(Y)=\vect(XY^{T}+YX^{T})\qquad\text{for all }Y\in\R^{n\times r}.\label{eq:Jdef}
\end{align}
allows us to write the gradient exactly as $\vect(\nabla f(X))=2\J^{T}\A^{T}\A\e$.
The noisy versions of (\ref{eq:edef}) and (\ref{eq:grad_v}) are
obvious, though we will defer these to Section~\ref{sec:Prelim_noisy}.

Recall that $\AA$ is assumed to satisfy RIP (Definition~\ref{def:rip})
with parameters $(2r,\delta)$. Here, we set $m=1$ without loss of
generality to avoid carrying the normalizing constant; the resulting
RIP inequality reads 
\begin{equation}
(1-\delta)\|M\|_{F}^{2}\le\|\AA(M)\|^{2}\le(1+\delta)\|M\|_{F}^{2}\text{ for all }M\text{ such that }\rank(M)\le2r,\label{eq:rip}
\end{equation}
where we recall that $0\le\delta<1$. It is easy to see that RIP preserves
the Cauchy--Schwarz identity for all rank-$2r$ matrices $G$ and
$H$:
\begin{equation}
\inner{\AA(G)}{\AA(H)}\le\|\AA(G)\|\|\AA(H)\|\le(1+\delta)\|G\|_{F}\|H\|_{F}.\label{eq:rip-cs}
\end{equation}

As before, we introduce the preconditioner matrix $P$ as 
\begin{align*}
P & \eqdef X^{T}X+\eta I_{r}, & \P & \eqdef P\otimes I_{n}=(X^{T}X+\eta I_{r})\otimes I_{n}
\end{align*}
and define a corresponding $P$-inner product, $P$-norm, and dual
$P$-norm on $\R^{n\times r}$ as follows \begin{subequations}
\begin{align}
\inner XY_{P} & \eqdef\vect(X)^{T}\P\vect(Y)=\inner{XP^{1/2}}{YP^{1/2}}=\tr{XPY^{T}},\label{eq:Pinner}\\
\|X\|_{P} & \eqdef\sqrt{\inner XX_{P}}=\|\P^{1/2}\vect(X)\|=\|XP^{1/2}\|_{F},\label{eq:Pnrm}\\
\|X\|_{P*} & \eqdef\max_{\|Y\|_{P}=1}\inner YX=\|\P^{-1/2}\vect(X)\|=\|XP^{-1/2}\|_{F}.\label{eq:dPnrm}
\end{align}
\end{subequations}Finally, we will sometimes need to factorize the
ground truth $M^{\star}=ZZ^{T}$ in terms of the low-rank factor $Z\in\mathbb{R}^{n\times r^{\star}}$. 

\section{Proof of Lipschitz-like Inequality (Lemma \ref{lem:lipp})}

\label{appendix:A}
In this section we give a proof of Lemma~\ref{lem:lipp},
which is a Lipschitz-like inequality under the $P$-norm. Recall that
we proved linear convergence for PrecGD by lower-bounding the linear
progress $\inner{\nabla f(X)}D$ and upper-bounding $\|D\|_{P}$. 
\begin{lem}
[Lipschitz-like inequality; Lemma~\ref{lem:lipp} restated]
Let $\|D\|_{P}=\|D(X^{T}X+\eta I)^{1/2}\|_{F}$. Then we have 
\[
f(X+D)\le f(X)+\inner{\nabla f(X)}D+\frac{1}{2}L_{P}(X,D)\|D\|_{P}^{2}
\]
where 
\[
L_{P}(X,D)=2(1+\delta)\left[4+\frac{2\|XX^{T}-M^{\star}\|_{F}+4\|D\|_{P}}{\lambda_{\min}(X^{T}X)+\eta}+\left(\frac{\|D\|_{P}}{\lambda_{\min}(X^{T}X)+\eta}\right)^{2}\right]
\]
\end{lem}

\begin{proof}
Recall that $E=XX^{T}-M^{\star}$. We obtain a Taylor expansion of the
quartic polynomial $f$ by directly expanding the quadratic terms
\begin{align*}
f(X+D) & =\|\AA((X+D)(X+D)^{T}-M^{\star})\|^{2}\\
 & =\underbrace{\|\AA(E)\|^{2}+2\langle\AA(E),\AA(XD^{T}+DX^{T})\rangle}_{f(X)+\inner{\nabla f(X)}D}+\underbrace{2\langle\AA(E),\AA(DD^{T})\rangle+\|\AA(XD^{T}+DX^{T})\|^{2}}_{\frac{1}{2}\inner{\nabla^{2}f(X)[D]}D}\\
 & \hfill+\underbrace{2\langle\AA(XD^{T}+DX^{T}),\AA(DD^{T})\rangle}_{\frac{1}{6}\inner{\nabla^{3}f(X)[D,D]}D}+\underbrace{\|\AA(DD^{T})\|^{2}}_{\frac{1}{24}\inner{\nabla^{4}f(X)[D,D,D]}D}.
\end{align*}
We evoke RIP to preserve Cauchy--Schwarz as in \eqref{eq:rip-cs},
and then bound the second, third, and fourth order terms
\begin{align}
T= & 2\langle\AA(E),\AA(DD^{T})\rangle+\|\AA(XD^{T}+DX^{T})\|^{2}+2\langle\AA(XD^{T}+DX^{T}),\AA(DD^{T})\rangle+\|\AA(DD^{T})\|^{2}\nonumber \\
\le & (1+\delta)\left(2\|E\|_{F}\|DD^{T}\|_{F}+\|XD^{T}+DX^{T}\|^{2}+2\|XD^{T}+DX^{T}\|_{F}\|DD^{T}\|_{F}+\|DD^{T}\|_{F}^{2}\right)\nonumber \\
\le & (1+\delta)\left(2\|E\|_{F}\|D\|_{F}^{2}+4\|XD^{T}\|^{2}+4\|XD^{T}\|_{F}\|D\|_{F}^{2}+\|D\|_{F}^{4}\right)\label{eq:Tbnd}
\end{align}
where the third line uses $\|DD^{T}\|_{F}\le\|D\|_{F}^{2}$ and $\|XD^{T}+DX^{T}\|_{F}\le2\|XD^{T}\|_{F}$.
Now, write $d=\vect(D)$ and observe that 
\begin{equation}
\|D\|_{F}^{2}=d^{T}d=(d^{T}\P^{1/2})\P^{-1}(\P^{1/2}d)\leq(d^{T}\P d)\lambda_{\max}(\P^{-1})=\|D\|_{P}^{2}/\lambda_{\min}(\P).\label{eq:dbnd}
\end{equation}
Similarly, we have 
\begin{align}
\|XD^{T}\|_{F} & =\|XP^{-1/2}P^{1/2}D^{T}\|_{F}\le\sigma_{\max}(XP^{-1/2})\|P^{1/2}D^{T}\|_{F}\le\|D\|_{P}.\label{eq:xdbnd}
\end{align}
The final inequality uses $\|P^{1/2}D^{T}\|_{F}=\|DP^{1/2}\|_{F}=\|D\|_{P}$
and that 
\begin{equation}
\sigma_{\max}(XP^{-1/2})=\sigma_{\max}[X(X^{T}X+\eta I)^{-1/2}]=\sigma_{\max}(X)/\sqrt{\sigma_{\max}^{2}(X)+\eta }\le1.\label{eq:XPsigma}
\end{equation}
Substituting (\ref{eq:dbnd}) and (\ref{eq:xdbnd}) into (\ref{eq:Tbnd})
yields 
\begin{align*}
T\leq(1+\delta)\left(2\|E\|_{F}\frac{\|D\|_{P}^{2}}{\lambda_{\min}(\P)}+4\|D\|_{P}^{2}+\frac{4\|D\|_{P}^{3}}{\lambda_{\min}(\P)}+\frac{\|D\|_{P}^{4}}{\lambda_{\min}^{2}(\P)}\right) & =\frac{1}{2}L_{P}(X,D)\|D\|_{P}^{2}
\end{align*}
where we substitute $\lambda_{\min}(\P)=\lambda_{\min}(X^{T}X)+\eta$. 
\end{proof}

\section{Proof of Bounded Gradient (Lemma \ref{lem:bndgrad})}

In this section we prove Lemma~\ref{lem:bndgrad}, which shows that
the gradient measured in the dual $P$-norm $\|\nabla f(X)\|_{P*}$
is controlled by the objective value as $\sqrt{f(X)}$. 
\begin{lem}[Bounded Gradient; Lemma~\ref{lem:bndgrad} restated]
For the search direction $D=\nabla f(X)(X^{T}X+\eta I)^{-1}$, we
have $\|D\|_{P}^{2}=\|\nabla f(X)\|_{P*}^{2}\le16(1+\delta)f(X)$. 
\end{lem}

\begin{proof}
We apply the variation definition of the dual $P$-norm in (\ref{eq:dPnrm})
to the gradient in (\ref{eq:grad_v}) to obtain
\begin{align*}
\|\nabla f(X)\|_{P^{*}} & =\max_{\|Y\|_{P}=1}\inner{\nabla f(X)}Y=\max_{\|Y\|_{P}=1}2\inner{\AA(XY^{T}+YX^{T})}{\AA(E)}\\
 & \overset{\text{(a)}}{\le}2\|\AA(E)\|\max_{\|Y\|_{P}=1}\|\AA(XY^{T}+YX^{T})\|\overset{\text{(b)}}{\le}4\sqrt{(1+\delta)f(X)}\max_{\|Y\|_{P}=1}\|XY^{T}\|_{F}
\end{align*}
Here (a) applies Cauchy--Schwarz; and (b) substitutes $f(X)=\|\AA(E)\|^{2}$
and $\|\AA(M)\|\le\sqrt{1+\delta}\|M\|_{F}$ for rank-$2r$ matrix
$M$ and $\|XY^{T}+YX^{T}\|_{F}\le2\|XY^{T}\|_{F}$. Now, we bound
the final term 
\[
\max_{\|Y\|_{P}=1}\|XY^{T}\|_{F}=\max_{\|YP^{1/2}\|_{F}=1}\|XY^{T}\|_{F}=\max_{\|\tilde{Y}\|_{F}=1}\|XP^{-1/2}\tilde{Y}^{T}\|_{F}=\sigma_{\max}(XP^{-1/2})\le1
\]
where the final inequality uses \eqref{eq:XPsigma}. 
\end{proof}

\section{Proof of Gradient Dominance (Theorem \ref{thm:pl})}

\label{appendix:B}In this section we prove our first main result:
the gradient $\nabla f(X)$ satisfies gradient dominance the $P$-norm.
This is the key insight that allowed us to establish the linear convergence
rate of PrecGD in the main text. The theorem is restated below. 
\begin{thm}[Gradient Dominance; Theorem~\ref{thm:pl} restated]
Let $\min_{X}f(X)=0$ for $M^{\star}\ne0$. Suppose that $X$ satisfies
$f(X)\le\rho^{2}\cdot(1-\delta)\lambda_{r^{\star}}^{2}(M^{\star})$
with radius $\rho>0$ that satisfies $\rho^{2}/(1-\rho^{2})\le(1-\delta^{2})/2$.
Then, we have 
\[
\eta\le C_{\ub}\|XX^{T}-M^{\star}\|_{F}\quad\implies\quad\|\nabla f(X)\|_{P*}^{2}\ge\mu_{P}f(X)
\]
where 
\begin{equation}
\mu_{P}=\left(\sqrt{\frac{1+\delta^{2}}{2}}-\delta\right)^{2}\cdot\min\left\{ \left(1+\frac{C_{\ub}}{\sqrt{2}-1}\right)^{-1},\left(1+3C_{\ub}\sqrt{\frac{(r-r^{\star})}{1-\delta^{2}}}\right)^{-1}\right\}. \label{eq:mueq2}
\end{equation}
\end{thm}

The theorem is a consequence of the following lemma, which shows that
the PL constant $\mu_{P}>0$ is driven in part by the alignment between
the model $XX^{T}$ and the ground truth $M^{\star}$, and in part
in the relationship between $\eta$ and the singular values of $X$.
We defer its proof to Section~\ref{subsec:Proof-of-Gradient} and first use it to prove Theorem~\ref{thm:pl}.
\begin{lem}[Gradient lower bound]
\label{lem:scaled} Let $XX^{T}=U\Lambda U^{T}$ where $\Lambda=\mathrm{diag}(\lambda_{1},\dots,\lambda_{r})$
, $\lambda_{1}\ge\cdots\ge\lambda_{r}\geq 0$, and $U^{T}U=I_{r}$ denote
the usual eigenvalue decomposition. Let $U_{k}$ denote the first
$k$ columns of $U$. Then, we have 
\begin{equation}
\|\nabla f(X)\|_{P^{*}}^{2}\ge\max_{k\in\{1,2,\dots,r\}}\frac{2(\cos\theta_{k}-\delta)^{2}}{1+\eta/\lambda_{k}}\|XX^{T}-M^{\star}\|_{F}^{2}\label{eq:maxcos}
\end{equation}
where each $\theta_{k}$ is defined 
\begin{equation}
\sin\theta_{k}=\frac{\left\Vert \left(I-U_{k}U_{k}^{T}\right)(XX^{T}-M^{\star})\left(I-U_{k}U_{k}^{T}\right)\right\Vert _{F}}{\left\Vert XX^{T}-M^{\star}\right\Vert _{F}}.\label{sin}
\end{equation}
\end{lem}

From Lemma~\ref{lem:scaled}, we see that deriving a PL constant
$\mu_{P}$ requires balancing two goals: 
(1) ensuring that $\cos\theta_{k}$
is large with respect to the RIP constant $\delta$;
(2) ensuring that
$\lambda_{k}(X^{T}X)$ is large with respect to the damping parameter
$\eta$.

As we will soon show, in the case that $k=r$, the corresponding
$\cos\theta_{r}$ is guaranteed to be large with respect to $\delta$,
once $XX^{T}$ converges towards $M^{\star}$. At the same time, we
have by Weyl's inequality 
\begin{equation*}
\lambda_{k}(X^{T}X)=\lambda_{k}(XX^{T})\ge\lambda_{k}(M^{\star})-\|XX^{T}-M^{\star}\|_{F}\text{ for all }k\in\{1,2,\dots,r\}.
\end{equation*}
Therefore, when $k=r^{\star}$ and $XX^{T}$ is close to $M^{\star}$, the corresponding $\lambda_{r^{\star}}(X^{T}X)$
is guaranteed to be large with respect to $\eta$ .
However, in order to use Lemma~\ref{lem:scaled} to derive a PL constant
$\mu_{P}>0$, we actually need $\cos\theta_{k}$ and $\lambda_{k}(X^{T}X)$
to both be large for the \emph{same} value of $k$. It turns out that
when $\eta\gtrsim\|XX^{T}-M^{\star}\|_{F}$, it is possible to prove
this claim using an inductive argument. 

Before we present the complete argument and prove Theorem \ref{thm:pl}, we state one more lemma that will be used in the proof.

\begin{lem}[Basis alignment]
\label{lem:sintheta}Define the $n\times k$ matrix $U_{k}$ in terms
of the first $k$ eigenvectors of $X$ as in Lemma~\ref{lem:scaled}.
Let $Z\in\R^{n\times r^{\star}}$ satisfy $\lambda_{\min}(Z^{T}Z)>0$
and suppose that $\|XX^{T}-ZZ^{T}\|_{F}\le\rho\lambda_{\min}(Z^{T}Z)$
with $\rho\le1/\sqrt{2}$. Then, 
\begin{equation}
\frac{\|Z^{T}(I-U_{k}U_{k}^{T})Z\|_{F}}{\|XX^{T}-ZZ^{T}\|_{F}}\le\frac{1}{\sqrt{2}}\frac{\rho}{\sqrt{1-\rho^{2}}}\quad\text{for all }k\ge r^{\star}.\label{eq:ratio}
\end{equation}
\end{lem}

Essentially, this lemma states that as the rank-$r$ matrix $XX^{T}$ converges
to the rank-$r^{\star}$ matrix $M^{\star}$, the top $r^{\star}$
eigenvectors of $XX^{T}$ must necessarily rotate into alignment with
$M^{\star}$. In fact, this is easily verified
to be sharp by considering the $r=r^{\star}=1$ case; we defer its
proof to Section~\ref{subsec:alignment}.

With Lemma~\ref{lem:scaled} and Lemma~\ref{lem:sintheta}, we are
ready to prove Theorem \ref{thm:pl}. 
\begin{proof}[Proof of Theorem \ref{thm:pl}]
We pick some $\mu$ satisfying $\delta<\mu<1$ and prove that $\frac{\rho^{2}}{1-\rho^{2}}\le1-\mu^{2}$
implies $\|\nabla f(X)\|_{P*}^{2}\ge\mu_{P}f(X)$ where 
\begin{equation}
\mu_{P}=(\mu-\delta)^{2}\cdot\min\left\{ \left(1+\frac{C_{\ub}}{\sqrt{2}-1}\right)^{-1},\left(1+3C_{\ub}\sqrt{\frac{r-r^{\star}}{1-\mu^{2}}}\right)^{-1}\right\} .\label{eq:mup}
\end{equation}
Then, setting $1-\mu^{2}=\frac{1}{2}(1-\delta^{2})$ yields our desired
claim.

To begin, note that the hypothesis $\frac{\rho^{2}}{1-\rho^{2}}\le1-\mu^{2}\le1$
implies $\rho\le1/\sqrt{2}$. Denote $E=XX^{T}-M^{\star}$. We have
\begin{align}
\frac{\|\nabla f(X)\|_{P^{*}}^{2}}{f(X)} & \overset{\text{(a)}}{\ge}\frac{\|\nabla f(X)\|_{P^{*}}^{2}}{(1+\delta)\|E\|_{F}^{2}}\overset{\text{(b)}}{\ge}\frac{2(\cos\theta_{k}-\delta)^{2}}{(1+\delta)(1+\eta/\lambda_{k}(X^{T}X))}\overset{\text{(c)}}{\ge}\frac{(\cos\theta_{k}-\delta)^{2}}{1+\eta/\lambda_{k}(X^{T}X)}\text{ for all }k\ge r^{\star}.\label{eq:gradlb}
\end{align}
Step (a) follows from RIP; Step (b) applies Lemma \ref{lem:scaled};
Step (c) applies $1+\delta\leq2$. Equation (\ref{eq:gradlb}) proves
gradient dominance if we can show that both $\lambda_{k}(X^{T}X)$
and $\cos\theta_{k}$ are large for the same $k$. We begin with $k=r^{\star}$.
Here we have by RIP and by hypothesis
\begin{equation}
(1-\delta)\|XX^{T}-M^{\star}\|_{F}^{2}\leq f(X)\le\rho^{2}\cdot(1-\delta)\lambda_{\min}^{2}(Z^{T}Z),\label{eq:ripEE}
\end{equation}
which by Weyl's inequality yields
\begin{align*}
\lambda_{r^{\star}}(X^{T}X) & =\lambda_{r^{\star}}(XX^{T})\ge\lambda_{r^{\star}}(M^{\star})-\|XX^{T}-M^{\star}\|_{F}\ge(1-\rho)\lambda_{r^{\star}}(M^{\star}).
\end{align*}
This, combined with (\ref{eq:ripEE}) and our hypothesis $\eta\leq C_{\ub}\|XX^{T}-ZZ^{T}\|_{F}$
and $\rho\le1/\sqrt{2}$ gives
\begin{equation}
\frac{\eta}{\lambda_{r^{\star}}(X^{T}X)}\le\frac{\rho C_{\ub}\lambda_{r^{\star}}(M^{\star})}{(1-\rho)\lambda_{r^{\star}}(M^{\star})}=\frac{\rho C_{\ub}}{1-\rho}\le\frac{C_{\ub}}{\sqrt{2}-1},\label{eq:gradub}
\end{equation}
which shows that $\lambda_{r^{\star}}(X^{T}X)$ is large. If $\cos\theta_{k}\ge\mu$
is also large, then substituting (\ref{eq:gradub}) into (\ref{eq:gradlb})
yields gradient dominance 
\[
\frac{\|\nabla f(X)\|_{P^{*}}^{2}}{f(X)}\ge(\mu-\delta)^{2}\left(1+\frac{C_{\ub}}{\sqrt{2}-1}\right)^{-1},
\]
and this yields the first term in (\ref{eq:mup}). If $\cos\theta_{k}<\mu$
is actually small, then $\sin^{2}\theta_{k}>1-\mu^{2}$ is large.
We will show that this lower bound on $\sin\theta_k$ actually implies that $\lambda_{k+1}(X^TX)$ will be large. 

To see this,
let us write $XX^{T}=U_{k}\Lambda_{k}U_{k}^{T}+R$ where
the $n\times k$ matrix of eigenvectors $U_{k}$ is defined as in
Lemma~\ref{lem:scaled}, $\Lambda_{k}$ is the corresponding
$k\times k$ diagonal matrix of eigenvalues, and $U_k^TR=0$. Denote $\Pi_k = I-U_kU_k^T$ and note that 
\[
\|\Pi_{k}(X X^{T}-M^{\star}) \Pi_{k}\|_F=\|\Pi_{k} X X^{T} \Pi_{k}-\Pi_{k} M^{\star} \Pi_{k}\|_F=\|R-\Pi_{k} M^{\star} \Pi_{k}\|_F.
\]
By the subaddivity of the norm $
\|R-\Pi_{k} M^{\star} \Pi_{k}\|_{F} \leq \|R\|_{F}+\|\Pi_{k} M^{\star} \Pi_{k}\|_{F}.$
Dividing both sides by $\|E\|_F$ yields
\[
\sin \theta_k = \frac{\|R-\Pi_{k} M^{\star} \Pi_{k}\|_{F}}{\|E\|_F} \leq \frac{\|\Pi_{k} M^{\star} \Pi_{k}\|_{F}}{\|E\|_{F}}+\frac{\|R\|_{F}}{\|E\|_{F}}.
\]
Since $\rho\leq 1/\sqrt{2}$ by assumption, Lemma \ref{lem:sintheta} yields
\[
\frac{\|\Pi_{k} M^{\star} \Pi_{k}\|_{F}}{\|E\|_{F}} \leq \frac{1}{\sqrt{2}} \frac{\rho}{\sqrt{1-\rho^{2}}} \leq \rho.
\]
In addition, 
\[
\|R\|_F \leq \|R\| \cdot \sqrt{\mathrm{rank}(R)} = \lambda_{k+1}(XX^T)\cdot \sqrt{r-k}.
\]
Combining the two inequalities above we get 
\[
\sqrt{1-\mu^2} \leq \sin \theta_{k} \leq \frac{1}{\sqrt{2}} \frac{\rho}{\sqrt{1-\rho^{2}}}+\sqrt{r-k} \cdot \frac{\lambda_{k+1}\left(X^{T} X\right)}{\|E\|_{F}}.
\]
Rearranging, we get
\[
\frac{\lambda_{k+1}\left(X^{T} X\right)}{\|E\|_{F}} \geq \frac{1}{\sqrt{r-k}} \left(\sqrt{1-\mu^2} - \frac{1}{\sqrt{2}} \frac{\rho}{\sqrt{1-\rho^{2}}}\right) \geq \left(1-\frac{1}{\sqrt{2}}\right) \sqrt{\frac{1-\mu^2}{r-k}}.
\]
Note that the last inequality above follows from the assumption that $\frac{\rho^2}{1-\rho^{2}}\leq 1-\mu^2$. Now 
substituting $\eta\leq C_{\ub}\|XX^{T}-M^{\star}\|_{F}$ and $r-k\leq r-r^{\star}$ and noting that $\left(1-\frac{1}{\sqrt{2}}\right)\leq 1/3$ we get
\begin{equation}
\frac{\eta}{\lambda_{k+1}(X^{T}X)}\le C_{\ub}\frac{\|XX^{T}-M^{\star}\|_{F}}{\lambda_{k+1}(X^{T}X)}\le 3C_{\ub}\sqrt{\frac{r-k}{1-\mu^{2}}}\le 3C_{\ub}\sqrt{\frac{r-r^{\star}}{1-\mu^{2}}},\label{sigup}
\end{equation}
which shows that $\lambda_{k+1}(X^{T}X)$ is large.

If $\cos\theta_{k+1}\ge\mu$
is also large, then substituting (\ref{sigup}) into (\ref{eq:gradlb})
yields gradient dominance
\begin{equation}
\frac{\|\nabla f(X)\|_{P^{*}}^{2}}{f(X)}\geq\frac{(\cos\theta_{k+1}-\delta)^{2}}{1+\eta/\lambda_{k+1}^{2}(X)}\ge(\mu-\delta)^{2}\left(1+3C_{\ub}\sqrt{\frac{r-r^{\star}}{1-\mu^{2}}}\right)^{-1},\label{eq:gdt2}
\end{equation}
and this yields the second term in (\ref{eq:mup}) so we are done. If $\cos\theta_{k+1}<\mu$ then we can simply repeat the argument above to show that $\lambda_{k+1}(X^TX)$ is large. We can repeat this process until $k+1=r$. At this point, we have
\[
\cos^{2}\theta_{r}=1-\sin^{2}\theta_{r}\ge1-\frac{1}{2}\frac{\rho^{2}}{1-\rho^{2}}\ge\mu^{2}
\]
where we used our hypothesis $1-\mu^{2}\ge\frac{\rho^{2}}{1-\rho^{2}}\ge\frac{1}{2}\frac{\rho^{2}}{1-\rho^{2}}$,
and substituting (\ref{sigup}) into (\ref{eq:gradlb}) again yields
gradient dominance in (\ref{eq:gdt2}). 
\end{proof}

\subsection{\label{subsec:Proof-of-Gradient}Proof of Gradient Lower Bound (Lemma
\ref{lem:scaled})}

In this section we prove Lemma~\ref{lem:scaled}, where we prove
gradient dominance $\|\nabla f(X)\|_{P^{*}}^{2}\ge\mu_{P}f(X)$ with
a PL constant $\mu_{P}$ that is proportional to $\cos\theta_{k}-\delta$
and to $\lambda_{k}(X^{T}X)/\eta$. We first prove the following result which will be useful in the proof of Lemma~\ref{lem:scaled}.

\begin{lem}
\label{lem:reform}Let $\AA$ satisfy RIP with parameters $(\zeta,\delta)$,
where $\zeta=\rank([X,Z])$. Then, we have 
\begin{equation}
\label{eq:max_y}
\|\nabla f(X)\|_{P*}\ge\max_{\|Y\|_{P}\le1}\langle XY^{T}+YX^{T},E\rangle-\delta\|XY^{T}+YX^{T}\|_{F}\|E\|_{F}
\end{equation}
\end{lem}
\begin{proof}
Let $Y$ maximize the right-hand side of \eqref{eq:max_y} and let $W$ be the matrix corrresponding to the orthogonal projection onto $\mathrm{range}(X)+\mathrm{range}(Y)$. Set $\tilde{Y} = WY$, then 
\begin{align*}
\langle X \tilde{Y}^{T}+\tilde{Y} X^{T}, E\rangle=\langle X Y^{T}, E W\rangle+\langle Y X^{T}, W E\rangle=\langle X Y^{T}+Y X^{T}, E\rangle.    
\end{align*}
On the other hand, we have 
\begin{align*}
    \|X \tilde{Y}^{T}+\tilde{Y} X^{T}\|_{F}=\|W\left(X Y^{T}+Y X^{T}\right) W\|_{F} \leq \|X Y^{T}+Y X^{T}\|_{F}
\end{align*}
and
\[
\|\tilde{Y}\|_{P} = \|WYP^{1/2}\|_F \leq \|YP^{1/2}\|_F = \|Y\|_P.
\]
This means that $\tilde{Y}$ is feasible and makes the right-hand side at least as large as $Y$. Since $Y$ is the maximizer by definition, we conclude that $\tilde{Y}$ also maximizes the right-hand side of \eqref{eq:max_y}.

By definition, $\operatorname{range}(\tilde{Y}) \subset \operatorname{range}(X)+\operatorname{range}(Z)$, so $(2r,\delta)$-RIP implies 
\[
|\langle A (X \tilde{Y}^{T}+\tilde{Y} X^{T}), A (E)\rangle-\langle X \tilde{Y}^{T}+\tilde{Y} X^{T}, E\rangle| \leq \delta \|X \tilde{Y}^{T}+\tilde{Y} X^{T}\|_F\| E \|_F.
\]
Now we have 
\begin{align*}
   \|\nabla f(X)\|_{P*} &=  \max_{\|Y\|_{P}\le1}\langle \AA(XY^{T}+YX^{T}),\AA(E)\rangle \\
   &\geq \langle \AA(X\tilde{Y}^{T}+\tilde{Y}X^{T}),\AA(E)\rangle \\
   &\geq \langle X\tilde{Y}^{T}+\tilde{Y}X^{T},E\rangle-\delta \|X \tilde{Y}^{T}+\tilde{Y} X^{T}\|_F\| E \|_F \\
   &= \max_{\|Y\|_{P}\le1}\langle XY^{T}+YX^{T},E\rangle-\delta\|XY^{T}+YX^{T}\|_{F}\|E\|_{F}.
\end{align*}
This completes the proof. 
\end{proof}

\begin{proof}[Proof of Lemma \ref{lem:scaled}]
\global\long\def\JU{\mathbf{U}}%
\global\long\def\JS{\mathbf{\Sigma}}%
\global\long\def\JV{\mathbf{V}}%
Let $X=\sum_{i=1}^{r}\sigma_{i}u_{i}v_{i}^{T}$ with $\|u_{i}\|=\|v_{i}\|=1$
and $\sigma_{1}\ge\cdots\ge\sigma_{r}$ denote the usual singular
value decomposition. Observe that the preconditioned Jacobian $\J\P^{-1/2}$
satisfies
\[
\J\P^{-1/2}\vect(Y)=\vect(XP^{-1/2}Y^{T}+YP^{-1/2}X^{T})=\vect\left(\sum_{i=1}^{r}\frac{u_{i}y_{i}^{T}+y_{i}u_{i}^{T}}{\sqrt{1+\eta/\sigma_{i}^{2}}}\right)
\]
where $y_{i}=Yv_{i}$. This motivates the following family of singular
value decompositions 
\begin{align}
\JU_{k}\JS_{k}\JV_{k}^{T}\vect(Y) & =\vect\left(\sum_{i=1}^{k}\frac{u_{i}y_{i}^{T}+y_{i}u_{i}^{T}}{\sqrt{1+\eta/\sigma_{i}^{2}}}\right)\text{ for all }k\in\{1,2,\dots,r\}, & \J\P^{-1/2} & =\JU_{r}\JS_{r}\JV_{r}^{T}.\label{eq:Jsvd}
\end{align}
Here, the $n^{2}\times\zeta_{k}$ matrix $\JU_{k}$ and the $nr\times\zeta_{k}$
matrix $\JV_{k}$ have orthonormal columns, and the rank can be verified
as $\zeta_{k}=nk-k(k-1)/2<nr\le n^{2}$. Now, we rewrite Lemma \ref{lem:reform}
by vectorizing $y=\vect(Y)$ and writing
\begin{align*}
\|\nabla f(X)\|_{P*} & \ge\max_{\|\P^{1/2}y\|\le1}\left(\frac{\e^{T}\J y}{\|\e\|\|\J y\|}-\delta\right)\|\e\|\|\J y\|\overset{\text{(a)}}{=}\max_{\|y'\|\le1}\left(\frac{\e^{T}\J\P^{-1/2}y}{\|\e\|\|\J\P^{-1/2}y\|}-\delta\right)\|\e\|\|\J\P^{-1/2}y\|\\
 & \overset{\text{(b)}}{=}\max_{\|y'\|\le1}\left(\frac{\e^{T}\JU_{r}\JS_{r}\JV_{r}^{T}y}{\|\e\|\|\JU_{r}\JS_{r}\JV_{r}^{T}y\|}-\delta\right)\|\e\|\|\JU_{r}\JS_{r}\JV_{r}^{T}y\|\\
 & \overset{\text{(c)}}{\ge}\left(\frac{\e^{T}\JU_{k}\JU_{k}^{T}\e}{\|\e\|\|\JU_{k}^{T}\e\|}-\delta\right)\|\e\|\frac{\|\JU_{k}^{T}\e\|}{\|\JS_{k}^{-1}\JU_{k}^{T}\e\|}\overset{\text{(d)}}{\ge}\left(\frac{\|\JU_{k}^{T}\e\|}{\|\e\|}-\delta\right)\|\e\|\lambda_{\min}(\JS_{k}).\\
\end{align*}
Step (a) makes a change of variables $y\gets\P^{1/2}y$; Step (b)
substitutes (\ref{eq:Jsvd}); Step (c) substitutes the heuristic choice
$y=d/\|d\|$ where $d=\JV_{k}\JS_{k}^{-1}\JU_{k}^{T}\e$; Step (d)
notes that $\e^{T}\JU_{k}\JU_{k}^{T}\e=\|\JU_{k}^{T}\e\|^{2}$ and
that $\|\JS_{k}^{-1}\JU_{k}^{T}\e\|\le\|\JU_{k}^{T}\e\|\cdot\lambda_{\max}(\JS_{k}^{-1})=\|\JU_{k}^{T}\e\|/\lambda_{\min}(\JS_{k})$.
Finally, we can mechanically verify from (\ref{eq:Jsvd}) that
\[
\cos^{2}\theta_{k}\eqdef\frac{\|\JU_{k}^{T}\e\|^{2}}{\|\e\|^{2}}=1-\frac{\|(I-\JU_{k}^{T}\JU_{k}^{T})\e\|^{2}}{\|\e\|^{2}}=1-\frac{\|(I-U_{k}U_{k}^{T})E(I-U_{k}U_{k}^{T})\|_{F}^{2}}{\|E\|_{F}^{2}}
\]
where $U_{k}=[u_{1},\dots,u_{k}]$, and that 
\[
\lambda_{\min}^{2}(\JS_{k})=\min_{\|y_{k}\|=1}\left\Vert \frac{u_{k}y_{k}^{T}+y_{k}u_{k}^{T}}{\sqrt{1+\eta/\sigma_{k}^{2}}}\right\Vert _{F}^{2}=\min_{\|y_{k}\|=1}\frac{2\|u_{k}\|^{2}\|y_{k}\|^{2}+2(u_{k}^{T}y_{k})^{2}}{1+\eta/\sigma_{k}^{2}}=\frac{2}{1+\eta/\sigma_{k}^{2}}.
\]
\end{proof}

\subsection{\label{subsec:alignment}Proof of Basis Alignment (Lemma \ref{lem:sintheta})}

Before we prove this lemma, we make two observations that simplifies
the proof. First, even though our goal is to prove the inequality
(\ref{eq:ratio}) for all $k\geq r^{*}$, it actually suffices to
consider the case $k=r^{*}$. This is because the numerator $\|Z^{T}(I-U_{k}U_{k}^{T})Z\|_{F}$
decreases monotonically as $k$ increases. Indeed, for any $k\ge r^{\star}$,
define $VV^{T}$ as below 
\[
I-U_{k}U_{k}^{T}=I-U_{r^{\star}}U_{r^{\star}}^{T}-VV^{T}=(I-U_{r^{\star}}U_{r^{\star}}^{T})(I-VV^{T})=(I-VV^{T})(I-U_{r^{\star}}U_{r^{\star}}^{T}).
\]
Then, we have 
\begin{align*}
\|Z^{T}(I-U_{k}U_{k}^{T})Z\|_{F} & =\|(I-U_{k}U_{k}^{T})ZZ^{T}(I-U_{k}U_{k}^{T})\|_{F}\\
 & =\|(I-VV^{T})(I-U_{r^{\star}}U_{r^{\star}}^{T})ZZ^{T}(I-U_{r^{\star}}U_{r^{\star}}^{T})(I-VV^{T})\|_{F}\\
 & \le\|(I-U_{r^{\star}}U_{r^{\star}}^{T})ZZ^{T}(I-U_{r^{\star}}U_{r^{\star}}^{T})\|_{F}.
\end{align*}
Second, due to the rotational invariance of this problem, we can assume
without loss of generality that $X,Z$ are of the form 
\begin{equation}
X=\begin{bmatrix}X_{1} & 0\\
0 & X_{2}
\end{bmatrix},\;Z=\begin{bmatrix}Z_{1}\\
Z_{2}
\end{bmatrix}.\label{invari}
\end{equation}
where $X_{1}\in\mathbb{R}^{k\times k},\;Z_{1}\in\R^{k\times r^{\star}}$
and $\sigma_{\min}(X_{1})\ge\sigma_{\max}(X_{2})$. (Concretely, we
compute the singular value decomposition $X=USV^{T}$ with $U\in\R^{n\times n}$
and $V\in\R^{r\times r}$, and then set $X\gets U^{T}XV$ and $Z\gets U^{T}Z$.)
We first need to show that as $XX^{T}$ approaches $ZZ^{T}$, the
dominant directions of $X$ must align with $Z$ in a way as to make
the $Z_{2}$ portion of $Z$ go to zero.
\begin{lem}
Suppose that $X,Z$ are in the form in (\ref{invari}), and $k\geq r^{\star}$.
If $\|XX^{T}-ZZ^{T}\|_{F}\leq\rho\lambda_{\min}(Z^{T}Z)$ and $\rho^{2}<1/2$,
then $\lambda_{\min}(Z_{1}^{T}Z_{1})\ge\lambda_{\max}(Z_{2}^{T}Z_{2})$.
\label{lem:ab} 
\end{lem}

\begin{proof}
Denote $\gamma=\lambda_{\min}(Z_{1}^{T}Z_{1})$ and $\beta=\lambda_{\max}(Z_{2}^{T}Z_{2})$.
We will assume $\gamma<\beta$ and prove that $\rho^{2}\geq1/2$,
which contradicts our hypothesis. The claim is invariant to scaling
of $X$ and $Z$, so we assume without loss of generality that $\lambda_{\min}(Z^{T}Z)=1$.
Our radius hypothesis then reads 
\begin{align*}
\|XX^{T}-ZZ^{T}\|_{F}^{2} & =\left\Vert \begin{bmatrix}X_{1}X_{1}^{T}-Z_{1}Z_{1}^{T} & -Z_{1}Z_{2}^{T}\\
-Z_{2}Z_{1}^{T} & X_{2}X_{2}^{T}-Z_{2}Z_{2}^{T}
\end{bmatrix}\right\Vert _{F}^{2}\\
 & =\|X_{1}X_{1}^{T}-Z_{1}Z_{1}^{T}\|_{F}^{2}+2\langle Z_{1}^{T}Z_{1},Z_{2}^{T}Z_{2}\rangle+\|X_{2}X_{2}^{T}-Z_{2}Z_{2}^{T}\|_{F}^{2}\leq\rho^{2}.
\end{align*}
Now, we optimize over $X_{1}$ and $X_{2}$ to minimize the left-hand
side. Recall by construction in (\ref{invari}) we restricted $\sigma_{\min}(X_{1})\ge\sigma_{\max}(X_{2})$.
Accordingly, we consider
\begin{equation}
\min_{X_{1},X_{2}}\left\{ \|X_{1}X_{1}^{T}-Z_{1}Z_{1}^{T}\|_{F}^{2}+\|X_{2}X_{2}^{T}-Z_{2}Z_{2}^{T}\|_{F}^{2}:\lambda_{\min}(X_{1}X_{1}^{T})\ge\lambda_{\max}(X_{2}X_{2}^{T})\right\} .\label{eq:s_0}
\end{equation}
We relax $X_{1}X_{1}^{T}$ and $X_{2}X_{2}^{T}$ into positive semidefinite
matrices
\begin{align}
\text{\eqref{eq:s_0}}\geq & \min_{S_{1}\succeq0,S_{2}\succeq0}\{\|S_{1}-Z_{1}Z_{1}^{T}\|_{F}^{2}+\|S_{2}-Z_{2}Z_{2}^{T}\|_{F}^{2}:\lambda_{\min}(S_{1})\ge\lambda_{\max}(S_{2})\}\label{eq:s_01}
\end{align}
The equation above is invariant to a change of basis for both $S_{1}$
and $S_{2}$, so we change the basis of $S_{1}$ and $S_{2}$ into
the eigenbases of $Z_{1}Z_{1}^{T}$ and $Z_{2}Z_{2}^{T}$ to yield
\begin{align}
\text{\eqref{eq:s_01}}= & \min_{s_{1}\ge0,s_{2}\ge0}\{\|s_{1}-\lambda(Z_{1}Z_{1}^{T})\|^{2}+\|s_{2}-\lambda(Z_{2}Z_{2}^{T})\|^{2}:\min(s_{1})\ge\max(s_{2})\}\label{s_1}
\end{align}
where $\lambda(Z_{1}Z_{1}^{T})\ge0$ and $\lambda(Z_{2}Z_{2}^{T})\ge0$
are the vector of eigenvalues. We lower-bound (\ref{s_1}) by dropping
all the terms in the sum of squares except the one associated with
$\lambda_{\min}(Z_{1}^{T}Z_{1})$ and $\lambda_{\max}(Z_{2}Z_{2}^{T})$
to obtain 
\begin{align}
(\ref{s_1}) & \geq\min_{d_{1},d_{2}\in\R_{+}}\{[d_{1}-\lambda_{\min}(Z_{1}^{T}Z_{1})]^{2}+[d_{2}-\lambda_{\max}(Z_{2}Z_{2}^{T})]^{2}:d_{1}\ge d_{2}\}\\
 & =\min_{d_{1},d_{2}\in\R_{+}}\{[d_{1}-\gamma]^{2}+[d_{2}-\beta]^{2}:d_{1}\ge d_{2}\}=(\gamma-\beta)^{2}/2,\label{optalpha}
\end{align}
where we use the fact that $\gamma<\beta$ to argue that $d_{1}=d_{2}$
at optimality. Now we have 
\begin{align*}
\rho^{2} & \ge\|X_{1}X_{1}^{T}-Z_{1}Z_{1}^{T}\|_{F}^{2}+\|X_{2}X_{2}^{T}-Z_{2}Z_{2}^{T}\|_{F}^{2}+2\langle Z_{1}^{T}Z_{1},Z_{2}^{T}Z_{2}\rangle\\
 & \ge\|X_{1}X_{1}^{T}-Z_{1}Z_{1}^{T}\|_{F}^{2}+\|X_{2}X_{2}^{T}-Z_{2}Z_{2}^{T}\|_{F}^{2}+2\lambda_{\min}(Z_{1}^{T}Z_{1})\lambda_{\max}(Z_{2}^{T}Z_{2})\\
 & \ge\min_{d_{1},d_{2}\in\R_{+}}\{[d_{1}-\gamma]^{2}+[d_{2}-\beta]^{2}:d_{1}\ge d_{2}\}+2\gamma\beta\\
 & \geq\frac{(\gamma-\beta)^{2}}{2}+2\gamma\beta=\frac{1}{2}(\gamma+\beta)^{2}.
\end{align*}
Finally, note that 
\begin{align*}
\gamma+\beta=\lambda_{\min}(Z_{1}^{T}Z_{1})+\lambda_{\max}(Z_{2}^{T}Z_{2})\geq\lambda_{\min}(Z_{1}^{T}Z_{1}+Z_{2}^{T}Z_{2})=\lambda_{\min}(Z^{T}Z)=1.
\end{align*}
Therefore, we have $\rho^{2}\geq1/2$, a contradiction. This completes
the proof. 
\end{proof}
Now we are ready to prove Lemma \ref{lem:sintheta}. 
\begin{proof}
As before, assume with out loss of generality that $X,Z$ are of the
form (\ref{invari}). From the proof of Lemma \ref{lem:ab} we already
know 
\begin{align*}
\|XX^{T}-ZZ^{T}\|_{F}^{2}=\|X_{1}X_{1}^{T}-Z_{1}Z_{1}^{T}\|_{F}^{2}+2\langle Z_{1}^{T}Z_{1},Z_{2}^{T}Z_{2}\rangle+\|X_{2}X_{2}^{T}-Z_{2}Z_{2}^{T}\|_{F}^{2}.
\end{align*}
Moreoever, we can compute 
\begin{equation}
\|Z^{T}(I-U_{k}U_{k}^{T})Z\|_{F}=\left\Vert \begin{bmatrix}Z_{1}\\
Z_{2}
\end{bmatrix}^{T}\left(I-\begin{bmatrix}I_{k} & 0\\
0 & 0
\end{bmatrix}\right)\begin{bmatrix}Z_{1}\\
Z_{2}
\end{bmatrix}\right\Vert _{F}=\|Z_{2}^{T}Z_{2}\|_{F}=\|Z_{2}Z_{2}^{T}\|_{F}.\label{eq:numer}
\end{equation}
We will show that in the neighborhood $\|XX^{T}-ZZ^{T}\|\leq\rho\lambda_{\min}(Z^{T}Z)$
that 
\begin{equation}
\rho\le1/\sqrt{2}\implies\sin\phi\eqdef\|(I-U_{k}U_{k}^{T})Z\|_{F}/\sigma_{k}(Z)=\|Z_{2}\|_{F}/\sigma_{r^{\star}}(Z)\le\rho.\label{eq:sinbnd}
\end{equation}
Then we obtain 
\begin{align}
\frac{\|Z_{2}Z_{2}^{T}\|_{F}^{2}}{\|XX^{T}-ZZ^{T}\|^{2}} & \overset{\text{(a)}}{\le}\frac{\|Z_{2}\|_{F}^{4}}{2\langle Z_{1}^{T}Z_{1},Z_{2}^{T}Z_{2}\rangle}\overset{\text{(b)}}{\le}\frac{\|Z_{2}\|_{F}^{4}}{2\lambda_{\min}(Z_{1}^{T}Z_{1})\|Z_{2}\|_{F}^{2}}\nonumber \\
 & \overset{\text{(c)}}{\le}\frac{\|Z_{2}\|_{F}^{2}}{2[\lambda_{\min}(Z^{T}Z)-\|Z_{2}\|_{F}^{2}]}=\frac{\sin^{2}\phi}{2[1-\sin^{2}\phi]}\\
 & \leq\frac{1}{2}\frac{\rho^{2}}{1-\rho^{2}}.\label{eq:mainbnd}
\end{align}
Step (a) bounds the numerator as $\|Z_{2}Z_{2}^{T}\|_{F}\leq\|Z_{2}\|_{F}^{2}$
and uses the fact that the denominator is greater than $2\langle Z_{1}^{T}Z_{1},Z_{2}^{T}Z_{2}\rangle$.
Step (b) follows from the inequality $\langle Z_{1}^{T}Z_{1},Z_{2}^{T}Z_{2}\rangle\geq\lambda_{\min}(Z_{1}^{T}Z_{1})\|Z_{2}Z_{2}^{T}\|_{F}$.
Finally, step (c) bounds the minimum eigenvalue of $Z_{1}^{T}Z_{1}$
by noting that 
\begin{align}
\lambda_{\min}(Z_{1}^{T}Z_{1}) & =\lambda_{\min}(Z_{1}^{T}Z_{1}+Z_{2}^{T}Z_{2}-Z_{2}^{T}Z_{2})\nonumber \\
 & \ge\lambda_{\min}(Z_{1}^{T}Z_{1}+Z_{2}^{T}Z_{2})-\lambda_{\max}(Z_{2}^{T}Z_{2})\nonumber \\
 & \geq\lambda_{\min}(Z^{T}Z)-\|Z_{2}\|_{F}^{2},\label{eq:eigen_pert}
\end{align}
where the last line bounds the operator norm of $Z_{2}$ with the
Frobenius norm.

To prove (\ref{eq:sinbnd}), we know from Lemma \ref{lem:ab} that
$\rho\leq1/\sqrt{2}$ implies that $\lambda_{\min}(Z_{1}^{T}Z_{1})\ge\lambda_{\max}(Z_{2}^{T}Z_{2})$.
This implies $\lambda_{\min}(Z_{1}^{T}Z_{1})\ge\frac{1}{2}\lambda_{\min}(Z^{T}Z)$,
since
\[
2\lambda_{\min}(Z_{1}^{T}Z_{1})\ge\lambda_{\min}(Z_{1}^{T}Z_{1})+\lambda_{\max}(Z_{2}^{T}Z_{2})\ge\lambda_{\min}(Z^{T}Z)
\]
This implies the following
\begin{align*}
\|XX^{T}-ZZ^{T}\|_{F}^{2} & =\|X_{1}X_{1}^{T}-Z_{1}Z_{1}^{T}\|_{F}^{2}+2\langle Z_{1}^{T}Z_{1},Z_{2}^{T}Z_{2}\rangle+\|X_{2}X_{2}^{T}-Z_{2}Z_{2}^{T}\|_{F}^{2}\\
 & \geq2\langle Z_{1}^{T}Z_{1},Z_{2}^{T}Z_{2}\rangle\geq2\lambda_{\min}(Z_{1}^{T}Z_{1})\|Z\|_{F}^{2}\geq\lambda_{\min}(Z^{T}Z)\|Z\|_{F}^{2}
\end{align*}
and we have therefore
\[
\rho^{2}\lambda_{\min}^{2}(Z^{T}Z)\ge\|XX^{T}-ZZ^{T}\|_{F}^{2}\ge\lambda_{\min}(Z^{T}Z)\|Z\|_{F}^{2}\ge\lambda_{\min}(Z^{T}Z)\|Z_{2}\|_{F}^{2}
\]
which this proves $\sin^{2}\phi=\|Z_{2}\|_{F}^{2}/\lambda_{\min}(Z^{T}Z)\le\rho^{2}$
as desired. 
\end{proof}

\section{\label{sec:Prelim_noisy}Preliminaries for the Noisy Case}

\subsection{Notations}

In the following sections, we extend our proofs to the noisy setting.
As before, we denote by $M^{\star}=ZZ^{T}\in\R^{n\times n}$ our ground
truth. Our measurements are of the form $y=\AA(ZZ^{T})+\epsilon\in\R^{m}$.
We make the standard assumption that the noise vector $\epsilon\in\mathbb{R}^{m}$
has sub-Gaussian entries with zero mean and variance $\sigma^{2}=\frac{1}{m}\sum_{i=1}^{m}\mathbb{E}[\epsilon_{i}^{2}]$.

In this case, the objective function can be written as 
\[
f(X)=\frac{1}{m}\|\AA(XX^{T})-y\|^{2}=f_{c}(X)+\frac{1}{m}\|\epsilon\|^{2}-\frac{2}{m}\langle\AA(XX^{T}-M^{\star}),\epsilon\rangle,
\]
where $f_{c}(X)=\frac{1}{m}\|\AA(XX^{T}-M^{\star})\|^{2}$ is the
objective function with clean measurements that are not corrupted
with noise. Note that compared to the noiseless case, we have rescaled
our objective by a factor of $1/m$ to emphasize the number of measurements
$m$.

Moreover, we say that an event $\mathcal{E}$ happens with overwhelming
or high probability, if its probability of occurrence is at least
$1-cn^{-c'}$, for some $0<c,c'<\infty$. Moreover, to streamline
the presentation, we omit the statement ``with high or overwhelming
probabily'' if it is implied by the context.

We make a few simplifications on notations. As before, we will use
$\alpha$ to denote the step-size and $D$ to denote the local search
direction. We will use lower case letters $x$ and $d$ to refer to
$\vect{(X)}$ and $\vect{(D)}$ respectively.

Similarly, we will write $f(x)\in\R^{nr}$ and $\nabla f(x)\in R^{nr}$
as the vectorized versions of $f(X)$ and its gradient. This notation
is also used for $f_{c}(X)$. As before, we define $P=X^{T}X+\eta I_{r}$
and $\P=(X^{T}X+\eta I_{r})\otimes I_{n}.$ For the vectorized version
of the gradient, we simply define its $P$-norm (and $P^{*}$-norm)
to be the same as the matrix version, that is, 
\[
\|\nabla f(x)\|_{P}=\|\nabla f(X)\|_{P},\qquad\|\nabla f(x)\|_{P^{*}}=\|\nabla f(X)\|_{P^{*}}.
\]

We drop the iteration index $k$ from our subsequent analysis, and
refer to $x_{k+1}$ and $x_{k}$ as $\tilde{x}$ and $x$, respectively.
Thus, with noisy measurements, the iterations of PrecGD take the form
\[
X_{k+1}=X_{k}-\alpha\nabla f(X_{k})(X_{k}^{T}X_{k})^{-1}.
\]
The \textit{vectorized} version of the gradient update above can be
written as $\tilde{x}=x-\alpha d$, where 
\begin{equation}
\begin{aligned}d & =\vect{(\nabla f(X)P^{-1})}=\vect\left(f_{c}(X)+\frac{1}{m}\|\epsilon\|^{2}-\frac{2}{m}\langle\AA(XX^{T}-M^{\star}),\epsilon\rangle\right)\\
 & =\P^{-1}\nabla f_{c}(x)-\frac{2}{m}\P^{-1}\left(I_{r}\otimes\sum_{i=1}^{m}\epsilon_{i}A_{i}\right)x.
\end{aligned}
\label{eq:d_gd}
\end{equation}
Inspired by the variational representation of the Frobenius norm,
for any matrix $H\in\R^{n\times n}$ we define its \textit{restricted
Frobenius norm} as 
\begin{equation}
\|H\|_{F,r}=\argmax_{Y\in S_{n}^{+},\mathrm{rank}(Y)\leq r}\inner HY,\label{eq:frob}
\end{equation}
where $S_{n}^{+}$ is the set of $n\times n$ positive semidefinite
matrices. It is easy to verify that $\|H\|_{F}=\|H\|_{F,n}$ and $\|H\|_{F,r}=\sqrt{\sum_{i=1}^{r}\sigma_{i}(H)^{2}}$.

For any two real numbers $a,b\in R$, we say that $a\asymp b$ if
there exists some constant $C_{1},C_{2}$ such that $C_{1}b\leq a\leq C_{2}b$.
Through out the section we will use one symbol $C$ to denote constants
that might differ.

Finally, we also recall that $\mu_{P}$, which is used repeatedly
in this section, is the constant defined in (\ref{eq:mueq2}).

\subsection{Auxiliary Lemmas}

Now we present a few auxiliary lemmas that we will use for the proof
of the noisy case. At the core of our subsequent proofs is the following
standard concentration bound.
\begin{lem}
\label{lem:AE} Suppose that the number of measurements satisfies
$m\gtrsim\sigma n\log n$. Then, with high probability, we have 
\[
\frac{1}{m}\left\Vert \sum_{i=1}^{m}A_{i}\epsilon_{i}\right\Vert _{2}\lesssim\sqrt{\frac{\sigma^{2}n\log n}{m}},
\]
where $\|\cdot\|_{2}$ denotes the operator norm of a matrix. 
\end{lem}

Lemma~\ref{lem:AE} will be used extensively in the proofs of Proposition
6, and Theorems 7 and 8.

Our strategy for establishing linear convergence is similar to that
of the noiseless case. Essentially, our goal is to show that with
an appropriate step-size, there is sufficient decrement in the objective
value in terms of $\|\nabla f_{c}(X)\|_{P^{*}}$. Then applying Theorem
\ref{thm:pl} will result in the desired convergence rate.

In the noiseless case, we proved a Lipschitz-like inequality (Lemma
\ref{lem:lipp}) and bounded the Lipschitz constant above in a neighborhood
around the ground truth. Similar results hold in the noisy case. However,
because of the noise, it will be easier to directly work with the
quartic polynomial $f_{c}(X-\alpha D)$ instead. In particular, we
have the following lemma that characterizes how much progress we make
by taking a step in the direction $D$.
\begin{lem}
\label{lem:descent} For any descent direction $D\in\mathbb{R}^{n\times r}$
and step-size $\alpha>0$ we have 
\begin{align}
f_{c}(X-\alpha D)&\le f_{c}(X)-\alpha\nabla f_{c}(X)^{T}D+\frac{\alpha^{2}}{2}D^{T}\nabla^{2}f_{c}(X)D\\
&+\frac{(1+\delta)\alpha^{3}}{m}\|D\|_{F}^{2}\left(2\|DX^{T}+XD^{T}\|_{F}+\alpha\|D\|_{F}^{2}\right).\label{eq:f_descent}
\end{align}
\end{lem}

\begin{proof}
Directly expanding the quadratic $f_{c}(X-\alpha D)$, we get 
\begin{align*}
f_{c}(X-\alpha D) & =\frac{1}{m}\|\AA((X-\alpha D)(X-\alpha D)^{T}-M^{\star})\|^{2}\\
 & =\frac{1}{m}\|\AA(XX^{T}-M^{\star})\|^{2}-\frac{2\alpha}{m}\langle\AA(XX^{T}-M^{\star}),\AA(XD^{T}+DX^{T})\rangle\\
 & \qquad+\frac{\alpha^{2}}{m}\left[2\langle\AA(XX^{T}-M^{\star}),\AA(DD^{T})\rangle+\|\AA(XD^{T}+DX^{T})\|^{2}\right]\\
 & \qquad-\frac{2\alpha^{3}}{m}\langle\AA(XD^{T}+DX^{T}),\AA(DD^{T})\rangle+\frac{\alpha^{4}}{m}\|\AA(DD^{T})\|^{2}.
\end{align*}
We bound the third- and fourth- order terms 
\begin{align*}
|\langle\AA(XD^{T}+DX^{T}),\AA(DD^{T})\rangle| & \overset{\text{(a)}}{\le}\|\AA(XD^{T}+DX^{T})\|\|\AA(DD^{T})\rangle\|\\
 & \overset{\text{(b)}}{\le}(1+\delta)\|XD^{T}+DX^{T}\|_{F}\|DD^{T}\|_{F}\\
 & \overset{\text{(c)}}{\le}(1+\delta)\|XD^{T}+DX^{T}\|_{F}\|D\|_{F}^{2}
\end{align*}
and 
\[
\|\AA(DD^{T})\|^{2}\overset{\text{(b)}}{\le}(1+\delta)\|DD^{T}\|_{F}^{2}\overset{\text{(c)}}{\le}(1+\delta)\|D\|_{F}^{4},
\]
Step (a) uses the Cauchy--Schwarz inequality; Step (b) applies $(\delta,2r)$-RIP;
Step (c) bounds $\|DD^{T}\|_{F}\le\|D\|_{F}^{2}$. Summing up these
inequalities we get the desired result. 
\end{proof}
It turns out that in our proofs it will be easier to work with the
\textit{vectorized} version of (\ref{eq:f_descent}), which we can
write as 
\begin{equation}
f_{c}(x-\alpha d)\le f_{c}(x)-\alpha\nabla f_{c}(x)^{T}d+\frac{\alpha^{2}}{2}d^{T}\nabla^{2}f_{c}(x)d+\frac{(1+\delta)\alpha^{3}}{m}\|d\|^{2}\left(2\|\J_{X}d\|+\alpha\|d\|^{2}\right),\label{eq:noisy_descent_vec}
\end{equation}
where we recall that $J_{X}:\R^{nr}\to\R^{n^{2}}$ is the linear operator
that satisfies $J_{X}d=\vect(XD^{T}+DX^{T})$.

Now we proceed to bound the higher-order terms in the Taylor-like
expansion above. 
\begin{lem}[Second-order term]
We have 
\[
\sigma_{\max}(\P^{-1/2}\nabla^{2}f_{c}(x)\P^{-1/2})\le\frac{2(1+\delta)}{m}\left(\frac{8\sigma_{r}^{2}(X)+\|XX^{T}-ZZ^{T}\|_{F}}{\sigma_{r}^{2}(X)+\eta}\right).
\]
\label{hess} 
\end{lem}

\begin{proof}
For any $v\in\R^{nr}$ where $v=\vect(V)$, we have 
\begin{align*}
m\cdot v^{T}\nabla^{2}f_{c}(x)v & =4\langle\AA(XX^{T}-ZZ^{T}),\AA(VV^{T})+2\|\AA(XV^{T}+VX^{T})\|^{2}\\
 & \le4\|\AA(XX^{T}-ZZ^{T})\|\|\AA(VV^{T})\|+2\|\AA(XV^{T}+VX^{T})\|^{2}\\
 & \le2(1+\delta)\left(\|XX^{T}-ZZ^{T}\|_{F}\|VV^{T}\|_{F}+2\|XV^{T}+VX^{T}\|_{F}^{2}\right)
\end{align*}
Now, let $v=\P^{-1/2}u$ for $u=\vect(U)$. Then, $V=UP^{-1/2}$ and
\begin{align*}
\|VV^{T}\|_{F} & =\|UP^{-1}U^{T}\|_{F}\le\sigma_{\max}(P^{-1})\|U\|_{F}^{2}=\frac{\|U\|_{F}^{2}}{\sigma_{r}^{2}(X)+\eta}.
\end{align*}
Also, $\|XV^{T}+VX^{T}\|_{F}\le2\|XV^{T}\|_{F}$ and 
\[
\|XV^{T}\|=\|XP^{-1/2}U^{T}\|\le\sigma_{\max}(XP^{-1/2})\|U\|_{F}=\left(\frac{\sigma_{r}^{2}(X)}{\sigma_{r}^{2}(X)+\eta}\right)^{1/2}\|U\|_{F}.
\]
Since $\|u\|=\|U\|_{F}$, it follows that 
\[
u^{T}\P^{-1/2}\nabla^{2}f_{c}(x)\P^{-1/2}u\leq\frac{2(1+\delta)}{m}\left(\frac{8\sigma_{r}^{2}(X)+\|XX^{T}-ZZ^{T}\|}{\sigma_{r}^{2}(X)+\eta}\right)\|u\|^{2},
\]
which gives the desired bound on the largest singular value. 
\end{proof}
The following lemma gives a bound on the third- and fourth-order terms
in (\ref{eq:noisy_descent_vec}).
\begin{lem}
Set $d=\P^{-1}\nabla f_{c}(x)$, then we have $\|\J d\|^{2}\leq8m^{2}\|\nabla f_{c}(x)\|_{P^{*}}^{2}$
and $\|d\|^{2}\leq\|\nabla f_{c}(x)\|_{P^{*}}^{2}/\eta$. \label{lem:high_order} 
\end{lem}

\begin{proof}
We have 
\begin{align*}
\|\J_{X}d\|^{2} & =\|\AA(XD^{T}+DX^{T})\|^{2}\le(1+\delta)\|XD^{T}+DX^{T}\|^{2}\\
 & =(1+\delta)\|\J_{X}d\|^{2}=m^{2}(1+\delta)\|\J\P^{-1}\nabla f_{c}(x)\|^{2}\\
 & \le m^{2}(1+\delta)\sigma_{\max}^{2}(\J\P^{-1/2})\|\P^{-1/2}\nabla f_{c}(x)\|^{2}\\
 & =4m^{2}(1+\delta)\frac{\sigma_{r}^{2}}{\sigma_{r}^{2}+\eta}\|\nabla f_{c}(x)\|_{P^{*}}^{2}\le8m^{2}\|\nabla f_{c}(x)\|_{P^{*}}^{2}
\end{align*}
and 
\begin{align*}
\|d\|^{2}=\|\P^{-1}\nabla f_{c}(x)\|^{2} & \le\sigma_{\max}(\P^{-1})\|\P^{-1/2}\nabla f_{c}(x)\|^{2}\\
 & =\frac{1}{\sigma_{r}^{2}+\eta}\|\nabla f(x)\|_{P^{*}}^{2}\le\|\nabla f(x)\|_{P^{*}}^{2}/\eta.
\end{align*}
\end{proof}

\section{Proof of Noisy Case with Optimal Damping Parameter }

Now we are ready to prove Theorem \ref{thm_oracle}, which we restate
below for convenience. 
\begin{thm}[Noisy measurements with optimal $\eta$]
Suppose that the noise vector $\epsilon\in\mathbb{R}^{m}$ has sub-Gaussian
entries with zero mean and variance $\sigma^{2}=\frac{1}{m}\sum_{i=1}^{m}\mathbb{E}[\epsilon_{i}^{2}]$.
Moreover, suppose that $\eta_{k}=\frac{1}{\sqrt{m}}\|\mathcal{A}(X_{k}X_{k}^{T}-M^{*})\|$,
for $k=0,1,\dots,K$, and that the initial point $X_{0}$ satisfies
$\|\mathcal{A}(X_{0}X_{0}^{T}-M^{*})\|^{2}<\rho^{2}(1-\delta)\lambda_{r^{*}}(M^{\star})^{2}$.
Consider $k^{*}=\arg\min_{k}\eta_{k}$, and suppose that 
$\alpha\leq1/L,$ where $L>0$ is a constant that only depends on
$\delta$. Then, with high probability, we have 
\begin{align}
\|X_{k^{*}}X_{k^{*}}^{T}-M^{\star}\|_{F}^{2}\lesssim\max\left\{ \frac{1+\delta}{1-\delta}\left(1-\alpha\frac{\mu_{P}}{2}\right)^{K}\|X_{0}X_{0}^{T}-M^{*}\|_{F}^{2},\mathcal{E}_{stat}\right\} ,
\end{align}
where $\mathcal{E}_{stat}:=\frac{\sigma^{2}nr\log n}{\mu_{P}(1-\delta)m}$. 
\end{thm}

\begin{proof}
\textbf{Step I. Using Lemma \ref{lem:descent} to establish sufficient
decrement.}

First, we write out the vectorized version of Lemma \ref{eq:noisy_descent_vec}:
\begin{equation}
f_{c}(x-\alpha d)\le f_{c}(x)-\alpha\nabla f_{c}(x)^{T}d+\frac{\alpha^{2}}{2}d^{T}\nabla^{2}f_{c}(x)d+\frac{(1+\delta)\alpha^{3}}{m}\|d\|^{2}\left(2\|\J_{X}d\|+\alpha\|d\|^{2}\right).\label{eq:noisy_descent2}
\end{equation}

To simplify notation, we define the error term $\E(x)=\frac{2}{m}\left(I_{r}\otimes\sum_{i=1}^{m}\epsilon_{i}A_{i}\right)x$,
so that the search direction (\ref{eq:d_gd}) can be rewritten as
$d=\P^{-1}(\nabla f_{c}(x)-\E(x))$.

Now plugging this $d$ into (\ref{eq:noisy_descent2}) yields 
\begin{align*}
f_{c}(x-\alpha d)\leq & f_{c}(x)-\alpha\|\nabla f_{c}(x)\|_{P*}^{2}+T_{1}+T_{2}+T_{3}
\end{align*}
where 
\begin{align*}
T_{1}= & \alpha\nabla f_{c}(x)^{T}\P^{-1}\E(x)\\
T_{2}= & \frac{\alpha^{2}}{2}\Big(\nabla f_{c}(x)^{T}\P^{-1}\nabla^{2}f_{c}(x)\P^{-1}\nabla f_{c}(x)+\E(x)^{T}\P^{-1}\nabla^{2}f_{c}(x)\P^{-1}\E(x)\\
 & -2\nabla f_{c}(x)^{T}\P^{-1}\nabla^{2}f_{c}(x)\P^{-1}\E(x)\Big)\\
T_{3}= & (1+\delta)\alpha^{3}\left(\|\P^{-1}\nabla f_{c}(x)-\P^{-1}\E(x)\|^{2}\right)\left(2\|\J\P^{-1}\nabla f_{c}(x)\|+2\|\J\P^{-1}\E(x)\|\right.\\
 & +\left.\alpha\|\P^{-1}\nabla f_{c}(x)-\P^{-1}\E(x)\|^{2}\right).
\end{align*}

\textbf{II. Bounding $T_{1},T_{2}$ and $T_{3}$}.

We control each term in the above expression individually. First,
we have 
\begin{align*}
T_{1} & =\alpha\nabla f_{c}(x)^{T}\P^{-1}\E(x)\leq\alpha\|\P^{-1}\nabla f_{c}(x)\|_{P}\|\E(x)\|_{P^{*}}=\alpha\|\nabla f_{c}(x)\|_{P^{*}}\|\E(x)\|_{P^{*}}.
\end{align*}

To bound $T_{2}$, first we note that for any vectors $x,y\in\mathbb{R}^{n}$
and any positive semidefinite matrix $P\in S_{+}^{n}$, we always
have $(x+y)^{T}P(x+y)\leq2(x^{T}Px+y^{T}Py)$. Therefore we can bound
\[
T_{2}\leq\alpha^{2}\left(\nabla f_{c}(x)^{T}\P^{-1}\nabla^{2}f_{c}(x)\P^{-1}\nabla f_{c}(x)+\E(x)^{T}\P^{-1}\nabla^{2}f_{c}(x)\P^{-1}\E(x)\right).
\]
Next, we apply Lemma \ref{hess} to arrive at 
\begin{align*}
\frac{1}{2}\sigma_{\max}(\P^{-1/2}\nabla^{2}f_{c}(x)\P^{-1/2}) & {\le}\frac{1+\delta}{m}\left(\frac{8\sigma_{r}^{2}(X)+\|XX^{T}-M^{\star}\|}{\sigma_{r}^{2}(X)+\eta}\right)\overset{def}{\le}L_{\delta},
\end{align*}
where $L_{\delta}$ is a constant that only depends on $\delta$ and
$m$. Note that the last inequality follows from the fact that $\eta=O(\|XX^{T}-M^{\star}\|)$.

Now based on the above inequality, we have 
\begin{align*}
 & \alpha^{2}\left(\nabla f_{c}(x)^{T}\P^{-1}\nabla^{2}f_{c}(x)\P^{-1}\nabla f_{c}(x)\right)\leq2\alpha^{2}L_{\delta}\|\nabla f_{c}(x)\|_{P^{*}}^{2}\\
 & \alpha^{2}\left(\E(x)^{T}\P^{-1}\nabla^{2}f_{c}(x)\P^{-1}\E(x)\right)\leq2\alpha^{2}L_{\delta}\|\E(x)\|_{P^{*}}^{2},
\end{align*}
which implies 
\[
T_{2}\leq2\alpha^{2}L_{\delta}\|\nabla f_{c}(x)\|_{P^{*}}^{2}+2\alpha^{2}L_{\delta}\|\E(x)\|_{P^{*}}^{2}
\]

Finally, to bound $T_{3}$, we first write 
\[
\|\P^{-1}\nabla f_{c}(x)-\P^{-1}\E(x)\|^{2}\leq2\|\P^{-1}\nabla f_{c}(x)\|^{2}+2\|\P^{-1}\E(x)\|^{2}.
\]
Moreover, invoking Lemma \ref{lem:high_order} leads to the following
inequalities 
\begin{align*}
 & \|\P^{-1}\nabla f_{c}(x)\|^{2}\leq\frac{\|\nabla f_{c}(x)\|_{P^{*}}^{2}}{\eta},\qquad &  & \|\P^{-1}\E(x)\|^{2}\leq\frac{\|\E(x)\|_{P^{*}}^{2}}{\eta}.\\
 & \|\J\P^{-1/2}\nabla f_{c}(x)\|\leq2\sqrt{2}\|\nabla f_{c}(x)\|_{P^{*}},\qquad &  & \|\J\P^{-1/2}\E(x)\|\leq2\sqrt{2}\|\E(x)\|_{P^{*}}.
\end{align*}
Combining the above inequalities with the definition of $T_{3}$ leads
to: 
\begin{align*}
T_{3}\leq & \frac{4(1+\delta)\alpha^{3}}{\eta}\left(\|\nabla f_{c}(x)\|_{P^{*}}^{2}+\|\E(x)\|_{P^{*}}^{2}\right)\\
 & \times\left(2\sqrt{2}\|\nabla f_{c}(x)\|_{P^{*}}+2\sqrt{2}\|\nabla\E(x)\|_{P^{*}}+\frac{\alpha}{\eta}\|\nabla f_{c}(x)\|_{P^{*}}^{2}+\frac{\alpha}{\eta}\|\E(x)\|_{P^{*}}^{2}\right).
\end{align*}

\textbf{III. Bounding the Error Term}

Next, we provide an upper bound on $\|\E(x)\|_{P^{*}}$. The following
chain of inequalities hold with high probability: 
\begin{align*}
\|\E(x)\|_{P^{*}}^{2} & =\E(x)^{T}\P^{-1}\E(x)=\left\Vert \left(\frac{2}{m}\sum_{i=1}^{m}\epsilon_{i}A_{i}\right)X(X^{T}X+\eta I)^{-1/2}\right\Vert _{F}^{2}\\
 & \leq\left\Vert \left(\frac{2}{m}\sum_{i=1}^{m}\epsilon_{i}A_{i}\right)\right\Vert _{2}^{2}\left\Vert X(X^{T}X+\eta I)^{-1/2}\right\Vert _{F}^{2}\\
 & \overset{(a)}{}\leq C\frac{\sigma^{2}n\log n}{m}\left(\sum_{i=1}^{r}\frac{\sigma_{i}^{2}(X)}{\sigma_{i}(X)^{2}+\eta}\right)\\
 & \leq C\frac{\sigma^{2}rn\log n}{m},
\end{align*}
where $C$ is an absolute constant and (a) follows from Lemma \ref{lem:AE}.

\textbf{IV. Bounding all the terms using $\|\nabla f_{c}(x)\|_{P*}$}

Combining the upper bound on $\|\E(X)\|_{P^{*}}$ with the previous
bounds for $T_{1},T_{2},T_{3}$ and denoting $\Delta=\|\nabla f_{c}(x)\|_{P*}$,
we have 
\begin{align*}
 & T_{1}\leq\alpha\Delta\sqrt{\frac{C\sigma^{2}rn\log n}{m},}\\
 & T_{2}\leq2\alpha^{2}L_{\delta}\Delta^{2}+2\alpha^{2}L_{\delta}\frac{\sigma^{2}rn\log n}{m}\\
 & T_{3}\leq\frac{4(1+\delta)\alpha^{3}}{\eta}\left(\Delta^{2}+\frac{C\sigma^{2}rn\log n}{m}\right)\left(\frac{\alpha\Delta^{2}}{\eta}+\frac{\alpha C\sigma^{2}rn\log n}{\eta m}+2\sqrt{2}\Delta+2\sqrt{2}\sqrt{\frac{C\sigma^{2}rn\log n}{m}}\right)
\end{align*}
Now, combining the upper bounds for $T_{1},T_{2}$ and $T_{3}$ with~\eqref{eq:noisy_descent2}
yields 
\begin{align}
f_{c}(x-\alpha d) & \leq f_{c}(x)-\alpha\Delta^{2}+\alpha\Delta\sqrt{\frac{C\sigma^{2}rn\log n}{m}}+2\alpha^{2}L_{\delta}\Delta^{2}+2C\alpha^{2}L_{\delta}\frac{\sigma^{2}rn\log n}{m}\nonumber \\
 & +\frac{4(1+\delta)\alpha^{3}}{\eta}\left(\Delta^{2}+\frac{C\sigma^{2}rn\log n}{m}\right)\left(\frac{\alpha\Delta^{2}}{\eta}+\frac{\alpha C\sigma^{2}rn\log n}{\eta m}+2\sqrt{2}\Delta+2\sqrt{2}\sqrt{\frac{C\sigma^{2}rn\log n}{m}}\right).\label{eq:eq_f2}
\end{align}
The above inequality holds with high probability for every iteration
of PrecGD.

\textbf{V. Two cases}

Now, we consider two cases. First, suppose that $\eta\leq2\sqrt{\frac{C\sigma^{2}nr\log n}{\mu_{P}m}}$.
This implies that $\min_{k}\eta_{k}\leq2\sqrt{\frac{C\sigma^{2}nr\log n}{\mu_{P}m}}$,
and hence, 
\begin{align*}
\|X_{k^{*}}X_{k^{*}}^{T}-M^{\star}\|_{F}^{2}\lesssim\frac{1}{1-\delta}\frac{1}{m}\|\mathcal{A}(X_{k^{*}}X_{k^{*}}^{T}-M^{\star})\|^{2}\lesssim\mathcal{E}_{stat}
\end{align*}
which completes the proof.

Otherwise, suppose that $\eta>2\sqrt{\frac{C\sigma^{2}nr\log n}{\mu_{P}m}}$.
Due to Theorem \ref{thm:pl}, we have $\Delta\geq2\sqrt{\frac{C\sigma^{2}rn\log n}{m}}$,
which leads to the following inequalities: 
\begin{align*}
-\alpha\Delta^{2}+\alpha\Delta\sqrt{\frac{C\sigma^{2}rn\log n}{m}}\leq-\frac{\alpha}{2}\Delta^{2},\qquad2\alpha^{2}L_{\delta}\Delta^{2}+2C\alpha^{2}L_{\delta}\frac{\sigma^{2}rn\log n}{m}\leq\frac{5}{2}\alpha^{2}L_{\delta}\Delta^{2}.
\end{align*}
Similarly, we have 
\begin{align*}
\Delta^{2}+\frac{C\sigma^{2}rn\log n}{m}\leq\frac{5}{4}\Delta^{2},\quad2\sqrt{2}\Delta+2\sqrt{2}\sqrt{\frac{C\sigma^{2}rn\log n}{m}}\leq3\sqrt{2}\Delta,
\end{align*}
and 
\[
\frac{\alpha\Delta^{2}}{\eta}+\frac{\alpha}{\eta}\frac{C\sigma^{2}rn\log n}{m}\leq\frac{5}{4}\frac{\alpha\Delta^{2}}{\eta}.
\]
Combined with~\eqref{eq:eq_f2}, we have 
\begin{align*}
f_{c}(x-\alpha d) & \leq f_{c}(x)-\frac{\alpha}{2}\Delta^{2}+\frac{5}{2}\alpha^{2}L_{\delta}\Delta^{2}+\frac{4(1+\delta)\alpha^{3}}{\eta}\left(\frac{5}{4}\Delta^{2}\right)\left(3\sqrt{2}\Delta+\frac{5}{4}\frac{\alpha\Delta^{2}}{\eta}\right)\\
 & \leq f_{c}(x)-\frac{\alpha}{2}\Delta^{2}\left(1-\frac{5}{2}L_{\delta}\alpha-60\sqrt{2}\frac{\alpha^{2}\Delta}{\eta}-25\alpha^{3}\left(\frac{\Delta}{\eta}\right)^{2}\right).
\end{align*}
Similar to the noiseless case, we can bound the ratio $\frac{\Delta}{\eta}$
as 
\begin{align*}
\frac{\Delta}{\eta}=\frac{\|\nabla f_{c}(x)\|_{P*}}{\eta} & \le\frac{(1+\delta)\sigma_{\max}(\J\P^{-1/2})\|\e\|}{\|\e\|}=(1+\delta)\frac{\sigma_{\max}^{2}(X)}{\sigma_{\max}^{2}(X)+\eta}\le1+\delta,
\end{align*}
which in turn leads to 
\begin{align*}
f_{c}(x-\alpha d)\leq f_{c}(x)-\frac{\alpha}{2}\Delta^{2}\left(1-\frac{5}{2}L_{\delta}\alpha-60\sqrt{2}\alpha^{2}(1+\delta)-25\alpha^{3}(1+\delta)^{2}\right).
\end{align*}
Now, assuming that the step-size satisfies $\alpha\leq\min\left\{ \frac{L_{\delta}}{60\sqrt{2}(1+\delta)+25(1+\delta)^{2}},\frac{1}{7L_{\delta}}\right\} .$
Since $L_{\delta}$ is a constant, we can simply write the condition
above as $\alpha\leq1/L$ where $L=\max\left\{ \frac{60\sqrt{2}(1+\delta)+25(1+\delta)^{2}}{L_{\delta}},{7L_{\delta}}\right\} .$
Now note that 
\begin{align*}
 & \frac{5}{2}L_{\delta}+60\sqrt{2}(1+\delta)\alpha+25(1+\delta)^{2}\alpha^{2}\leq\frac{7}{2}L_{\delta}\\
\implies & 1-\frac{5}{2}L_{\delta}\alpha-60\sqrt{2}(1+\delta)\alpha^{2}-25(1+\delta)^{2}\alpha^{3}\geq1-\frac{7}{2}L_{\delta}\alpha\geq\frac{1}{2}.
\end{align*}

This implies that 
\[
f_{c}(x-\alpha d)\leq f_{c}(x)-\frac{t\Delta^{2}}{4}\leq\left(1-\frac{\alpha\mu_{P}}{4}\right)f_{c}(x),
\]
where in the last inequality, we used $\Delta^{2}\geq\mu_{P}f_{c}(x)$,
which is just the PL-inequality in Theorem \ref{thm:pl}. Finally,
since $f_{c}(x)$ satisfies the RIP condition, combining the two cases
above we get 
\begin{align}
\|X_{k^{*}}X_{k^{*}}^{T}-M^{\star}\|_{F}^{2}\lesssim\max\left\{ \frac{1+\delta}{1-\delta}\left(1-\alpha\frac{\mu_{P}}{2}\right)^{k}\|X_{0}X_{0}^{T}-M^{*}\|_{F}^{2},\mathcal{E}_{stat}\right\} ,
\end{align}
as desired. 
\end{proof}

\section{Proof of Noisy Case with Variance Proxy (Theorem \ref{thm_noisy_var})}

In this section we prove Theorem \ref{thm_noisy_var}, which we restate
below for convenience. The only difference between this theorem and
Theorem \ref{thm_oracle} is that we de not assume that we have access
to the optimal choice of $\eta$. Instead, we only assume that we
have some proxy $\hat{\sigma}^{2}$ of the true variance of the noise.
For convenience we restate our result below. 
\begin{thm}[Noisy measurements with variance proxy]
Suppose that the noise vector $\epsilon\in\mathbb{R}^{m}$ has sub-Gaussian
entries with zero mean and variance $\sigma^{2}=\frac{1}{m}\sum_{i=1}^{m}\mathbb{E}[\epsilon_{i}^{2}]$.
Moreover, suppose that $\eta_{k}=\sqrt{|f(X_{k})-\hat{\sigma}^{2}|}$
for $k=0,1,\dots,K$, where $\hat{\sigma}^{2}$ is an approximation
of $\sigma^{2}$, and that the initial point $X_{0}$ satisfies $\|\mathcal{A}(X_{0}X_{0}^{T}-M^{*})\|_{F}^{2}<\rho^{2}(1-\delta)\lambda_{r^{*}}(M^{\star})^{2}$.
Consider $k^{*}=\arg\min_{k}\eta_{k}$, and suppose that 
$\alpha\leq1/L,$ where $L>0$ is a constant that only depends on
$\delta$. Then, with high probability, we have 
\begin{align}
\|X_{k^{*}}X_{k^{*}}^{T}-M^{*}\|_{F}^{2}\lesssim\max\Bigg\{ & \frac{1+\delta}{1-\delta}\left(1-\alpha\frac{\mu_{P}}{2}\right)^{K}\|X_{0}X_{0}^{T}-M^{*}\|_{F}^{2},\mathcal{E}_{stat},\mathcal{E}_{dev},\mathcal{E}_{var}\Bigg\},
\end{align}
where 
\begin{align}
\mathcal{E}_{stat}:=\frac{\sigma^{2}nr\log n}{\mu_{P}(1-\delta)m},\quad\mathcal{E}_{dev}:=\frac{\sigma^{2}}{1-\delta}\sqrt{\frac{\log n}{m}},\quad\mathcal{E}_{var}:=|\sigma^{2}-\hat{\sigma}^2|^{2}.
\end{align}
\end{thm}

The proof of Theorem 8 is similar to that of Theorem 7, with a key
difference that $\eta_{k}=\frac{1}{\sqrt{m}}\|\mathcal{A}(X_{k}X_{k}^{T}-M^{\star})\|$
is replaced with $\eta_{k}=\sqrt{|f(x_{k})-\hat{\sigma}^{2}|}$. Our
next lemma shows that this alternative choice of damping parameter
remains close to $\frac{1}{\sqrt{m}}\|\mathcal{A}(X_{k}X_{k}^{T}-M^{\star})\|$,
provided that the error exceeds a certain threshold.
\begin{lem}
\label{lem:proxy} Set $\eta=\sqrt{|f(x)-\hat{\sigma}^{2}|}$. Then,
with high probability, we have 
\[
\sqrt{\frac{1/4-\delta}{1+\delta}}\frac{1}{\sqrt{m}}\left\Vert \mathcal{A}(XX^{T}-M^{\star})\right\Vert \leq\eta\leq\sqrt{\frac{7/4+\delta}{1-\delta}}\frac{1}{\sqrt{m}}\left\Vert \mathcal{A}(XX^{T}-M^{\star})\right\Vert 
\]
provided that 
\[
\|XX^{T}-M^{\star}\|_{F}^{2}\gtrsim\max\left\{ {\frac{\sigma^{2}rn\log n}{m}},\sqrt{\frac{\sigma^{2}\log n}{m}},|\sigma^{2}-\hat{\sigma}^{2}|\right\} .
\]
\end{lem}

\begin{proof}
One can write 
\begin{align*}
f(x) & =\frac{1}{m}\|y-\AA(XX^{T})\|^{2}=\frac{1}{m}\|\AA(M^{\star}-XX^{T})+\epsilon\|^{2}\\
 & =\frac{1}{m}\|\AA(M^{\star}-XX^{T})\|^{2}+\frac{1}{m}\|\epsilon\|^{2}+\frac{2}{m}\inner{\AA(M^{\star}-XX^{T})}{\epsilon}.
\end{align*}
Due to the definition of the restricted Frobenius norm (\ref{eq:frob}),
we have 
\[
|\inner{\AA(M^{\star}-XX^{T})}{\epsilon}|\leq\|M^{\star}-XX^{T}\|_{F}\left\Vert \frac{1}{m}\sum_{i=1}^{m}A_{i}\epsilon_{i}\right\Vert _{F,2r}.
\]
Therefore, we have 
\begin{align}
 & \left|\frac{1}{m}\|\AA(M^{\star}-XX^{T})\|^{2}+\frac{1}{m}\|\epsilon\|^{2}-\hat{\sigma}^{2}-2\|M^{\star}-XX^{T}\|_{F}\left\Vert \frac{1}{m}\sum_{i=1}^{m}A_{i}\epsilon_{i}\right\Vert _{F,2r}\right|\leq\eta^{2}\label{eq_eta}\\
 & \left|\frac{1}{m}\|\AA(M^{\star}-XX^{T})\|^{2}+\frac{1}{m}\|\epsilon\|^{2}-\hat{\sigma}^{2}+2\|M^{\star}-XX^{T}\|_{F}\left\Vert \frac{1}{m}\sum_{i=1}^{m}A_{i}\epsilon_{i}\right\Vert _{F,2r}\right|\geq\eta^{2}.
\end{align}
Since the error $\epsilon_{i}$ is sub-Gaussian with parameter $\sigma$,
the random variable $\epsilon_{i}^{2}$ is sub-exponential with parameter
$16\sigma$. Therefore, 
\[
\mathbb{P}\left(\left|\frac{1}{m}\|\epsilon\|^{2}-\sigma^{2}\right|\geq t\right)\leq2\exp\left(-\frac{Cmt^{2}}{\sigma^{2}}\right).
\]
Now, upon setting $t=\sqrt{\frac{\sigma^{2}\log n}{m}}$, we have
\[
\left|\frac{1}{m}\|\epsilon\|^{2}-\sigma^{2}\right|\leq\sqrt{\frac{\sigma^{2}\log n}{m}},
\]
Moreover, we have 
\begin{align}
\left\Vert \frac{1}{m}\sum_{i=1}^{m}A_{i}\epsilon_{i}\right\Vert _{F,2r}\leq\sqrt{2r}\left\Vert \frac{1}{m}\sum_{i=1}^{m}A_{i}\epsilon_{i}\right\Vert _{2}\lesssim\sqrt{\frac{\sigma^{2}rn\log n}{m}}.
\end{align}
Combining the above two inequalities with~\eqref{eq_eta} leads to
\begin{align}
\eta^{2} & \geq\frac{1}{m}\|\AA(M^{\star}-XX^{T})\|^{2}-C\|M^{\star}-XX^{T}\|_{F}\sqrt{\frac{\sigma^{2}rn\log n}{m}}-\sqrt{\frac{\sigma^{2}\log n}{m}}-|\sigma^{2}-\hat{\sigma}^{2}|\nonumber \\
 & \geq(1-\delta)\|XX^{T}-M^{\star}\|_{F}^{2}-C\|XX^{T}-M^{\star}\|_{F}\sqrt{\frac{\sigma^{2}rn\log n}{m}}-\sqrt{\frac{\sigma^{2}\log n}{m}}-|\sigma^{2}-\hat{\sigma}^{2}|.\label{eq_eta2}
\end{align}
Now assuming that 
\[
\|XX^{T}-M^{\star}\|_{F}^{2}\geq\max\left\{ 16C^{2}{\frac{\sigma^{2}rn\log n}{m}},4\sqrt{\frac{\sigma^{2}\log n}{m}},4|\sigma^{2}-\hat{\sigma}^{2}|\right\} ,
\]
the inequality~\eqref{eq_eta2} can be further lower bounded as 
\[
\eta^{2}\geq(1/4-\delta)\|XX^{T}-M^{\star}\|_{F}^{2}\geq\frac{1/4-\delta}{1+\delta}\frac{1}{m}\|\mathcal{A}(XX^{T}-M^{\star})\|,
\]
which completes the proof for the lower bound. The upper bound on
$\eta^{2}$ can be established in a similar fashion. 
\end{proof}
Now we are ready to prove Theorem \ref{thm_noisy_var}. 
\begin{proof}
We consider two cases. First, suppose that 
\[
\min_{k}\eta_{k}\lesssim\max\left\{ {\frac{\sigma^{2}rn\log n}{m}},\sqrt{\frac{\sigma^{2}\log n}{m}},|\sigma^{2}-\hat{\sigma}^{2}|\right\} .
\]
Combined with~\eqref{eq_eta2}, this implies that 
\begin{align}
 & (1-\delta)\|X_{k^{*}}X_{k^{*}}^{T}-M^{\star}\|_{F}^{2}-C\|X_{k^{*}}X_{k^{*}}^{T}-M^{\star}\|_{F}\sqrt{\frac{\sigma^{2}rn\log n}{m}}\nonumber \\
 & \lesssim\max\left\{ {\frac{\sigma^{2}rn\log n}{m}},\sqrt{\frac{\sigma^{2}\log n}{m}},|\sigma^{2}-\hat{\sigma}^{2}|\right\} .\label{eq:eq_up}
\end{align}
Now, if $\|X_{k^{*}}X_{k^{*}}^{T}-M^{\star}\|_{F}\leq2C\sqrt{\frac{\sigma^{2}rn\log n}{m}}$,
then the proof is complete. Therefore, suppose that $\|X_{k^{*}}X_{k^{*}}^{T}-M^{\star}\|_{F}>2C\sqrt{\frac{\sigma^{2}rn\log n}{m}}$.
This together with~\eqref{eq:eq_up} leads to 
\[
\|X_{k^{*}}X_{k^{*}}^{T}-M^{\star}\|_{F}^{2}\lesssim\frac{1}{1/2-\delta}\max\left\{ {\frac{\sigma^{2}rn\log n}{m}},\sqrt{\frac{\sigma^{2}\log n}{m}},|\sigma^{2}-\hat{\sigma}^{2}|\right\} ,
\]
which again completes the proof. Finally, suppose that 
\[
\min_{k}\eta_{k}\gtrsim\max\left\{ {\frac{\sigma^{2}rn\log n}{m}},\sqrt{\frac{\sigma^{2}\log n}{m}},|\sigma^{2}-\hat{\sigma}^{2}|\right\} .
\]
This combined with~\eqref{eq_eta} implies that 
\begin{align*}
 & (1+\delta)\|X_{k^{*}}X_{k^{*}}^{T}-M^{\star}\|_{F}^{2}+C\|X_{k^{*}}X_{k^{*}}^{T}-M^{\star}\|_{F}\sqrt{\frac{\sigma^{2}rn\log n}{m}}\\
 & \gtrsim\max\left\{ {\frac{\sigma^{2}rn\log n}{m}},\sqrt{\frac{\sigma^{2}\log n}{m}},|\sigma^{2}-\hat{\sigma}^{2}|\right\} ,
\end{align*}
for every $k=0,1,\dots,K$. If $\|X_{k^{*}}X_{k^{*}}^{T}-M^{\star}\|_{F}\leq2C\sqrt{\frac{\sigma^{2}rn\log n}{m}}$,
then the proof is complete. Therefore, suppose that $\|X_{k^{*}}X_{k^{*}}^{T}-M^{\star}\|_{F}>2C\sqrt{\frac{\sigma^{2}rn\log n}{m}}$.
This together with the above inequality results in 
\begin{align*}
\|X_{k}X_{k}^{T}-M^{\star}\|_{F}^{2} & \gtrsim\frac{1}{3/2+\delta}\max\left\{ {\frac{\sigma^{2}rn\log n}{m}},\sqrt{\frac{\sigma^{2}\log n}{m}},|\sigma^{2}-\hat{\sigma}^{2}|\right\} \\
 & \gtrsim\max\left\{ {\frac{\sigma^{2}rn\log n}{m}},\sqrt{\frac{\sigma^{2}\log n}{m}},|\sigma^{2}-\hat{\sigma}^{2}|\right\} 
\end{align*}
for every $k=0,1,\dots,K$. Therefore, Lemma~\ref{lem:proxy} can
be invoked to show that 
\[
\eta_{k}\asymp\frac{1}{\sqrt{m}}\|\mathcal{A}(X_{k}X_{k}^{T}-M^{\star})\|.
\]
With this choice of $\eta_{k}$, the rest of the proof is identical
to that of Theorem 7, and omitted for brevity. 
\end{proof}

\section{Proof for Spectral Initialization (Proposition \ref{prop:spectinit})}

In this section we prove that spectral initialization is able to generate
a sufficiently good initial point so that PrecGD achieves a linear
convergence rate, even in the noisy case. For convenience we restate
our result below. 
\begin{prop}[Spectral Initialization]
Suppose that $\delta\leq(8\kappa\sqrt{r^{*}})^{-1}$ and $m\gtrsim\frac{1+\delta}{1-\delta}\frac{\sigma^{2}rn\log n}{\rho^{2}\lambda_{r^{\star}}^{2}(M^{\star})}$
where $\kappa=\lambda_{1}(M^{\star})/\lambda_{r^{\star}}(M^{\star})$.
Then, with high probability, the initial point $X_{0}$ produced by~\eqref{eq:spec}
satisfies the radius condition (\ref{eq:radius}). 
\end{prop}

\begin{proof}
Let $\mathcal{A}^{*}:\mathbb{R}^{m}\to\mathbb{R}^{n\times n}$ be
the dual of the linear operator $\mathcal{A}(\cdot)$, defined as
$\mathcal{A}^{*}(y)=\sum_{i=1}^{m}y_{i}A_{i}.$ Based on this definition,
the initial point $X_{0}\in\mathbb{R}^{n\times r}$ satisfies $X_{0}=\mathcal{P}_{r}\left(\frac{1}{m}\AA^{*}(y)\right),$
where we recall that 
\[
\mathcal{P}_{r}(M)=\arg\min_{X\in\R^{n\times r}}\|XX^{T}-M\|_{F}.
\]
Define $E=X_{0}X_{0}^{T}-M^{\star}$, and note that $\mathrm{rank}(E)\leq2r$.
It follows that 
\begin{align*}
\|E\|_{F} & =\sqrt{\sum_{i=1}^{r}\sigma_{i}(E)^{2}+\sum_{i=r+1}^{2r}\sigma_{i}(E)^{2}}\leq\sqrt{2}\|E\|_{F,2r}\\
 & \leq\sqrt{2}\left\Vert X_{0}X_{0}^{T}-\frac{1}{m}\AA^{*}(y)\right\Vert _{F,2r}+\sqrt{2}\left\Vert \frac{1}{m}\AA^{*}(y)-M^{\star}\right\Vert _{F,2r}\\
 & \leq2\sqrt{2}\left\Vert \frac{1}{m}\AA^{*}(y)-M^{\star}\right\Vert _{F,2r}\\
 & \leq2\sqrt{2}\left\Vert \frac{1}{m}\AA^{*}(\AA(M^{\star}))-M^{\star}\right\Vert _{F,2r}+2\sqrt{2}\left\Vert \frac{1}{m}A_{i}\epsilon_{i}\right\Vert _{F,2r}\\
 & \leq2\sqrt{2}\delta\|M^{\star}\|_{F}+2\sqrt{2}\left\Vert \frac{1}{m}A_{i}\epsilon_{i}\right\Vert _{F,2r}.
\end{align*}
Now, note that $\|M^{\star}\|_{F}\leq\sqrt{r^{*}}\kappa\lambda_{r^{*}}(M^{\star})$.
Moreover, due to Lemma \ref{lem:AE}, we have 
\begin{align}
2\sqrt{2}\left\Vert \frac{1}{m}A_{i}\epsilon_{i}\right\Vert _{F,2r}\leq2\sqrt{2}\sqrt{2r}\left\Vert \frac{1}{m}A_{i}\epsilon_{i}\right\Vert _{2}\lesssim\sqrt{\frac{\sigma^{2}rn\log n}{m}}.
\end{align}
This implies that 
\begin{align*}
\frac{1}{m}\|\mathcal{A}(X_{0}X_{0}^{T}-M^{\star})\|^{2}\leq16(1+\delta)r^{*}\kappa^{2}\lambda_{r^{*}}(M^{\star})^{2}\delta^{2}+C\frac{\sigma^{2}rn\log n}{m}
\end{align*}
Therefore, upon choosing $\delta\leq\frac{\rho}{8\sqrt{r^{*}}{\kappa}}$
and $m\gtrsim\frac{1+\delta}{1-\delta}\frac{\sigma^{2}rn\log n}{\rho^{2}\lambda_{r^{*}}^{2}(M^{\star})}$,
we have 
\begin{align}
\frac{1}{m}\|\mathcal{A}(XX^{T}-M^{*})\|^{2}\leq\rho^{2}(1-\delta)\lambda_{r^{*}}(M^{\star})^{2}
\end{align}
This completes the proof. 
\end{proof}

\section{\label{app_aux} Proof of Lemma~\ref{lem:AE}}

First we state a standard concentration inequality. A proof of this
result can be found in \citet{tropp2015introduction}.
\begin{lem}[Matrix Bernstein's inequality]
Suppose that $\{W_{i}\}_{i=1}^{m}$ are matrix-valued random variables
such that $\E[W_{i}]=0$ and $\|W_{i}\|_{2}\leq R^{2}$ for all $i=1,\dots,m$.
Then 
\[
\mathbb{P}\left(\left\Vert \sum_{i=1}^{m}W_{i}\right\Vert \geq t\right)\leq n\exp\left(\frac{-t^{2}}{2\left\Vert \sum_{i=1}^{m}\E\left[W_{i}^{2}\right]\right\Vert _{2}+\frac{2R^{2}}{3}t}\right).
\]
\end{lem}

We also state a standard concentration bound for the operator norm
of Gaussian ensembles. A simple proof can be found in \citet{wainwright2019high}. 
\begin{lem}
\label{Abound} Let $A\in\R^{n\times n}$ be a standard Gaussian ensemble
with i.i.d. entries. Then the largest singular value of $A$ (or equivalently,
the operator norm) satisfies 
\[
\sigma_{\max}(A)\leq(2+c)\sqrt{n}
\]
with probability at least $1-2\exp(-nc^{2}/2)$. 
\end{lem}

For simplicity, we assume that the measurement matrices $A_{i},i=1,\dots m$
are fixed and all satisfy $\|A_{i}\|\leq C\sqrt{n}$. Due to Lemma
\ref{Abound}, this assumption holds with high probability for Gaussian
measurement ensembles. Next, we provide the proof of Lemma~\ref{lem:AE}.
\begin{proof}[\textit{\emph{Proof of Lemma~\ref{lem:AE}.}}]
First, note that $\|A_{i}\varepsilon_{i}\|_{2}\leq\|A_{i}\|\cdot|\varepsilon_{i}|$.
The assumption $\|A_{i}\|\lesssim\sqrt{n}$ implies that $\|A_{i}\varepsilon_{i}\|$
is sub-Gaussian with parameter $C\sqrt{n}\sigma$. Therefore, we have
$\mathbb{P}(\|A_{i}\varepsilon\|\gtrsim\sqrt{n}t)\geq1-2\exp\left(-\frac{t^{2}}{2\sigma^{2}}\right).$
Applying the union bound yields 
\[
\mathbb{P}(\max_{i=1,\dots,m}\|A_{i}\varepsilon\|\geq\sqrt{n}t)\geq1-2m\exp\left(-\frac{t^{2}}{2\sigma^{2}}\right).
\]
Moreover, one can write 
\begin{align}
\left\Vert \sum_{i=1}^{m}\E[(A_{i}\varepsilon_{i})^{2}]\right\Vert \leq\sum_{i=1}^{m}\|A_{i}\|^{2}\E[\varepsilon_{i}^{2}]\lesssim\sigma^{2}mn
\end{align}
Using Matrix Bernstein's inequality, we get 
\[
\mathbb{P}\left(\frac{1}{m}\left\Vert \sum_{i=1}^{m}A_{i}\varepsilon\right\Vert \leq t\right)\geq1-n\exp\left(-\frac{t^{2}m^{2}}{2C\sigma^{2}mn+\frac{2}{3}C'\sqrt{n}mt}\right)-2m\exp\left(-\frac{t^{2}}{2}\right).
\]
Using $t\asymp\sqrt{\frac{\sigma^{2}n\log n}{m}}$ in the above inequality
leads to 
\begin{align*}
\mathbb{P}\left(\frac{1}{m}\left\Vert \sum_{i=1}^{m}A_{i}\varepsilon\right\Vert \lesssim\sqrt{\frac{\sigma^{2}n\log n}{m}}\right) & \geq1-n^{-C}-2m\exp\left(-\frac{t^{2}}{2}\right)\\
 & \gtrsim1-3n^{-C},
\end{align*}
where the last inequality follows from the assumption $m\gtrsim\sigma n\log n$.
This completes the proof. 
\end{proof}

\end{document}